\def\flnm{gPs-} % base name for file
\def\version{BM29} % version number
\def\datum{September 29, 2021}

\let\iflabels=\iffalse % show labels
\let\iftxfonts=\iftrue % more heavy font
\let\iffilename=\iffalse % file name at bottom of page
\let\ifhyperref=\iftrue % internal references clickable
\iflabels
\documentclass[11pt,reqno]{amsart}
\else
\documentclass[11pt]{amsart}
\fi

\usepackage{amssymb,amsmath,amsthm}
\usepackage{verbatim}
\usepackage[margin=1cm]{caption}
\usepackage{calc,graphicx}

\iflabels
\usepackage[notcite,notref]{showkeys}
\fi

\iftxfonts
\usepackage[varg]{txfonts}
\fi

\usepackage{color}

\definecolor{blue}{rgb}{0,0,1}
\definecolor{red}{rgb}{1,0,0}
\definecolor{purple}{rgb}{1,0,1}
\definecolor{green}{rgb}{0,1,.5}
\definecolor{orange}{rgb}{1,.4,0}

\long\def\red#1\endred{{\color{red}#1}}
\long\def\blue#1\endblue{{\color{blue}#1}}
\long\def\purple#1\endpurple{{\color{purple}#1}}
\long\def\green#1\endgreen{{\color{green}#1}}
\long\def\orange#1\endorange{{\color{orange}#1}}

\ifhyperref
  \usepackage[colorlinks,breaklinks,hypertexnames=false]{hyperref}
\else
  \newcommand\texorpdfstring[2]{#1}
\fi

\iffilename
\makeatletter\def\@oddfoot{\rm{\footnotesize{File name
\tt\flnm\version.tex}\quad\today\hfil\thepage}}
\let\@evenfoot\@oddfoot 
\makeatother
\fi

\numberwithin{equation}{section}
\def\be#1\ee{\begin{equation}#1\end{equation}}
\def\bad#1\ead{\be\begin{aligned}#1\end{aligned}\ee}
\def\badl#1#2\eadl{\be\label{#1}\begin{aligned}#2\end{aligned}\ee}

\newtheorem{thm}{Theorem}[section]
\newtheorem{prop}[thm]{Proposition}

\newtheorem{lem}[thm]{Lemma}

\theoremstyle{definition}
\newtheorem{defn}[thm]{Definition}

\newcommand\rmrk[1]{\medskip\par\noindent\emph{#1. }}

\renewcommand\={\,=\,}

\renewcommand\setminus{\smallsetminus}

\newcommand\re{{\mathrm{Re}\,}}
\newcommand\im{{\mathrm{Im}\,}}
\newcommand\sign{{\mathrm{Sign}\,}}%
\newcommand\SU{{\mathrm{SU}}}
\newcommand\SL{{\mathrm{SL}}}

\newcommand\U{{\mathrm{U}}}
\newcommand\oh{\mathrm{O}}

\newcommand\dis{\stackrel.=}
\newcommand\ddis{\;\dis\;}

\newcommand\Gf{\Gamma} % gamma-function

\newcommand\pmat[1]{\begin{pmatrix}#1\end{pmatrix}}

\newcommand\al{\alpha}

\newcommand\bt{\beta}
\newcommand\Gm{\Gamma}
\newcommand\gm{\gamma}
\newcommand\Dt{\Delta}
\newcommand\dt{\delta}
\newcommand\e{\varepsilon}
\newcommand\z{\zeta}

\newcommand\Th{\Theta}
\renewcommand\th{\vartheta}
\newcommand\kp{\kappa}
\let\k=\kp
\newcommand\Ld{\Lambda}
\newcommand\ld{\lambda}
\newcommand\Mu{{\mathsf M}}

\newcommand\s{\sigma}
\newcommand\Ups{\Upsilon}
\newcommand\ups{\upsilon}
\newcommand\Ph{\Phi}
\newcommand\ph{\varphi}
\newcommand\ch{\chi}

\newcommand\ps{\psi}
\newcommand\Om{\Omega}
\newcommand\om{\omega}

\newcommand\CC{{\mathbb{C}}}
\newcommand\NN{{\mathbb{N}}}

\newcommand\RR{{\mathbb{R}}}
\newcommand\ZZ{{\mathbb{Z}}}

\newcommand\CK{{\mathbf C}}
\newcommand\HH{{\mathbf H}}

\newcommand\WW{{\mathbf W}}
\newcommand\XX{{\mathbf X}}

\newcommand\Z{{\mathbf Z}}

\newcommand\alie{{\mathfrak a}}
\newcommand\glie{{\mathfrak g}}

\newcommand\klie{{\mathfrak k}}
\newcommand\mlie{{\mathfrak m}}
\newcommand\nlie{{\mathfrak n}}

\newcommand\Poin{{\mathbf M}}
\let\M=\Poin
\newcommand\poin{{\mathbf P}}
\newcommand\Four{{\mathbf F}}

\newcommand\Ffu{{\mathcal F}}

\newcommand\Nfu{{\mathcal N}}
\newcommand\Mfu{{\mathcal M}}
\newcommand\Wfu{{\mathcal W}}

\newcommand\n{{\mathbf n}}

\newcommand\am{{\mathrm a}}

\newcommand\km{{\mathrm k}}
\newcommand\mm{{\mathrm m}}
\newcommand\nm{{\mathrm n}}
\newcommand\wm{{\mathrm w}}

\newcommand\Wo{{\mathrm{O}_{\mathrm{W}}}}

\newcommand\MS{{\mathrm{MS}}}
\let\W=\WW
\newcommand\wo{{\mathbf{wo}}}
\let\ord=\wo
\newcommand\po{{\mathbf{po}}}

\newcommand\Ws[1]{S_{\!#1}}

\newcommand\Kph[4]{\,{}^{#1}\Ph^{#2}_{#3,#4}}
\newcommand\kph[4]{\,{}^{#1}\ph^{#2}_{#3,#4}}

\newcommand\II{\mathit{II}}
\newcommand\IF{\mathit{IF}}
\newcommand\FI{\mathit{FI}}

\newcommand\A{{\mathbf A}}
\newcommand\Au[1]{\mathbf{A}^{\!#1}}

\newcommand\discr{{\mathrm{discr}}}
\newcommand\cont{{\mathrm{cont}}}

\newcommand\sym{{\mathcal X}}

\newcommand\Schw{{\mathcal S}}

\newcommand\B{{\mathcal B}}

\renewcommand\Wr{\mathrm{Wr}}
\newcommand\Sie{{\mathfrak S}}

\begin{document}
\title{Generalized Poincar\'e series for $\SU(2,1)$}

\author{Roelof W. Bruggeman}

\address{Mathematisch Instituut, Universiteit Utrecht, Postbus 80010,
	3508 TA Utrecht, Nederlands}

\email{r.w.bruggeman@uu.nl}

\author{Roberto J. Miatello}

\address{FAMAF-CIEM, Universidad Nacional de C\'or\-do\-ba,
	C\'or\-do\-ba~5000, Argentina}

\email{miatello@famaf.unc.edu.ar}

\date{\datum}

\begin{abstract}
We define and study 'non-abelian' Poincar\'e series for the group
$G=\mathrm{SU} (2,1)$, that is, Poincar\'e series attached to a
Stone-Von Neumann representation of the unipotent subgroup $N$ of $G$.
Such Poincar\'e series have in general exponential growth. In this
study we use results on abelian and non-abelian Fourier term modules
obtained in \cite{BMSU21}. We compute the inner product of truncations
of these series and of those associated to unitary characters of $N$,
with square integrable automorphic forms in connection with their
Fourier expansions. As a consequence, we obtain general completeness
results that, in particular, generalize those valid for the classical
holomorphic (and antiholomorphic) Poincar\'e series for $\SL(2,\RR)$.
\end{abstract}

\maketitle

\tableofcontents

\section{Introduction}

In this paper we study Poincar\'e series on the Lie group $\SU(2,1)$.
Classically, in the study of automorphic forms one considers functions
on a space $X$ that are invariant under the action of a group $\Gm$ of
transformations of~$X$. The idea of a Poincar\'e series is to take as a
germ a simple function $h$ on $X$ and to form the sum of transformed
functions $x\mapsto \sum_{\gm\in \Gm} h(\gm x)$. If this sum converges
absolutely one has a $\Gm$-invariant function on~$X$. Poincar\'e
\cite{Po1882} used this idea to construct what he called
\emph{fonctions th\^etafuchsiennes}.

For the well-known cuspidal Poincar\'e series in the theory of
holomorphic modular forms, the germ $h(z)=e^{2\pi i n z}$,
$n\in \ZZ_{\geq 1}$ on the upper half-plane is invariant under the
transformation $z\mapsto z+1$. With a suitable automorphy factor, it
leads to series that converge absolutely, where the sum is over
$\Gm_\infty\backslash \SL_2(\ZZ)$, and $\Gm_\infty$ is generated by
$\pm\begin{pmatrix}1&1\\0&1\end{pmatrix}$. The resulting series may be
identically zero. Hecke showed that the non-zero series of this type
span the spaces of holomorphic cusp forms.

The same idea can be applied to the function
\[ h_s(x+iy) = e^{2\pi i n x} \, y^{1/2} \,I_{s-1/2}(2\pi|n|y)\]
on the %complex 
upper half-plane, with $I_{s-1/2}$ an exponentially increasing modified
Bessel function. The function $h_s$ is an eigenfunction of the
hyperbolic Laplace operator. For $\re s>1$, the sum
$\sum_{\gm\in \Gm_\infty \backslash \SL_2(\ZZ)} h_s(\gm z)$ converges
absolutely and defines a real-analytic Maass form with exponential
growth. The resulting family of functions, called real-analytic
Poincar\'e series, have a meromorphic continuation to the complex plane
and the spaces of Maass cusp forms are spanned by residues of these
families. (See Neunh\"offer \cite{Neu73}, Niebur \cite{Nie73}.)
  This construction can be translated in a standard way yielding
  functions on the group $G=\SL(2,\RR)$ that satisfy a transformation
  rule under the discrete subgroup $\Gm$.

Miatello and Wallach \cite{MW89} extended the methods to automorphic
forms on arbitrary Lie groups $G$ of real rank one. For a discrete
cofinite subgroup $\Gm$ of $G$, they used a holomorphic family of germs
$h$ of Casimir eigenfunctions on the Lie group $G$ such that
$h(n g) \= \ch(n) \, h(g)$, where $\ch$ is a character of a maximal
unipotent subgroup $N$ of $G$ that is trivial on $N\cap \Gm$. Under
suitable conditions on the parameters, this leads to Poincar\'e series
on $G$ with exponential growth. It is shown that there is a meromorphic
continuation and a functional equation and the residues at
singularities are square integrable automorphic forms. In many cases
the corresponding spaces of cusp forms are spanned by residues of these
Poincar\'e series (see [Theorem 3.2] and Corollary 3.4 in \cite{MW89}).
On the other hand, as it has been observed by several authors (for
instance Gelbart-Piatetski-Shapiro (\cite{GPS84}))
there exist many non-generic cusp forms, i.e. those having all their
$\ch$-Fourier coefficients equal to zero, that cannot be detected by
the series in \cite{MW89}. These cusp forms correspond to
representations of $G$ that do not have Whittaker models (see
\cite{KO95},\cite{VZ84} for algebraic criteria satisfied by such
representations).

The goal of the present paper is to refine the results in \cite{MW89} in
the case of the group $\SU(2,1)$, by incorporating Poincar\'e series
attached to irreducible, infinite dimensional representations of $N$.
For this purpose, we were led to study the so called Fourier term
modules attached to the Stone-Von Neumann representation of $N$, a task
we carry out mainly in \cite{BMSU21}.

Here we define, initially for $\re \nu >2$ in a way similar to
\cite{MW89}, holomorphic families $\M_\Nfu (\xi,\nu)$ of exponentially
increasing automorphic forms attached to an arbitrary irreducible,
unitary representation of $N$. We carry out its meromorphic
continuation to $\CC$, proving that the poles are simple (except
possibly at $\nu=0$) and they occur at spectral parameters. We show
that resolving these poles we obtain square integrable automorphic
forms. In Section~\ref{sect-stcr} we state our main results. We prove
that what we call the \emph{main parts} of these Poincar\'e series span
the spaces of cusp forms (see Theorem~\ref{thm-complgen}). The main
part is an automorphic form that in most cases is either a residue at a
simple pole, or just a value at a point $s$ of holomorphy of the
series.

  We consider the different cases of $(\glie ,K)$-modules in
  Propositions 6.2 through 6.4. In particular we study the cases of
holomorphic and antiholomorphic discrete series representations and the
so called thin representations that come up in the computation of the
cohomology of the associated locally symmetric spaces. The resulting
forms occur at integral parameters, mostly in the region of absolute
convergence of the Poincar\'e series, and we show that their values or
residues span the spaces of holomorphic cusp forms (see
Proposition~\ref{prop-hads}). A comparison with results in the
classical case is made in \S~\ref{comparison} and in
Appendix~\ref{PSSl2R}.

Some of our results are connected with results by Ishikawa \cite{Ish99},
\cite{Ish00} (see \cite[\S16.2]{BMSU21}). Similar methods as in this
paper could be applied in the case of the more general group
$\SU(n,1)$, but we have preferred to stick to the case $n=2$ to be able
to give explicit results. We use computer aid to carry out some of
calculations, specially in Proposition 7.1. We first work with discrete
subgroups satisfying the conditions in \cite{BMSU21}, \S5.1, and we
further restrict the discussion to a class of discrete subgroups of
$\SU(2,1)$ having only one cusp to avoid notational complications that
might obscure the methods. See~\S\ref{discretesubgroups}.

\section{Preliminaries. The group \texorpdfstring{$\SU(2,1)$}{SU21}}

  We use the following realization of $\SU(2,1)$
\badl{SU} G&\;:=\; \SU(2,1) \= \bigl\{ g\in \SL_3(\CC)\;:\; \bar g^t
I_{2,1} g \= I_{2,1}\bigr\}\,,\eadl
with $I_{2,1} = \textrm{diag}\{1,1,-1\}\}$.

We fix subgroups $K$, $A$, and $N$ in an Iwasawa decomposition $G=NAK$.
Here, the subgroup $K$ consists of the matrices
\be\label{Kdef} K\= \biggl\{ \pmat {u&0\\0&\dt }\;:\; u\in \U(2) \,, \;
\dt\,\det u\=1\biggr\}\,, \ee
and is a maximal compact subgroup of $G$. There is an isomorphism
$\bigl(U(1) \times \SU(2)\bigr)/\{1,-1\}   \;\cong\; K $ via
$(\eta,v) \mapsto \pmat{\eta v&0\\0&\eta^{-2}}$. The group
\badl{Adef} A&\= \bigl\{ \am(y)\;:\; y>0\bigr\}\,, \textrm{ with }%\\
\am(y)&=\pmat{(y+y^{-1})/2&0 &(y-y^{-1})/2\\
0&1&0\\
(y-y^{-1})/2&0&(y+y^{-1})/2} \eadl
is the identity component of a split torus.

The unipotent subgroup
$N = \bigl\{ \nm(b,r)\;:\; b\in \CC,\, r\in \RR\bigr\}\,,$ with
$$\nm(b;r)\=\pmat{ 1+i r-|b|^2/2& b & -ir+|b|^2/2\\
-\bar b & 1 & \bar b\\
ir-|b|^2/2&b& 1-ir +|b|^2/2} \,,$$
is isomorphic to the Heisenberg group on $\RR^2$. The multiplication is
given by
\be \nm(b_1,r_1) \,
\nm(b_2,r_2)\=\nm\bigl(b_1+b_2,r_1+r_2+\im(\overline{b_1}b_2)\bigr)\,,
\ee
and its center is $ Z(N) \= \bigl\{ \nm(0,r)\;:\; r\in \RR\bigr\}\,.$

The group $K$ contains the subgroup
$M\= \{ \mm(\eta) \;:\; |\eta|=1\bigr\}\,,$ with
\be\label{M-def}\mm(\eta) \=\textrm{diag} \{ \eta, \eta^{-2},
\eta\}\,.\ee
The group $M$ commutes with~$A$, and $AM$ normalizes $N$. Namely,
\be \label{normN} \am(y) \nm(b,r) \am(y)^{-1}\= \nm\bigl( y
b,y^2r)\,,\qquad \mm(\eta)
\nm(b,r)
\mm(\eta)^{-1} \= \nm(\eta^3 b,r)\,.\ee

Furthermore, $G$ has a \emph{Bruhat decomposition}
\be\label{Brud} G = NAM \sqcup N\wm AMN\,,\ee
with $ \wm\= \textrm{diag}\{-1,-1,1\}$ and the set
$\mathcal C=P\wm P=N\wm AMN$ is called the big cell.

In Appendix \ref{app-Lie} we discuss the Lie algebras of~$G$, $N$, $A$
and~$K$.

\subsection {Upper half-plane model of the symmetric space} The space
\be \sym \= \bigl\{(z,u) \in \CC^2\;:\; \im z >
|u|^2\bigr\}\ee
is a model of the symmetric space $G/K$. As a variety, $\sym$ is
diffeomorphic to $NA$ with a left action
\be n(b,r)a(t)\cdot (z,u) = (t^2z + 2tbu + 2r + i|b|^2, tu + i b)\,. \ee
We do not need to describe the action of $K$ explicitly. The point
$(i,0)$ is fixed by $K$. The space $\sym$ inherits the complex
structure of $\CC^2$, which is preserved by the action of G. The
boundary of $\sym$ consists of the points $(z,u)$ with $\im z =|u|^2$,
plus a point at~$\infty$ which is fixed by the subgroup $NAM$ of $G$.
Thus, we have
\be \partial \sym = N\cdot(
0,0) \cup \{\infty\} . \ee

\subsection {Discrete subgroups}\label{discretesubgroups}
We will work with discrete subgroups $\Gm$ of $G$ of finite covolume
with one orbit of cusps, imposing one additional condition that
makes  the formulation of our results as simple as possible.

We can conjugate $\Gm$ by an element $g\in G$ such that the unipotent
subgroup $N$ intersects $\Gm$ in a lattice $\Gm_N $. We impose the
additional condition that $g$ can be chosen in such a way that $\Gm_N$
is a standard lattice $\Ld_\s$, generated by $\nm(1,0)$, $\nm(i,0)$ and
$\nm(0,2/\s)$, with $\s\in \ZZ_{\geq 1}$. In this way we can apply
results from \cite{BMSU21}.\smallskip

In Appendix C in \cite{BMSU21} we discuss two discrete subgroups
satisfying these conditions. The first one is the group
$\Gm_{\!0} = G \cap \SL_2\bigl( \ZZ[i]\bigr)$, which satisfies
$\Gm_{\!0} \cap N = g_\infty \Ld_4 g_\infty^{-1}$ with
$g_\infty = \am(\sqrt 2) \mm(e^{\pi i/4})$.  So
$g_\infty^{-1} \Gm_{\!0} g_\infty$ is a discrete subgroup satisfying
all our conditions.

Francsics and Lax \cite{FL5} study a discrete subgroup in another
realization of $\SU(2,1)$ of the form
\newcommand\fl{{\mathrm{FL}}}$G_\fl=U_\fl^{-1}GU_\fl$ with
\be U_\fl \= \begin{pmatrix}\frac{-i}{\sqrt 2}& 0 & \frac{1}{\sqrt 2}\\
0&1&0\\
\frac{-i}{\sqrt 2} &0& \frac{-1}{\sqrt 2}\end{pmatrix}\,.\ee
They give an explicit fundamental domain in~$G/K$ for
$G_\fl \cap \SL_3\bigl( \ZZ[i]\bigr)$, from which it is clear
that all cusps for this group are equivalent.  Conjugating
$G_\fl$ to our realization~$G$, we get the discrete subgroup
$\Gm_{\!\fl} = G \cap U_\fl \SL_3\bigl( \ZZ[i]\bigr) U_\fl^{-1}$. For
this group one has
$\Gm_{\!\fl} \cap N = \mm(e^{\pi i/4}) \Ld_4 \mm(e^{-\pi i/4})$.
So $\mm(e^{-\pi i/4}) \Gm_{\!\fl} \mm(e^{\pi i/4})$ satisfies
the conditions.
\medskip

The requirements for $\Gm$ simplify 
the book-keeping
of Poincar\'e series and Fourier expansions. These choices are not
essential for our results, but will simplify several expressions.

\subsection{Irreducible representations of \texorpdfstring{$K$}{K}} The
irreducible representations of $K$ are obtained from pairs of
irreducible representations of $U(1)$ and $\SU(2)$ having the same
parity. See \cite[(1), (2) on p.~185]{Wal76}. The irreducible
representations $\tau_p$ of $\SU(2)$ are classified by their dimension
$p+1$ with $p\in \ZZ_{\geq0}$. In this way we arrive at the
representations $\tau_p^h$ of $K$, with $h\equiv p\bmod 2$, $p\geq 0$
given by
\be \label{tauph} \tau^h_p \pmat{\z^{1/2}u&0\\0&\z^{-1}} \=
(\z^{1/2})^{-h}\, \tau_p(u)\,. \ee

The representation $\tau^h_p$ has $p+1$ orthogonal realizations in
$L^2(K)$, each with a basis of polynomial functions $\Kph h p r q$ on
$K$. See the description in Appendix~\ref{app-K}.

\subsection{Decomposition of
\texorpdfstring{$L^2(\Ld_\s\backslash N)$}{L2Ld4N}}\label{sect-decpmLdN}
We consider the quotient $\Ld_\s\backslash N$ for an arbitrary
$\s\in  \ZZ_{\geq 1}$.

We use $dn=dx\,dy\,dr$ for $n=\nm(x+iy,r)$ as the Haar measure on $N$,
so $\Ld_\s\backslash N$ has volume~$\frac2\s$.

\subsubsection{ Irreducible representations} The center of $N$
%, $Z(N)$,
acts by a central character for any irreducible representation of $N$.
If this central character is trivial, the representation corresponds to
an irreducible representation of $N/Z(N)\cong \RR^2$, hence to a
character of $N$. Characters of $N$ occurring in
$L^2(\Ld_\s\backslash N)$ are trivial on $\Ld_\s$, hence they are of
the form
%\ir{chbt}{\ch_\bt}
\be\label{chbt} \ch_\bt: \nm(b,r) \mapsto e^{2\pi i \im (\bar \bt
b)}\,,\ee
with $\bt\in \ZZ[i]$. We note that characters of $N$ are trivial on the
center $\bigl\{ \nm(0,r) \;:\; r\in \RR\bigr\}$.

If the center acts by $\nm(0,r) \mapsto e^{ir\ld}$, then for each fixed
$\ld \in \RR_{\neq 0}$, there is only one isomorphism class of
irreducible representations of $N$ with this central character. It is
%infinite-dimensional,
unitary, and is called the Stone-von Neumann representation.

It can be realized on functions in $L^2(\RR)$ with a central character
of the form $\nm(0,r) \mapsto e^{2\pi i r \ell}$, with
$\ell \in \RR_{\neq 0}$. The action of $\nm(x+iy,r) \in N$ is given by
\be\label{pild}
\pi_{2\pi\ell}\bigl(\nm(x+iy,r)\bigr)\ph (\xi) \= e^{2\pi i \ell (
r-2\xi x
- xy )}\ph(\xi+y)\,. \ee
For $\ell \in \RR_{\neq 0}$, this defines a unitary representation on
$L^2(\RR)$ for the standard inner product that leaves the Schwartz
space invariant. It is called the Schr\"odinger representation.

\subsubsection{Theta-functions}\label{sect_thfonb}
The Stone-Von Neumann representation can also be realized in
$L^2(\Ld_\s\backslash N)$ by means of theta functions.

The central character equals $e^{2\pi i r \ell}$ on $\nm(0,r)$, and is
trivial on $\Ld_\s\cap Z(N)$, hence $\ell \in \frac\s 2 \ZZ$, with
$\s \in \ZZ_{\geq 1}$ as usual.

Thus, for each $c\in \ZZ\bmod 2|\ell|$, we get a realization given by
intertwining operators $\ph \mapsto \Th_{\ell,c}(\ph)$ from the
Schr\"odinger representation of $N$ in the space  $\Schw (\RR)$ of Schwartz functions
on $\RR$, into the module $C^\infty\bigl(\Ld_\s\backslash N\bigr)$ with
the action of $N$ by right translation. The theta functions are given
by
\be\label{Thla} \Th_{\ell,c}(\ph)(\nm(x+iy,r)) \= \sum_{k\in \ZZ}
\ph\Bigl( \frac c{2\ell}+k+y\Bigr)
e^{2\pi i \ell r }\, e^{-2\pi i \ell x(c/\ell+2k+y)}\,.\ee
In \S\ref{app-Theta} we indicate how one can arrive at this intertwining
operator, and how it leads to $2|\ell|$ orthogonal realizations of the
Stone-von Neumann representation in $L^2(\Ld_\s\backslash N)$.\medskip

We shall use a convenient orthogonal basis of $L^2(\RR)$. For
$k\in \ZZ_{\geq 0} $ the Hermite polynomial $H_k$ is determined by the
identity
\be e^{-\xi^2}\, H_k(\xi) \= (-1)^k \,\partial_\xi^k e^{-\xi^2}\,. \ee
The normalized Hermite functions are given by
\be\label{hlk} h_{\ell,k}(\xi)\= 4|\ell|^{1/4} 2^{-k/2}\, (k!)^{-1/2} \,
H_k\bigl(\sqrt{4\pi|\ell|}\xi\bigr)
e^{-2\pi|\ell|\xi^2}\,.\ee
In this way, for any $\ell \in \RR$, $\ell \ne 0$, we obtain an
orthogonal basis $\bigl\{ h_{\ell,k}\,:\, k\in \ZZ_{\geq 0}\bigr\}$ of
$L^2(\RR)$.

Going over to theta-functions, for each
$\ell \in \frac \s2\ZZ_{\neq 0}$, we get an orthogonal system
\be\label{Ldbas} \Bigl\{ \sqrt {\s/2} \Th_{\ell,c}(h_{\ell,k})\;:\;
0\leq c \leq 2|\ell|-1,\; k\geq 0 \Bigr\} \ee
of the subspace $L^2(\Ld_\s\backslash N)_\ell$ determined by the central
character $\nm(0,r)\mapsto e^{2\pi i \ell r}$. On the other hand, an
orthogonal basis of $L^2(\Ld_\s\backslash N)_0$ is given by
$\Bigl \{ \sqrt {\s/2} \; \ch_\bt \;:\; \ch_\bt \in \Z[i]\Bigr\}$. Thus
\bad \label{basisL2Nsigma} \B &\= \Bigl\{ \sqrt {\s/2} \; \ch_\bt,\, \bt
\in \Z[i]\}
\\ &\quad\hbox{} \cup \Bigl\{ \sqrt {\s/2} \;
\Th_{\ell,c}(h_{\ell,k})\;:\; \ell \in \frac \s2\ZZ_{\neq 0}\;,
c=0,1,\ldots,2|\ell|-1\,,\;k \in \Z _\geq 0 \Bigr\}\,
\ead
is an orthogonal basis of $L^2(\Ld_\s\backslash N)$.

\subsection{Haar measures}\label{Haarmeas}

A Haar measure of $dg$ on $G$ can be chosen of the form
$dg = a^{-2\rho}\,dn \, da\, dk$ with respect to the Iwasawa
decomposition $G=NAK$, where $2 \rho $ is the sum of the restricted
roots of $\nlie$, the Lie algebra of $N$. Thus $2 \rho = 4 \ld$, where
$\ld$ is the simple restricted root. We take the coordinate
$t \leftrightarrow \am(t)$ on $A$, and the measure $da = \frac {dt}t$.
The root $\ld$ satisfies $a^\ld = t$, thus
$dg = t^{-5}\,dn \, dt\, dk$.

Furthermore, we normalize the Haar measure on $K$ so that
$\int_K dk = 1$.  
In \S\ref{sect-decpmLdN} we have already fixed
the Haar measure $dn$ on~$N$.

\section{Fourier expansions}\label{sect-Ftm} In this section we give an
overview of the Fourier expansion of $\Ld_\s$-invariant $K$-finite
functions on $G=\SU(2,1)$.

\subsection{Fourier term operators} The basis $\B$
in~\eqref{basisL2Nsigma} can be used to give a generalized Fourier
expansion of elements $f\in C^\infty(\Ld_\s\backslash G)_K$. We
formulate it for a general $\s\in \ZZ_{\geq 1}$, though later we will
use it only for $\s=4$.
\be f(n g) \= \sum_{b\in \B} b(n) \, \int_{\Ld_\s\backslash N}
\overline{b(n_1)} \, f(n_1 g)\, dn_1\ee
with $g\in G$ and $n \in N$. This Fourier expansion converges in the
$L^2$-sense on $\Ld_\s\backslash N$ for $f$ and for all its derivatives
$uf$ with $u\in U(\glie)$.

For the terms with $b = \sqrt{\s/2} \,\ch_\bt$, $\bt\in \ZZ[i]$, we
define the Fourier term operator
\be\label{Fbtdef}
\Four_\bt f(g) \= \frac \s 2 \int_{\Ld_\s\backslash N} \ch_\bt(n_1)^{-1}
\, f(n_1 g)\, dn_1\,.\ee

The functions $\Four_\bt f$ are $\ch_\bt$-equivariant under left
translations by elements of $N$. The \emph{abelian part} of the Fourier
expansion of $f(g)$ is the sum
\be \sum_{\bt\in \ZZ[i]} \Four_\bt f(g)\,.\ee
The operator $\Four_\bt$ is an intertwining operator for the action of
$\glie$ by right differentiation:
$\Four_\bt(\XX f) = \XX\bigl( \Four_\bt f\bigr)$ for $\XX\in \glie$.

The action of $\glie$ does not preserve the individual terms
corresponding to the basis elements
$\sqrt{\s/2} \, \Th_{\ell,c}(h_{\ell,m})$. We group these basis
elements together so that the action of $\glie$ preserves the space of
linear combinations of terms in each group. This leads to operators
$\Four_{\ell,c,d}$ with $\ell \in \frac \s 2 \ZZ_{\neq 0}$,
$c\in \ZZ\bmod 2\ell \ZZ$, and $d\in 1+2\ZZ$. For $n\in N$, $a\in A$,
$k\in K$:
\badl{Flcd} \Four_{\ell,c,d} &f(nak)
\= \sum_{m,h,p,r,q} \Th_{\ell,c}(h_{\ell,m})(n) \, \frac{\Kph h p r
q(k)}{\bigl\|\Kph h p r q \bigr\|^2_K}\\
&\quad\hbox{}\cdot
\frac \s 2 \int_{n_1\in \Ld_\s\backslash N} \int_{k_1\in K}
\overline{\Th_{\ell,c}(h_{\ell,m}(n') }\, f(n_1 a k_1)
\, \overline{\Kph hprq(k_1)}\, dk_1\, dn_1\,,
\eadl
where the sum runs over integers satisfying $m \geq 0$,
$h\equiv p \equiv r \equiv q\bmod 2$, $|r|\leq p$, $|q|\leq p$, and
\be\label{mtpleq} (6m+3) \sign(\ell) + h - 3r \= d\,. \ee
When dealing with these Fourier terms we often abbreviate
$(\ell,c,d) = \n$. We call the sum $\sum_\n \Four_\n f$ is the
\emph{non-abelian part} of the Fourier expansion.

See \S\ref{app-Fto} for a further discussion.

\subsection{Fourier term modules}\label{sect-Ftm0}
As before, we define a \emph{Fourier term order} $\Nfu$ as a realization
of an irreducible representation of $N$ in
$C^\infty( \Ld_\s\backslash N)$; $\Nfu$ is the index in the Fourier
expansion of elements of $C^\infty(\Ld_\s\backslash G)_K$; so
$\Four_\Nfu=\Four_\bt$ for $\Nfu=\Nfu_\bt$ and
$\Four_\Nfu = \Four_{\ell,c,d}$ for $\Nfu=\Nfu_{\ell,c,d}=\Nfu_\n$.
Since the Fourier term operators are intertwining operators of
$(\glie,K)$-modules, the spaces
$\Ffu_\Nfu = \Four_\Nfu C^\infty(\Ld_\s\backslash G)_K$ are
$(\glie,K)$-modules. We call them \emph{ Fourier term modules}.

In the modules $\Ffu_\Nfu$ we have the submodules $\Ffu_\Nfu^\ps$ in
which the center $ZU(\glie)$ of the universal enveloping algebra of
$\glie$ acts according to the character $\ps$ of $ZU(\glie)$. These
characters $\ps$ are determined by their values on the generators $C$
and $\Dt_3$. So $f\in \Ffu_\Nfu$ is in $\Ffu^\ps_\Nfu$ if
$C f = \ld_2 \, f$ and $\Dt_3 f = \ld_3 f$ for fixed
$\ld_2, \ld_3\in \CC$. Later on we deal mainly with the Casimir element
$C$, and abbreviate $\ld_2$ as~$\ld$.

The characters of $ZU(\glie)$ are parametrized by the spectral
parameters $(\xi,\nu)$ where $\xi$ is a character of $M$ and
$\nu\in \CC$ corresponds to the character $\am(t) \mapsto t^{\nu}$
of~$A$. A character $\xi$ of $M$ has the form
$\mm(\z) \mapsto \z^{j_\xi}$ with $j_\xi\in \ZZ$. Furthermore, if
$(j_1,\nu_1)$ and $(j_2,\nu_2) $ in $\ZZ\times\CC$ are in the same
orbit of the Weyl group, then they determine the same character of
$ZU(\glie)$.

The Weyl group of $\SU(2,1)$ is isomorphic to the symmetric group on
three elements. Its action on $\CC\times\CC$ is generated by the
transformation of $\CC^2$ with matrices
\be \Ws 1 \= \begin{pmatrix}-\frac 12& \frac 32 \\ \frac12 & \frac12
\end{pmatrix}\quad\text{ and }\quad \Ws2\=\begin{pmatrix}-\frac 12&
-\frac 32 \\ -\frac12 & \frac12
\end{pmatrix}\,. \ee
The Weyl group does not preserve $\ZZ\times\CC\subset \CC^2$. By
$\Wo(\ps)
$ we indicate the intersection of the Weyl group orbit corresponding to
$\ps$ with $\ZZ \times \CC$, and by $\Wo(\ps)^+ $ the subset
$ \bigl\{ (j,\nu) \in \Wo(\ps)\;:\; \re\nu\geq 0 \bigr\}$. The set of
characters $\ps$ of $ZU(\glie)$ can be divided in the following way.
\begin{itemize}
\item \emph{Generic parametrization: } All $(j,\nu) \in \Wo(\ps)$
satisfy $\nu  \notin  j+2\ZZ$, or $\Wo(\ps)=\bigl\{(0,0)\bigr\}$.
\item \emph{Integral parametrization: } All $(j,\nu) \in \Wo(\ps)$
satisfy $\nu \in j+2\ZZ$ and $(j,\nu) \neq (0,0)$.
\end{itemize}
\medskip

The study of the structure of the modules $\Ffu^\ps_\Nfu$ is the main
object of \cite{BMSU21}. Here we summarize the main results, with more
details in the appendix. The main difference in the structure is
between the case when $\Nfu=\Nfu_0$ and the other Fourier term orders.

Under generic parametrization the modules $\Ffu_\Nfu^\ps$ for are
isomorphic for all choices of $\Nfu$. The main difference in the
structure occurs under integral parametrization, between $\Nfu=\Nfu_0$
and the other $\Nfu$. The module $\Ffu_{\Nfu_0}^\ps$ still is the sum
of a number of principal series modules, whereas for $\Nfu\neq \Nfu_0$
the structure is more complicated.

\subsubsection{\texorpdfstring{$N$}{N}-trivial Fourier term modules} The
$N$-trivial Fourier term module $\Ffu^\ps_0$ contain the principal
series modules $H^{\xi,\nu}_K$ for which $(j_\xi,\nu) \in \Wo(\ps)$.
They consist of the elements of $C^\infty(N\backslash G)_K$ that
transform on the left according to the character of $AM$,
$\am(t)\mm(\z) \mapsto \z^{j_\xi} t^{2+\nu}$. For different spectral
parameters $(\xi,\nu)$ the corresponding principal series are disjoint.
If $\Wo(\ps)$ does not contain elements of the form $(\xi,0)$, then
\be\label{Ffu0decomp} \Ffu^\ps_0 \= \bigoplus _{(j_\xi,\nu) \in
\wo(\ps)} H^{\xi,\nu}_K\,.\ee

\subsubsection{Submodules determined by boundary behavior} In the
modules $\Ffu^\ps_\Nfu$ with $\Nfu\neq \Nfu_0$ there are submodules
having the following characterization:
\begin{enumerate}
\item $\Wfu^\ps_\Nfu$ consists of the elements $f\in \Ffu^\ps_\Nfu$ with
exponential decay
\be f\bigl( n \am(t) k \bigr) \= \oh(e^{-\e t}) \qquad\text{ as
}t\uparrow \infty\,,\ee
uniform in $n$ and $k$ for some $\e>0$.
(In the non-abelian case $\Nfu=\Nfu_{\ell,c,d}$ we can replace
$e^{-\e t}$ by $e^{-\e t^2}$.)

\item $\Mfu^\ps_\Nfu$ consists of the linear combinations of elements
$f\in \Ffu^\ps_\Nfu$ for which, for some $(j_\xi,\nu) \in \Wo(\ps)^ +$,
the function
\be \label{nureg} t \mapsto f\bigl( n \am(t) k\bigr)
\, t^{-2-\nu} \ee
extends for each $n$ and $k$ to a holomorphic function of $t\in \CC$.
\end{enumerate}
These spaces are $(\glie,K)$-submodules by \cite[Proposition
10.6]{BMSU21}. In the particular case of $\Nfu=\Nfu_0$ we obtain that
$\Wfu_0^\ps=\{0\}$, and that $\Mfu^\ps_0$ contains all $H^{\xi,\nu}_k$
with $(j_\xi,\nu) \in \Wo(\ps)$ and $\re \nu \geq 0$.

The main theorems in \cite[\S2]{BMSU21} give a further description of
the modules $\Ffu^\ps_\Nfu$ with $\Nfu\neq \Nfu_0$. We summarize the
results.

\subsubsection{Abelian Fourier term modules}Let $\bt\in \ZZ[i]$,
$\bt\neq0$. Then there is a direct sum decomposition
\be \Ffu^\ps_\bt \= \Wfu^\ps_\bt \oplus \Mfu^\ps_\bt\,.\ee

For each $(j_\xi,\nu) \in \Wo(\ps)^+$ there are submodules
$\Wfu^{\xi,\nu}_\bt\subset\Wfu^\ps_\bt$ and
$\Mfu^{\xi,\nu}_\bt \subset\Mfu^\ps_\bt$ of the form
\badl{MWxinubt} \Wfu^{\xi,\nu}_\bt &\= U(\glie)\,
\om^{0,0}(j_\xi,\nu)\,,\qquad \Mfu^{\xi,\nu}_\bt \= U(\glie)\,
\mu^{0,0}(j_\xi,\nu)\,, \textrm{ where }\\
\om^{0,0}_\bt&(j_\xi,\nu)\bigl( n \am(t) k \bigr)
\= \ch_\bt(n) \, t^2 K_{\nu}(2\pi|\bt|t)\, \Kph{2j_\xi}000(k)\,\\
\mu^{0,0}_\bt&(j_\xi,\nu)\bigl( n \am(t) k \bigr)
\= \ch_\bt(n) \, t^2 I_{\nu}(2\pi|\bt|t)\, \Kph{2j_\xi}000(k)\,
\eadl
with the modified Bessel functions $I_\nu$ (exponentially increasing)
and $K_\nu$ (exponentially decreasing).

Under generic parametrization the modules $\Wfu^{\xi,\nu}_\bt$ and
$\Mfu^{\xi,\nu}_\bt$ are isomorphic to the principal series module
$H^{\xi,\nu}_K$, and
\be \Wfu_\bt^\ps \= \bigoplus_{(j_\xi,\nu) \in \Wo(\ps)^+}
\Wfu_\bt^{\xi,\nu}
\qquad \Mfu_\bt^\ps \= \bigoplus_{(j_\xi,\nu) \in \Wo(\ps)^+}
\Mfu_\bt^{\xi,\nu} \,.\ee

Under integral parametrization the modules $\Wfu^{\xi,\nu}_\bt$ and
$\Mfu^{\xi,\nu}_\bt$ are reducible, isomorphic to each other, but for
most $(j_\xi,\nu)$ not isomorphic to $H^{\xi,\nu}_K$. We now have
\be \Wfu_\bt^\ps \= \sum_{(j_\xi,\nu) \in \Wo(\ps)^+} \Wfu_\bt^{\xi,\nu}
\qquad \Mfu_\bt^\ps \= \sum_{(j_\xi,\nu) \in \Wo(\ps)^+}
\Mfu_\bt^{\xi,\nu} \,.\ee
If a $K$-type $\tau^h_p$ occurs in modules $\Wfu^{\xi,\nu}_\bt$ and
$\Wfu^{\xi',\nu'}_\bt$ then the (one-dimensional) subspaces
$\Wfu^{\xi,\nu}_{\bt;h,p}$ and $\Wfu^{\xi',\nu'}_{\bt;h,p}$ with this
$K$-type coincide, and similarly for $\Mfu^{\xi,\nu}_\bt$.

\subsubsection{Non-abelian Fourier term modules} In the non-abelian case
several complications occur. We consider $\n=(\ell,c,d)$ with
$\ell \in \frac \s2 \ZZ_{\neq 0}$, $c\bmod 2\ell$, and $d\in 1+2\ZZ$.

We need a further restriction on the elements of Weyl orbits. Set
\be \Wo(\ps)^+_\n \= \bigl\{ (j,\nu) \in \Wo(\ps)\;:\; \re\nu\geq 0,\;
\sign(\ell)(2j-d)+ 3 \leq 0 \bigr\}\,.\ee
In case the set $\Wo(\ps)^+_\n$ is empty, then $\Ffu^\ps_\n = \{0\}$.
For each $(j_\xi,\nu)_\n^+$ there are submodules
$\Wfu^{\xi,\nu}_\n \subset \Wfu^\ps_\n$ and
$\Mfu^{\xi,\nu}_\n \subset\Mfu^\ps_\n$. These modules contain the
elements
\badl{om-mu-nab} \om^{0,0}_\n(j_\xi,\nu) \bigl( n\am(t)k\bigr) &\=
\Th_{\ell,c}(h_{\ell,m_0}(n)) \, t W_{\k,\nu/2}(2\pi|\ell|t^2)\,
\Kph{2j_\xi}000(k)
&&\text{ for } \Wfu_\n^{\xi,\nu}\,,\\
\mu^{0,0}_\n(j_\xi,\nu) \bigl( n\am(t)k\bigr) &\=
\Th_{\ell,c}(h_{\ell,m_0}(n)) \, t M_{\k,\nu/2}(2\pi|\ell|t^2)\,
\Kph{2j_\xi}000(k)
&&\text{ for } \Mfu_\n^{\xi,\nu}\,,
\eadl
with
\be \label{m0kap}m_0 \= \frac{\sign(\ell)}6(d-2j_\xi)-\frac12\,,\qquad
\k \= - m_0
-\frac12(j_\xi\sign(\ell)+1)\,.\ee
We have that $m_0 \in \ZZ_{\geq 0}$, and $\k\in \frac12\ZZ$. Here we use
the Whittaker functions $W_{\k,s}$ (exponentially decreasing) and
$M_{\k,s}$ (exponentially increasing for most combinations of $\k$ and
$s$).

Under generic parametrization, the modules $\Wfu^{\xi,\nu}_\n $ and
$\Mfu^{\xi,\nu}_\n$ are generated respectively by
$\om^{0,0}_\n(j_\xi,\nu)$ and $\mu_\n^{0,0}(j_\xi,\nu)$. Both modules
are irreducible and isomorphic to the principal series module
$H^{\xi,\nu}_K$. Moreover,
\begin{align}
\label{na-WMd} \Wfu^\ps_\n &\= \bigoplus_{(j_\xi,\nu) \in \Wo(\ps)_\n^+}
\Wfu^{\xi,\nu}\,,
\qquad\Mfu^\ps_\n = \bigoplus_{(j_\xi,\nu) \in \Wo(\ps)_\n^+}
\Mfu^{\xi,\nu}\,,\\
\label{nads}
\Ffu^\ps_\n &\= \Wfu^\ps_\n \oplus \Mfu^\ps_\n\,.
\end{align}

Under integral parametrization both \eqref{nads} and \eqref{na-WMd}
break down. For some characters $\ps$ the modules $\Wfu^\ps_\n$ and
$\Mfu^\ps_\n$ have a non-empty intersection and do not span
$\Ffu^\ps_\n$. The modules $\Wfu^{\xi,\nu}_\n$ and $\Mfu^{\xi,\nu}_\n$
are reducible, and in most cases non-isomorphic. Also, the elements
in~\eqref{na-WMd} need not generate the modules.

There are several properties that stay true under integral
parametrization. In \eqref{na-WMd} we have to replace $\bigoplus$ by
$\sum$. The $K$-type $\tau^h_p$ occurs in $\Wfu^{\xi,\nu}_\n$ and in
$\Mfu^{\xi,\nu}$ with multiplicity one for $|h-2j_\xi|\leq p$. If
$\tau^h_p$ satisfies this condition for $(j_\xi,\nu)$ and
$(j_{\xi'},\nu')$ in $\Wo(\ps)^+_\n$, then
$\Wfu^{\xi,\nu}_{\n;h,p} = \Wfu^{\xi',\nu'}_{\n;h,p}$ and
$\Mfu^{\xi,\nu}_{\n;h,p} = \Mfu^{\xi',\nu'}_{\n;h,p}$.

In \S\ref{app-famFt} in the Appendix we give a further
discussion of elements of Fourier term modules.
Proposition~\ref{prop-MuOm} gives, for $\Nfu\neq \Nfu_0$, families
$\nu \mapsto \Om_{\Nfu;h,p,q}(\xi,\nu)$ and
$\nu \mapsto \Mu_{\Nfu;h,p,q}(\xi,\nu)$ that we will use in the
description of Fourier terms of automorphic forms and in the construction
of Poincar\'e series.

\section{Automorphic forms}\label{sect-afPs}
In this section we give the definitions of several modules of
automorphic forms, and discuss the corresponding Fourier expansions.
After that we turn to Poincar\'e series, which provide us with an
explicit construction.

\subsection{General facts}\label{sect-af}

In \S\ref{discretesubgroups} we have introduced the class of discrete
subgroups $\Gm$ that we will use in this paper.

\begin{defn}\label{def-aut}
An \emph{automorphic form} for $\Gm$ is a smooth function $f$ on
$G=\SU(2,1)$ with the following properties:
\begin{enumerate}
\item[a)] $f$ is $\Gm$-invariant on the left and $K$-finite on the
right: $f\in C^\infty(\Gm\backslash G)_K$\,,
\item[b)] $f$ is an eigenfunction of the center of the enveloping
algebra: $ u f= \ps(u)$ for all $u\in ZU(\glie)$, for some character
$\ps$ of $ZU(\glie)$\,,
\item[c)] $f$ has polynomial growth:
\be f\bigl( n \am(t) k \bigr) \= \oh(t^a)\qquad\text{ as
$t\uparrow\infty$ for some }a\in \RR\,,\ee
uniform in $n\in N$ and $k\in K$.
\end{enumerate}
\end{defn}

\begin{defn}\label{autps} We denote by $\A(\ps)$ the $(\glie,K)$-module
of automorphic forms with character $\ps$ of $ZU(\glie)$. By
$\A(\ps)_{h,p}$ we denote the subspace of $\A(\ps)$ of $K$-type
$\tau^h_p$, and by $\A(\ps)_{h,p,q}$ the subspace of weight $q$ in that
$K$-type.
\end{defn}
The spaces $\A(\ps)_{h,p}$ have finite dimension (see Theorem~1 of
\cite[p~8]{HCh68}, for instance).

\

As in \ref{Haarmeas}, we use the Haar measures $dn=dx\,dy\,dr$ on $N$,
for $n=\nm(x+\nobreak i y,r)$, $da(t)=t^{-1}dt$ on $A$, for $a=\am(t)$,
and the Haar measure $dk$ on $K$ that gives to $K$ volume~$1$. Then
$dg= t^{-5} \, dn\, dt\, dk$ for $g=n\am(t)k$ is a Haar measure on~$G$.

  In $L^2(\Gm\backslash G)$ we use the standard scalar product
\be\label{ip} (f_1,f_2)\= \int_{\Gm\backslash G}
f_1(g)\,\overline{f_2(g)}\, dg\,\ee
and right translation gives a unitary representation of $G$ on
$L^2(\Gm\backslash G)$. We will be concerned mainly with the subspace
$L^2(\Gm\backslash G)_K$ of $K$-finite vectors. The integral in
\eqref{ip} can be computed by letting $g=nak$ run over $K$ and
$na \in NA$ be such that $na\cdot (i,0)$ runs over a fundamental domain
of $\Gm\backslash \sym$. Then the norm in $L^2(\Gm\backslash G)_K$ can
be bounded by an integral over a Siegel set.

\begin{defn}The submodule $\Au{(2)}(\ps)$ of square integrable
automorphic forms consists of the functions $f\in \A(\ps)$ that
represent an element of $L^2(\Gm\backslash G)_K$.
\end{defn}

The space $\Au{(2)}(\ps) \subset L^2(\Gm\backslash G)$ decomposes as the
orthogonal direct sum of finitely many irreducible $(\glie,K)$-modules
$V_i$ with a unitary structure, in which $ZU(\glie)$ acts via the
character $\ps$. So there are finitely many isomorphism classes to
which the $V_i$ belong. Furthermore, the sets of $K$-types occurring in
different isomorphism classes are disjoint (see for instance
\cite[Proposition 12.2]{BMSU21}). That means that each space
$\Au{(2)}(\ps)_{h,p}$ generates a $(\glie,K)$-module that is a finite
sum of copies of one irreducible $(\glie,K)$-module with a unitary
structure.

   \rmrk{Fourier terms} There is an absolutely convergent Fourier
expansion
\be f (g) \= \sum_\Nfu \Four_\Nfu f (g)\,,\ee
where the Fourier term order $\Nfu$ runs over a system of realizations
of irreducible representations of $N$ in
$C^\infty(\Ld_\s\backslash N)$. So $\Nfu$ runs over the $\Nfu_\bt$ with
$\bt\in \ZZ[i]$ and the $\Nfu_\n$ with $\n=(\ell,c,d)$,
$\ell \in \frac \s2\ZZ_{\neq 0}$, $c\bmod 2\ell$, and $d\in 1+2\ZZ$.

  In the case of automorphic forms, the Fourier term operators are
intertwining operators of $(\glie,K)$-modules
$\Four_\Nfu : \A(\ps) \rightarrow \Ffu^\ps_\Nfu$. The property of
polynomial growth is preserved by the Fourier term operators, and hence
imposes restrictions on the Fourier terms. All elements in $\Ffu_0^\ps$
have polynomial growth. For $f\in \A(\ps)$ the $N$-trivial Fourier term
$\Four_0 f$ is, in general, a sum of contributions in the principal
series modules $H^{\xi,\nu}_K$ with $(\xi,\nu) \in \Wo(\ps)$, a set
with at most six elements. However, on each space $\A(\ps)_{h,p}$ for a
  fixed $K$-type, only two of these contributions can be non-zero. In
case $\Wo(\ps)$ contains elements of the form $(j_\xi,0)$ we may also
have logarithmic contributions.

For all $\Nfu\neq \Nfu_0$ the property of polynomial growth forces that
$\Four_\Nfu f \in \Wfu^\ps_\Nfu$. The multiplicity of each $K$-type in
$\Wfu^\ps_\Nfu$ is at most $1$. The operator
$\Four_\Nfu: \A(\ps)_{h,p,q} \rightarrow \Wfu_{\Nfu;h,p,q}^\ps$ lies in
a one-dimensional space for each $K$-type and weight.  (We use
the notation given in Definition~\ref{autps}.) For each
$f\in \A(\ps)_{h,p,q}$ there is a Fourier coefficient $C_\Nfu(f)$ such
that
\be \Four_\Nfu f \= c_\Nfu(f)\, \Om_{\Nfu;h,p,q}(\xi,\nu)\,,\ee
with $\ps$ determined by $(\xi,\nu)$. See Proposition~\eqref{prop-MuOm}
for the family $\Om_{\Nfu;h,p,q}$.
\smallskip

The integral in~\eqref{ip} can be restricted to a fundamental domain
contained in a Siegel set of the form
\be \Sie(t_0) \= \bigl\{ \nm(z,r) \am(t) k\;:\; |\re z|,\,|\im z|\leq
\tfrac12,\; |r|\leq \tfrac1\s,\, t\geq t_0,\, k\in K\bigr\} \ee
with $t_0>0$. For $f\in \Au{(2)}(\ps)$ the integral of $|f|^2$ over this
Siegel set is finite, a property inherited by the Fourier terms. This
implies that $\Four_0 f$ can have only components in principal series
modules $H^{\xi,\nu}_K$ with $ \re \nu<0$. The square integrability
imposes no further restriction on the other Fourier terms.\medskip

The form of the Fourier expansion can also be used to define other
modules of automorphic forms.

\begin{defn}\label{def-cu-meg}
The submodule $\Au 0(\ps)$ of \emph{ cusp forms } consists of the
functions $f\in \A(\ps)$ for which $\Four_0 f=0$.

An automorphic form is \emph{ rapidly decreasing } if for all $a\in \RR$
\be\label{qd} f\bigl( n\am(t)k\bigr) \= \oh(t^{-a})\quad \text{ as }
t\uparrow\infty\,,\ee
uniformly in $n$ and $k$.
\end{defn}

Cusp forms are rapidly decreasing
(see for instance \cite[\S4, Chap I]{HCh68}, or Corollary to Lemma~3.4
on \cite[p 46]{Lal76}).

All Poincar\'e series in this paper have exponential growth concentrated
  in a few Fourier terms. We define

\begin{defn}\label{def-A!}Let $\ps$ be a character of $ZU(\glie)$.\! We
say that a function $f\!\in\! C^\infty(\Gm\backslash G)_K$ that
satisfies $u f=\ps(u)\, f$ for all $u\in ZU(\glie)$, has \emph{
moderate exponential growth} if there exists a finite set $E$ of
Fourier term orders containing $\Nfu_0$ such that
\be f - \sum_{\Nfu\in E} \Four_\Nfu f \ee
is in $L^{2-\al}\bigl( \Ld_\s\backslash N A_T K)_K$ for some
$\al >0, T>0$, where $A_T = \bigl \{ \am(t) \;:\;t\geq T \bigr \}$. We
denote by $\Au!(\ps) \supset \A (\ps)$ the space of automorphic forms
of moderate exponential growth.
\end{defn}

The dimensions of the spaces $\Au0(\ps)_{h,p}$ are finite, by the
inclusion $\Au0(\ps)\subset\A (\ps)$. The choice of the set $E$ in
Definition~\ref{def-A!} depends on $f$. The spaces $\Au!(\ps)_{h,p}$
can have infinite dimensions. For instance for $\ps=\ps[\xi,\nu]$
described by generic parametrization and for one-dimensional $K$-types
$\tau^h_p=\tau^{2j_\xi}_0$, the construction of Poincar\'e series will
provide us with elements $f = \M_\bt(\ps,\nu) \in \Au!(\ps)_{2j_\xi,0}$
for which $\Four_\bt f $ has genuine exponential growth if $\bt$ lies
in one $\ZZ[i]^\ast$-orbit $\{ \bt_0, i \bt_0, - \bt_0,-i\bt_0\}$ in
$\ZZ[i]\setminus \{0\}$, and has exponential decay for all other
$\bt\in \Z[i]\setminus\{0\}$. See Proposition~\ref{prop-Ps}.

There is also the much larger module of functions $\A^{u} (\ps)$,
satisfying only conditions a) and~b) in Definition~\ref{def-cu-meg},
the space of automorphic forms with \emph{ unrestricted growth}.
However, for the purposes of this paper, the module $\Au!(\ps)$ is all
we need.

\subsection{Poincar\'e series}\label{sect-genPs}

Given a $K$-finite function $f$ on $\Gm_{\!N}\backslash G $, under some
standard growth conditions one can form the Poincar\'e series
\be\label{poindef} \poin f(g) \= \sum_{\gm \in \Gm_{\!N}\backslash \Gm}
f(\gm g)\,. \ee
For instance, if we take $f(n \am(t) k) = t^{2+\nu}\,\Kph h{p}{r}{q}(k)$
with $\re \nu>2$ and $h-3r=2 j_\xi$ then we have the Eisenstein series
\be \label{E}
E(\xi,\nu)_{ h,p,q} (g) \= \poin f(g) \,.\ee
The convergence is absolute and uniform on compact subsets for
$\re \nu>2$, and $E(\xi,\nu)_{h,p,q}$ depends holomorphically on~$\nu$
in this region. It is non-zero only if $\xi\bigl( \mm(i) \bigr)=1$ (see
\eqref{M-def}), which is equivalent to $j_\xi\in 4\ZZ$. The Eisenstein
series $E(\xi,\nu)_{h,p,q}$ are elements of $ \A_{h,p,q}(\ps)$ for
$\re\nu>2$, and have a meromorphic continuation to $\mathbb C$. We will
apply the same approach to elements of $\Mfu^{\xi,\nu}_\Nfu$ for
$\Nfu\neq \Nfu_0$.\smallskip

Lemma 2.1 in~\cite{MW89} applied to the present situation (with $\rho $
corresponding to $\nu=2$) implies that, if
$f\in C(\Gm_ N \backslash G)_K$ satisfies
\be\label{est0nu} f(n\am(t) k ) = \oh\bigl(\am(t)^{2 \rho+\e} \bigr) =
\oh \bigl( t^{4+\e} \bigr)\qquad\text{ as }t\downarrow 0,\ee
for some $\e>0$, $\poin f$ converges absolutely and uniformly for
$n\in N$, $k\in K$. Moreover, this convergence is uniform when $g$ and
$\nu$ vary over compact sets. This implies that if $\XX\in \glie$, and
both $f$ and $\XX f$ satisfy~\eqref{est0nu}, then $\XX \poin f$ is
well-defined and equal to $\poin \XX f$.

For $f\in C^\infty(\Gm_ N\backslash G)_K$ satisfying
condition~\eqref{est0nu} we have the decomposition
\be \label{poin-split}\poin f \= \poin^\infty f + \poin' f\,,\ee
where $\poin^\infty f$ is given by the finite sum over
$\Gm_{\!N}\backslash \Gm_P$, and $\poin' f$ is given by the remaining
sum over $\Gm_{\!N}\backslash (\Gm\cap N \wm AMN)$, with $N \wm AMN$
the big cell in the Bruhat decomposition.

We will make use of the following lemma. We give a proof for
completeness.
\begin{lem}\label{lem-p'}For $f\in C^\infty(\Gm_N\backslash G)_K$
satisfying \eqref{est0nu} one has
\be \poin' f\bigl( n \am(t) k \bigr) \= \oh(1)\qquad\text{ as }
t\uparrow\infty\ee
uniformly in $n\in N$, $k\in K$.
\end{lem}

\begin{proof}

Let
\be A_T \=\bigl\{ \am(t) \;:\; t\ge T\bigr\}\,,\qquad A_{<T}\= A
\setminus A_T\,.\ee

We have to prove that $\poin'f$ is bounded on a region $N A_T K$ for
some sufficiently large $T>1$. Below we will argue that if $T$ is
sufficiently large, then
\be \label{gmincl} \gm N A_T K \subset NA_{<1}K \quad\text{ for all }
\gm\in \Gm \cap N \wm AMN\,.\ee
So we need the bound in~\eqref{est0nu} only for $t<1$, and can assume
that $f$ is the positive function $f(n \am(t) k ) = t^{2+\s}$ with
$\s>2$. So we consider a part of the Eisenstein series
$E(1,\s)_{0,0,0}$.

We will also show below that there are a relatively compact neighborhood
$\Om$ of $1\in G$ and a positive number $T_1$ such that for all
$g\in NA_{<T_1}$
\be \label{2est} f(g) < 2 f(g_1) \quad\text{ for all } g_1\in g\Om\,.\ee

Now we apply the well-known approach of estimating a sum by an integral.
For $p\in NA_T$\newcommand\vol{\mathrm{vol}}
\begin{align*}
\sum_{\gm\in \Gm_N \backslash(\Gm\cap N\wm AMN)}& f(\gm p) \;<\;
\sum_{\gm\in \Gm_N \backslash(\Gm\cap N\wm AMN)} \frac1{\vol\Om}\,
\int_{g_1 \in p\Om} 2\, f(\gm g_1) \, dg_1
\displaybreak[0]\\
&\;\leq\; \frac{2 n}{\vol \Om} \int_{\Gm_N \backslash N A_{<T_1}K } \,
f(g_1)\, dg_1 \;\leq \; C_1 \int_{t=0}^{T_1} t^{2+\s }
\frac{dt}t^5<\infty\,,
\end{align*}
uniform in~$g$. The constant $C_1$ depends on the maximal number of
fundamental domains for $\Gm\backslash G$ intersecting a set $g \Om$,
on the volume of $\Om$, on the volumes of $K$ and $\Gm\backslash N$,
and on $\s$. Since $f$ is left-invariant under~$N$ we can replace the
integral over $\gm p \Om$ by an integral over a $\Gm_N$-equivalent set
inside a fundamental domain for $\Gm_N\backslash G$.

The bound is uniform for $g\in NA_T$. Since $f$ is right invariant under
$K$ this gives the lemma.\smallskip

We still have to establish~\eqref{gmincl} and \eqref{2est}. We use
Lemma~2.1 in \cite{BMSU21} which gives explicitly the transition from
the big cell of the Bruhat decomposition to the Iwasawa decomposition.
In particular,
\badl{BrIw} N\wm &\am(t_1) m \nm(b,r) \am(t) K \= N \am(t') K\,,\\
t' &\= \frac t{t \, |D|}\,,\qquad D= 2ir+t^2+|b|^2\,.\eadl

We take $\Om$ of the form
\[ \Om=\bigl\{ \nm(b,r) \am(t) k\in NAK\;:\; |b|<\e, |r|<\e, \, k\in
K,\, (1-\e) < t < (1+\e) \bigr\} \]
with $\e >0$ small. Elements of $na(t)\Om\in NA\Om$ have the form
\[ \nm(b,r) \am(t) \nm(b_\e,r_\e)\am(t_\e) k_\e \in N \am(tt_\e) K\,,\]
with $tt_\e = t\bigl(1+\oh(\e)\bigr)$.

We apply this with
\[\gm= n_1 \wm \am(t_1) m \nm(b_1,r_1)\in N\wm AMN\,,\qquad p =
\nm(b_2,r_2) \am(t_2) k_2 \in \nm(b,r)\am(t) \Om\,.\]
Then
\begin{align*}
\gm p &\= n_1 \wm \am(t_1) m \nm(b_1,r_1) t \nm(b_2,r_2) \am(t_2) k_2\\
&\qquad\hbox{} \in N \wm \am(t_1) m
\nm\bigl(b_1+b_2,r_1+r_2+\im(\overline{b_1}b_2) \bigr) \am(t_2) K
\\
&\= N \am(t') K\,,\qquad \text{ with } t' \= \frac{t_2}{t_1 \bigl| 2i
r_3+t_2^2+ |b_3| ^2 \bigr|}\,.
\end{align*}
Now we note that $t_2= t \, \bigl(1+\oh(\e) \bigr)$, which implies that
$t'$ can be made uniformly small by increasing~$t$. This
gives~\eqref{gmincl}. \smallskip

For \eqref{2est} we take $p\in N \am(t)$ and
$q=n_\e\am(t_\e)k_\e\in \Om$. Then $ pq \in N \am(t t_\e) K $, with
$tt_\e<t(1+c\e)$ for some $c>1$. It suffices to take
$T_1<(1+c\e)^{-1}$.
\end{proof}
\emph{ Automorphic Poincar\'e series.} We now apply $\poin$ to a
meromorphic family of elements in a module $\Mfu^{\xi,\nu}_\Nfu$ with
$\re\nu>2$, to obtain an automorphic Poincar\'e series.

We work with a fixed Fourier term order $\Nfu=\Nfu_\bt$,
$\bt\in \ZZ[i]\setminus\{0\}$ or $\Nfu=\Nfu_{\ell,c,d}$, with a
character $\xi$ of~$M$, and a variable spectral parameter~$\nu$. All
functions $f \in \Mfu_\Nfu^{\xi,\nu}$ satisfy~\eqref{nureg}, hence the
Poincar\'e series $\poin f$ are well defined and $f\mapsto \poin f$ is
an intertwining operator of $(\glie,K)$-modules
$\Mfu_\Nfu^{\xi,\nu} \mapsto C^\infty(\Gm\backslash G)_K$.

\begin{prop}\label{prop-Mudef}Let $\xi$ be a character of~$M\subset K$
and let $\Nfu$ be a Fourier term order that is not $N$-trivial. Then,
there exist non-zero meromorphic families $\Mu_\Nfu(\xi,\nu)$ of
  functions on ${\Gm_ N} \backslash G $ %for $\nu \in \CC$
   such that:
\begin{enumerate}
\item[a)] For each $g\in G$, the function
$\nu \mapsto \Mu_\Nfu(\xi,\nu)(g)$ is meromorphic in $\CC$ with
%at most
first order singularities occurring possibly at points in
$\ZZ_{\leq -1}$.
\item[b)] If $\Mu_\Nfu(\xi,\nu)$ is holomorphic at $\nu_0$, then its
value at $\nu = \nu_0$ is in $\Ffu^{\ps[\xi,\nu_0]}_\Nfu$, and if
furthermore $\re\nu_0\geq 0$, then its value is in
$\Mfu^{\xi,\nu_0}_\Nfu$. (By $\ps[\xi,\nu]$ we denote the character of
$ZU(\glie)$ parametrized by $( \xi,\nu)$.)

\item [c)] If $\Mu_\Nfu(\xi,\nu)$ has a singularity at
$\nu_0 \in \ZZ_{\leq -1}$, then it is a simple pole and the residue at
$\nu = \nu_0$ is an element of $\Ffu^{\ps[\xi,\nu_0]}$.
\end{enumerate}
\end{prop}
Examples of such families are $\mu^{0,0}_\bt(j,\xi,\nu)$ in
\eqref{MWxinubt} and $\mu^{0,0}_\n(j,\xi, \nu)$ in~\eqref{om-mu-nab}.
 See Appendix B, Proposition~\ref{prop-MuOm} for more
details.

\rmrk{Remark}In this definition we only prescribe the Fourier term order
$\Nfu$ and the character $\xi$ of $M$. The holomorphy ensures that in a
given family only finitely many $K$-types occur.

\begin{prop}\label{prop-Ps}Let $\Nfu\neq \Nfu_0$ be a Fourier term
order, and let $\xi$ be a character of~$M$. For each family $\Mu_\Nfu$
as in Proposition~\ref{prop-Mudef}, the automorphic Poincar\'e series
\be \label{PMdef}\M_\Nfu(\xi,\nu) \;:= \;\poin \Mu_\Nfu(\xi,\nu) \ee
is a well-defined element of $\Au!\bigl([\ps[\xi,\nu]\bigr)$ for each
$\nu$ with $\re\nu>2$. Here, by $\ps[\xi,\nu]$ we denote the character
of $ZU(\glie)$ with spectral parameters $(\xi,\nu)$. Furthermore,
\begin{enumerate}
\item[(i)] The family $\nu \mapsto \M_\Nfu(\xi,\nu)$ is holomorphic on
the half-plane $\re\nu>2$.
\item[(ii)] For each $\XX\in \glie$ the relation
$\XX \M_\Nfu(\xi,\nu) \= \poin \XX \Mu_\Nfu(\xi,\nu)$ holds.
\item[(iii)] The Fourier terms $\Four_{\Nfu'} \M_\Nfu(\xi,\nu)$ are
holomorphic families of elements of $\Ffu_{\Nfu'}^{\xi,\nu}$ for
$\re\nu>2$. We have
\begin{enumerate}
\item[a)]\emph{Abelian cases. }If $\Nfu=\Nfu_\bt$,
$\bt\in \ZZ[i]\setminus\{0\}$, then
$\Four_{\Nfu'} \M_\bt(\xi,\nu) \in \Wfu_{\Nfu'}^{\xi,\nu}$ for all
$\Nfu'\not\in \Bigl\{\Nfu_{i^a\bt}\;:\; a\bmod 4\Bigr\}\cup\{\Nfu_0\}$.
\item[b)] \emph{Non-abelian cases. }If $\Nfu = \Nfu_{\ell,c,d}$, then
$\Four_{\Nfu'} \M_{\ell,c,d}(\xi,\nu)\in \Wfu_{\Nfu'}^{\xi,\nu}$ for
all $\Nfu'\not\in \Bigl\{ \Nfu_{\ell,c',d}\;:\; c'\bmod 2\ell\Bigr\}
\cup\{\Nfu_0\}$.
\end{enumerate}
\end{enumerate}
\end{prop}
%\rmrk{Remark} When we work with a fixed family $\Mu_\Nfu$, we will use the short notation $\M_\Nfu(\xi,\nu)$. However, in situations where confusion may arise, the notation $\poin \Mu_\Nfu(\xi,\nu)$ will be preferred.

\begin{proof}
It suffices to work with $\nu$ on a relatively compact open set in the
half-plane $\re\nu>2$. The uniform convergence of the Poincar\'e series
on compact sets in $(g,\nu)$ implies that holomorphy in $\nu$ is
preserved and that $\Mu_\Nfu(\xi,\nu) \mapsto \M_\Nfu(\xi,\nu)$ is an
intertwining operator
$\Mfu_\Nfu^{\xi,\nu} \rightarrow C^\infty(\Gm\backslash G)_K$. This
gives parts (i) and~(ii). The intertwining property implies that to
have $\M_\Nfu(\xi,\nu) \in \Au!(\ps_{\xi,\nu})$ we need only prove
parts a) and b) in~(iii); see Definition~\ref{def-cu-meg}.

The Fourier terms are obtained by finitely many integrals over compact
sets, hence they depend holomorphically on~$\nu$.

We have $\M_\Nfu(\xi,\nu)= \poin^\infty \Mu_\Nfu(\xi,\nu) 
+ \poin' \Mu_\Nfu(\xi,\nu)$ as in~\eqref{poin-split}. Since
$\Gm_{\!N} \backslash \Gm_{\!P} $ is generated by $\Gm_{\!N} \mm(i)$
for the group that we have fixed in Subsection~\ref{discretesubgroups},
there are four summands that contribute
to~$\poin^\infty \Mu_\Nfu(\xi,\nu)$. In the abelian case, a computation
based on \eqref{Fbtdef} shows that $\poin^\infty \Mu_\bt(\xi,\nu)$ is a
linear combination of the functions $\Mu_{i^a \bt}(\xi,\nu)$, with
$a=0,1,2,3$. These are among the excluded Fourier term orders in (iii)
(a). In the non-abelian case we proceed similarly on the basis of
Proposition~\ref{prop-actmi} in Appendix A.

The component $\poin' \Mu_\Nfu(\xi,\nu)$ is bounded, according to
Lemma~\ref{lem-p'}. The contribution of
$\poin^\infty \Mu_{\Nfu}(\xi,\nu)$ is in general not bounded, but it
contributes to only finitely many Fourier terms, with orders contained
in the sets in iii), a) and~b). This implies that the Poincar\'e series
$\M_\Nfu(\xi,\nu)$ has at most moderate exponential growth.
\end{proof}

In the next result we analyze the growth of the first sum in
\eqref{poin-split}.

\begin{prop}\label{prop-Minf} Let $\Nfu\neq \Nfu_0$. The finite sum
\be \poin^\infty\Mu_\Nfu(\xi,\nu)(g)\= \sum_{\gm\in \Gm_{\!N}\backslash
\Gm_P} \Mu_\Nfu(\xi,\nu)(\gm g)\ee
is holomorphic in $\nu \in \CC\setminus \ZZ_{\leq-1}$, and satisfies
\be \Four_\Nfu \poin^\infty\Mu_\Nfu(\xi,\nu) \= a \,
\Mu_\Nfu(\xi,\nu)\ee
for some $a\in \CC$. The Poincar\'e family
$\nu \mapsto \M_\Nfu(\xi,\nu)$ is identically zero for $\nu\in \CC$ if
and only if $a= 0$.
\end{prop}
\begin{proof}

In the proof of Proposition~\ref{prop-Ps} we saw that the finite sum
$\poin^\infty\Mu_\Nfu(\xi,\nu)$ is a finite linear combination of
$\Mu_{\Nfu'}(\xi,\nu)$. Let $\Dt$ be the subgroup of $\gm\in\Gm_{\!P}$
for which $\Mu_\Nfu(\xi,\nu)(\gm g)$ is a multiple of
$\Mu_\Nfu(\xi,\nu)(g)$. Then
\[ \sum_{\dt\in \Gm_{\!N} \backslash \Dt} \Mu_\Nfu(\xi,\nu)(\dt g)
\= a\, \Mu_\Nfu(\xi,\nu)(g)\]
for some $a\in \CC$. Thus, we have
\[ \poin^\infty \Mu_\Nfu(\xi,\nu)(g)
\= a\sum_{\gm\in \Dt\backslash \Gm_{\!P}} \Mu_\Nfu(\xi,\nu)(\gm g)\,.\]
If $a\neq 0$ then all $\gm\not\in \Dt$ in this sum contribute to a
Fourier term of order $\Nfu'\neq \Nfu$, and
$\Four_\Nfu \poin^\infty \Mu_\Nfu(\xi,\nu)$ is equal to
$a\, \Mu_\Nfu(\xi,\nu)(g)$. If $a=0$, then
\[ \M_\Nfu(\xi,\nu) (g)
\= \sum_{\gm_1\in \Gm_{\!P}\backslash \Gm} \, a\, \sum_{\gm\in
\Dt\backslash \Gm_{\!P}}\, \Mu_\Nfu(\xi,\nu)(\gm\gm_1 g) \=0\,.
\qedhere\]
\end{proof}

\rmrk{Growth}If $\M_\Nfu(\xi,\nu)$ is non-zero, then
$\poin^\infty \Mu_\Nfu(\xi,\nu)$ can have exponential growth. For
instance if $\Mu_\Nfu(\xi,\nu) = \mu_\n^{0,0}(\xi,\nu)$, see
\eqref{om-mu-nab}, this follows from the fact that the $M$-Whittaker
function has in general exponential growth. There are exceptions, since
for special combinations of $\k$ and~$s$, $M_{\k,s}$ and $W_{\k,s}$ are
proportional.

\section{Meromorphic continuation of Poincar\'e series} In this section
we will prove the following result.
\begin{thm}\label{thm-mer-cont}Let $\Nfu\neq \Nfu_0$ be a Fourier term
order, and let $\xi$ be a character of~$M$.

Then, for each family $\nu \mapsto \Mu_\Nfu(\xi,\nu)$ as in
Proposition~\ref{prop-Mudef}, the family defined by
$\nu \mapsto \M_\Nfu(\xi,\nu) = \poin \Mu_\Nfu(\xi,\nu)$ for $\re\nu>2$
has a meromorphic continuation to $\CC$, and
%as a family
$\M_\Nfu(\xi,\nu) \= \poin_\Nfu(\xi,\nu)
\Mu_\Nfu(\xi,\nu) \in \Au!\bigl(\ps[\xi,\nu]\bigr)$, where defined.
\begin{enumerate}
\item[(i)] The singularities of $\nu\mapsto \M_\Nfu(\xi,\nu)
\= \poin_\Nfu(\xi,\nu) \Mu_\Nfu(\xi,\nu)
$ in the closed right half-plane occur at $\nu_0$ for which there exists
a square integrable automorphic form in
$\Au{(2)}\bigl( \ps[\xi,\nu_0]\bigr)$. The singularities at these
   points are simple poles, except for $\nu_0 = 0$, that may be a double
   pole.
\item[(ii)] For each $\XX\in \glie$ the relation
$\XX\Poin_\Nfu(\xi,\nu)  = \poin_\Nfu(\xi,\nu) \, \XX \Mu_\Nfu(\xi,\nu)$
holds as an identity of meromorphic families.
\end{enumerate}
\end{thm}

\rmrk{Remarks}(1)
Here, for $\re\nu>2$, $\poin_\Nfu(\xi,\nu)$ is given by the operator
$\poin$, while for other values of $\nu$ it is obtained by meromorphic
continuation. From now on, we will just write
$\M_\Nfu(\xi,\nu) \= \poin_\Nfu \Mu_\Nfu(\xi,\nu)$ in all cases, for
simplicity.

(2) Part~ii) implies that $\poin_\Nfu(\xi,\nu)$ is an intertwining
operator of $(\glie,K)$-modules for all $\nu\in \CC$ where the
Poincar\'e families have no singularity.

\smallskip
The proof of Theorem~\ref{thm-mer-cont} follows the methods in
\cite[\S2]{MW89}. We shall give it in subsections
\ref{sect-trunc}--\ref{sect-mainpart}. We carry out a \emph{reduction}
before we start on the actual proof. Proposition~\ref{prop-Mudef}
provides us with a family $\Mu_\Nfu(\xi,\nu)$ that may contain many
$K$-types. As noted after Proposition~\ref{prop-MuOm} in the appendix,
such a function is obtained as a holomorphic linear combinations of
families $\Mu_{\Nfu;h,p,q}$ with $K$-type $\tau^h_p$ and weight $q$ in
that $K$-type. Since the Poincar\'e series satisfies
\be \poin \biggl( \sum_j c_j(\nu) \, \Mu_{\Nfu;h_j,p_j,\nu_j}(\xi,\nu)
\biggr)
\= \sum_j c_j(\nu) \poin \Mu_{\Nfu;h_j,p_j,q_j}(\xi,\nu) \,,\ee
  it will thus suffice to prove the theorem for Poincar\'e series of the
  form $\poin \Mu_{\Nfu;h,p,q}$.

\subsection{Truncation}\label{sect-trunc}

We take a function $\ph\in C^\infty(0,\infty)$ that is equal to $1$ on a
neighborhood of $0$, and equal to $0$ on a neighborhood of~$\infty$. We
use the same symbol for the extended function $\ph\in C^\infty(G)$
defined by $\ph\bigl( n \am(t) k \bigr) =  \ph(t)$.

We fix $(h,p,q)$ satisfying $h\equiv q \equiv p\bmod 2$,
$p\in \ZZ_{\geq0}$, and consider a family
$\nu \mapsto \Mu_{\Nfu;h,p,q}(\xi,\nu)$ as in
Proposition~\ref{prop-MuOm}. We allow ourselves to drop the subscript
${}_{h,p,q}$ from the notation of the family, and from related families
occurring later in the proof.

We write $\Mu_\Nfu(\xi,\nu)$ as \be \Mu_\Nfu(\xi,\nu) \= \ph
\Mu_\Nfu(\xi,\nu) + \bigl( 1- \ph\bigr)
\Mu_\Nfu(\xi,\nu)\,.\ee

For $\re\nu>2$ this leads to the decomposition
\bad \M_\Nfu(\xi,\nu) &\= \tilde\M_\Nfu(\xi,\nu) + \Bigl(
\M_\Nfu(\xi,\nu) - \tilde\M_\Nfu(\xi,\nu)\Bigr)\,,\\
\tilde\M_\Nfu(\xi,\nu) &\= \poin \bigl( \ph \,\Mu_\Nfu(\xi,\nu)
\bigr)\,.
\ead
The behavior of
$\bigl( \ph \Mu_\Nfu(\xi,\nu)\bigr) \bigl(n\am(t)k\bigr)$ as
$t\downarrow 0$ is the same as that of $\Mu_\Nfu(\xi,\nu)$. The
contribution of
$\poin' \bigl( \ph \Mu_\Nfu(\xi,\nu)\bigr)\bigl(n \am(t)k\bigr)$ as
$t\uparrow \infty$ is bounded (see Lemma~\ref{lem-p'}). The finite sum
$\poin^\infty\bigl( \ph \Mu_\Nfu(\xi,\nu)\bigr)\bigl(n \am(t)k\bigr)$
is zero for sufficiently large values of~$t$. Hence
$\tilde \M_\Nfu(\xi,\nu)$ is bounded, and
\be \tilde \M_\Nfu(\xi,\nu) \in L^2\bigl( \Gm\backslash G)_K\qquad\text{
for }\re\nu>2\,.\ee

The function
$\bigl( (1-\ph)\Mu_\Nfu(\xi,\nu) \bigr)\bigl( n \am(t) k \bigr)$
vanishes for all $t$ near zero, uniformly in $n$ and~$k$. Hence the sum
$\M_\Nfu(\xi,\nu) - \tilde\M_\Nfu(\xi,\nu)$ is given, locally in $g$,
by a finite sum. So $\M_\Nfu(\xi,\nu) - \tilde\M_\Nfu(\xi,\nu)$ is
meromorphic in $\nu \in \CC$, with at most first order singularities in
$\ZZ_{\leq -1}$. This difference has in general exponential growth as
$t\uparrow\infty$. In this way, the proof of the meromorphic
continuation is reduced to the continuation of the truncated Poincar\'e
series $\tilde\M_\Nfu(\xi,\nu)$.\medskip

For each $\XX\in \glie$
\be\label{XtM} \XX \tilde \M_\Nfu (\xi,\nu)(g) =\! \sum_{\gm\in
\Gm_{\!N}\backslash \Gm} \ph(\gm g)\, \XX \Mu_\Nfu(\xi,\nu)(\gm g) +
\sum_{\gm\in \Gm_{\!N}\backslash \Gm} \XX\ph(\gm g)\,
\Mu_\Nfu(\xi,\nu)(\gm g)\,.\ee
The decay of $\XX \ph$ near zero is better than the decay of $\ph$, so
both terms in \eqref{XtM} yield a bounded function (see also
\cite[Lemma 2.2]{MW89}).
\smallskip

The Casimir operator $C$ acts on $\Mfu_\Nfu^{\xi,\nu}$ by multiplication
by $\ld_2(\xi,\nu) = \nu^2-4+\frac13 j_\xi^2 $ (see (9.3) in [rFtm]), which
we now denote by
$\ld(\xi,\nu)$. 
  The quantity 
\[ \bigr(C-\ld(\xi,\nu)\bigr)\bigl(\ph  \Mu_\Nfu (\xi,\nu)\bigr) (g) \] 
is non-zero only in the region where the cut-off function $\ph$ is not
constant. So, the function
\be\label{chM} \check \M_\Nfu (\xi,\nu)(g) :\= \bigl( C -\nobreak \ld(\xi,\nu)
\bigr)\allowbreak \, \tilde \M_\Nfu (\xi,\nu)(g)\ee
has compact support in $\Gm\backslash G$. The argument in Lemma~2.3 in
\cite{MW89} gives the meromorphic continuation of
$\check \M_\Nfu (\xi,\nu)(g)$ to $\nu \in \CC$ with at most first order
singularities in $\ZZ_{\leq -1}$.

Since $\Mu_\Nfu(\xi,\nu) = \Mu_{\Nfu;h,p,q}$ we have
\bad\tilde\M_\Nfu(\xi,\nu) &\in L^2(\Gm\backslash G)_{h,p,q}&\text{ for
}& \re\nu>2\,,\\
\check \M_\Nfu(\xi,\nu) &\in C_c^\infty(\Gm\backslash G)_{h,p,q}
&& \nu \in \CC\setminus \ZZ_{\leq -1}\,,\\
\M_\Nfu(\xi,\nu)&-\tilde \M_\Nfu(\xi,\nu) \in C^\infty(\Gm\backslash
G)_{h,p,q}&& \nu \in \CC\setminus \ZZ_{\leq -1}\,.
\ead

Here $L^2_{h,p,q}(\Gm\backslash G) \subset L^2(\Gm\backslash G)$ is the
Hilbert space of $\Gm$-invariant square integrable functions of
$K$-type $\tau^h_p$ and weight $q$, with $p\geq 0$,
$h\equiv p \equiv q\bmod 2$, $|q|\leq p$. There is a direct sum
decomposition
\be\label{L2decomp} L^2_{h,p,q}(\Gm\backslash G)\=
L^{2,\discr}_{h,p,q}(\Gm\backslash G)
\oplus L^{2,\cont}_{h,p,q}(\Gm\backslash G)\,,\ee
of subspaces in which the Casimir operator induces a self-adjoint
operator having discrete (resp.~continuous) spectrum on
$L^{2,\discr}_{h,p,q}(\Gm\backslash G)$ \big(resp.~ on
$L^{2,\cont}_{h,p,q}(\Gm\backslash G)$\big). The bounded function
$\tilde\M_\Nfu (\xi, \nu)$ is square integrable, and hence has a
decomposition
\be \label{Mf-dcp}
\tilde \M_\Nfu (\xi, \nu) \= \tilde\M_\Nfu^{\discr}
(\xi,\nu) + \tilde \M_\Nfu^{\cont} (\xi,\nu) \qquad
(\re\nu>2)\,.\ee
We shall consider the meromorphic continuation of these two components
separately, with the help of the analogous decomposition of
$\check \M_\Nfu(\xi,\nu)$, valid at least for
$\nu \in \CC\setminus \ZZ_{\leq -1}$.

\subsection{Contribution of the discrete spectrum} \label{sect-discr}
The space $L^{2,\discr}_{h,p,q}(\Gm\backslash G)$ has an orthonormal
basis $\bigl\{ f_l \bigr\}_{l \in \NN}$ of eigenfunctions of $C$ with
real eigenvalues $\{\mu_l \}$ with finite multiplicity, ordered
increasingly, and chosen so that the $f_l$ are eigenfunctions of
$\Delta_3$ as well.

We define for $\re\nu>2$
\be c_l(\nu) \= \left( \tilde \M_\Nfu (\xi,\nu),\, f_l\right)\quad
\textrm{ and } %\ee\be
d_l(\nu) \= \left( \check \M_\Nfu (\xi,\nu),\, f_l\right)\,. \ee
Since $\check \M_\Nfu (\xi,\nu) $ is compactly supported and smooth, the
series $\sum_j d_l(\nu) f_l$ converges absolutely and uniformly on any
compact subset $\Omega$ of $\CC \smallsetminus \ZZ_{\le -1}.$
Furthermore, the coefficient $d_l$ inherits from
$\check \M_{\Nfu}(\xi,\nu)$ meromorphy in $\CC$ with at most first
order singularities in points of $\ZZ_{\leq -1}$, and we obtain from
  the fact that
$\check \M_\Nfu(\xi,\nu) \= (C-\ld(\xi,\nu) \bigr) \tilde \M_\Nfu(\xi,\nu)$,
the relation
\be d_l(\nu) \= \bigl(\mu_l - \ld(\xi,\nu) \bigr)\, c_l(\nu)
\qquad \textrm{ for } \re\nu>2 \,. \ee

So $c_l$ extends meromorphically to $\CC\setminus\ZZ_{\leq -1}$, and if
$\ld(\xi,\nu) \neq \mu_l$ we have
\be\label{tMdext} \tilde \M_\Nfu^{\discr} (\xi, \nu) \= \sum_l \frac
{d_l(\nu)}{\mu_l-\ld(\xi,\nu)} \,f_l\,.\ee
%Since the set $\bigl\{(\xi_j,\nu_j)\bigr\}$ 
This series converges in the $L^2$-sense, uniformly for $\nu$ on compact
sets $\Omega$ such that $\ld(\xi,\nu)\ne \mu_l$ for all
$\nu \in \Omega$ and all~$l$
(see Theorem~2.5 in \cite[p.~428]{MW89}). In this way,
$\tilde \M_\Nfu^{\discr}  (\xi, \nu)$ has been extended to a
meromorphic family in $\CC$ with values in
$L^{2,\discr}(\Gm\backslash G)$.
%Repeating the differentiation of $\check \M_\Nfu(\xi,\nu)$
%shows that the convergence is much better. 
Furthermore, we can repeatedly apply elements of $\glie$, and get that
$u \tilde\M^\discr_\Nfu(\xi,\nu)$ is square integrable on
$\Gm\backslash G$ for all $u\in U(\glie)$. This implies that we have a
family that is smooth in $(\nu,g)$ jointly, and holomorphic in~$\nu$ as
long as $\ld(\xi,\nu) \neq \mu_l$ for all~$l$. If there is a
singularity at a point $\nu_0$, it is also a singularity of the sum of
the terms in~\eqref{tMdext} for which $\mu_l  =  \ld(\xi, \nu_0 )$.

\subsection {Contribution of the continuous spectrum}\label{sect-cont}
We now consider the meromorphic continuation of
$\tilde \M_\Nfu^{\cont} (\xi, \nu)$. In~\eqref{E} we have introduced
the Eisenstein series in their domain of absolute convergence
$\re\nu > 2$. For $L^{2,\cont}(\Gm\backslash G)_{h,p,q}$ we need
$E(\xi_r,\nu)_{h,p,q}$ for the finitely many $r \in \ZZ$ such that
$r\equiv p \bmod 2$, $|r|\leq p$, and $j_r=\frac12(h-3r) \in 4\ZZ$.

The Eisenstein series have a meromorphic continuation to $\CC$, and the
possible singularities in the closed half-plane $\re\nu\geq 0$ lie at a
finite set of points in the interval $(0,2]$. The residue at
$\nu_1 \in (0,2]$ is a function in
$\Au{(2)}\bigl( \ps[\xi_r,\nu_1]\bigr)$, and gives a contribution to
the discrete spectrum.

The projection onto the continuous spectrum of a function
$f \in L^2(\Gm\backslash G)_{h,p,q}$ takes the form $ f_\cont(g) 
= \sum_{r} \frac{m_{h,p,q}}{2\pi i} \int_{\re \ld=0} f_{j_r}(\ld)\,
E(\xi_r,\ld)_{h,p,q} (g)\, d\ld\,,$ where
\be\label{Eisint} f_{j_r}(\ld) =\int_{\Gm\backslash G}f(g)
E(\xi_r,\ld)_{h,p,q} (g)\, dg \,,\ee
for $\re (\ld)=0$.

For $\varepsilon$ sufficiently small, if $0 \le \re \ld < \varepsilon $,
the functions $E(\xi_r,\ld)=  E(\xi_r,\ld)_{h,p,q}$ do not have poles
and are integrable on $\Gm\backslash G$, and hence
\badl{phrdef} \varphi_{j_r}(\ld,\nu) &\= \int_{\Gm\backslash G}
\tilde\M_\Nfu \,
(\xi, \nu)(g)\, \overline{E(\xi_r,\bar \ld)(g)} \, dg
\\
&\= \int_{\Gm_{\!N} \backslash G} \ph(g) \, \Mu_\Nfu (\xi, \nu)(g)\,
\overline{E(\xi_r,\bar \ld)(g)}\, dg\,.
\eadl
is a holomorphic function of $(\ld,\nu)$ for
$0\le \re\ld < \varepsilon $ and $\re \nu >2$.

\begin{lem}\label{lem-mercph}The family $(\ld,\nu) \mapsto 
\ph_{j_r}(\ld,\nu)$ extends as a meromorphic function of $(\ld,\nu) $
on $\CC^2$ with singularities occurring only in the union of the
following sets:
\begin{align*} &\CC\times \left( \CC\setminus \ZZ_{\leq -1}\right),
\qquad \Bigl\{ (\ld,\nu) \;:\; \ld+\nu \text{ or } \ld-\nu \in \ZZ_{\leq
0} \Bigr\}\,,\\
&\Bigl\{ (\ld,\nu)\;:\; \nu\in \CC\,,\;E(\xi_r,\cdot)\text{ has a
singularity at }\bar \ld\Bigr\}\,.
\end{align*}
\end{lem}
This is a preliminary continuation result, comparable to \cite[Lemma
2.4]{MW89}.

\begin{proof}The Eisenstein series $E(\xi_r,\ld)_{h,p,q}$ has a Fourier
expansion in which the term of order $\Nfu$ has the form
$C_\Nfu(\xi_r,\ld) \, \Om_{\Nfu;h,p,q}(\xi_r,\ld)$, with the
exponentially decreasing basis function given in
Proposition~\ref{prop-MuOm}. The coefficient $C_\Nfu(\xi_r,\ld)$ is
holomorphic where $E(\xi_r,\ld)$ is holomorphic.

We start with $(\ld,\nu)$ for which \eqref{phrdef} holds. The second
integral in \eqref{phrdef} is, up to a multiple, equal to
\be \label{CE} \overline{ C_\Nfu(\xi_r,\bar\ld) }
\int_{\Gm_{\!N}\backslash G} \ph (g) \, \Mu_{\Nfu;h,p,r} (\xi,\nu)(g)
\overline{\Om_{\Nfu;h,p,q} (\xi_r,\bar \ld)}\, dg\,.\ee
The influence of the singularities of the Eisenstein series is given by
the Fourier coefficient $C_\Nfu(\xi_r,\cdot)$.
The integral in~\eqref{CE} is a quantity depending on the basis
functions $\Mu_\Nfu$ and $\Om_\Nfu$ and the cut-off function $\ph$.
Only the Fourier coefficient depends on the group~$\Gm$.

For general $K$-types $\tau^h_p$ we do not have an explicit description
of the function $\Mu_{\Nfu;h,p,q}(\xi,\nu)$. However we know that for
$n\am(t)k\in NAK$
\be \Mu_{\Nfu;h,p,q}(\xi,\nu)\bigl( n am(t) k \bigr)
\= \sum_r w_r(n)\, t^{2+\nu}\, h^\mu_r(\xi,\nu,t)\, \Kph h p r q(k)\,,
\ee
with $r$ satisfying the conditions $|r|\leq p$, $r\equiv p\bmod 2$, and
\be h^\mu_r( \xi,\nu, t) \= \sum_{m=0}^\infty \al_m(\xi,\nu) t^m\ee
holomorphic in
$(\nu,t)\in \left( \CC\setminus \ZZ_{\leq-1}\right)\times \CC$, with
possible singularities at points $\nu  =\nu_0$ with $
\nu_0 \in \ZZ_{\leq -1}$. The functions $w_r$ are a character of $N$ or
elements of an orthogonal system of theta functions on~$N$. The
$\Kph h p r q$ belong to an orthogonal system of polynomial functions
on $K$. See Proposition~\ref{prop-MuOm}, and the discussion of
$\nu$-regular behavior at $0$ in \cite[\S10.2]{BMSU21}.

The function $\Om_{\Nfu;h,p,q}(\xi_r,\ld)\bigl( n \am(t)\bigr)$ does not
have this form as $t\downarrow 0$. However, for $\ld\not\in \ZZ$ we can
write it as a linear combination of $\Mu_{\Nfu;h,p,q}(\xi,\nu)$ and
$\Mu_{\Nfu;h,p,q}(\xi,-\nu)$ with coefficients that are holomorphic in
$\CC\setminus \ZZ$. See \cite[(10.21)]{BMSU21}. In this way we find
that the integral in \eqref{CE} has the form
\be \label{intser} \int_{t=0}^\infty \sum_\pm A_\pm(\ld)\, \ph(t)
t^{\nu\pm \ld} \, \sum_{m,l\geq 0} \al_m(\xi,\nu) \,
\overline{\al_l(\xi_r,\bar\ld)}\, t^{m+l-1}\, dt\,,\ee
where $A_+$ and $A_-$ are meromorphic functions which may have
singularities only in~$\ZZ$. The series in $m$ and $l$ converge
absolutely on the compact support of~$\ph$.

For $\re\nu - |\re\ld|>0$ the integral converges and defines a
holomorphic function of $(\ld,\nu)$. We apply partial integration to
the integral of individual terms to see that they have a meromorphic
extension with singularities along lines
$\nu \pm \ld \in \ZZ_{\leq 0}$. Splitting off more and more terms, we
get holomorphy of the remaining contributions on larger sets. Thus we
have a meromorphic extension which is holomorphic outside the sets
indicated in the lemma.

If $\ld_0 \in \ZZ$, then there will be coinciding terms in
\eqref{intser}. The singularities of $A_\pm(\ld)$ at $\ld_0$ are
canceled, and the integrand in \eqref{intser} extends holomorphically
to~$\ld_0$. This reflects the fact that the family
$\nu\mapsto \Om_{\Nfu;h,p,q}(\xi,\nu)$ is holomorphic in~$\CC$. There
arises a series with terms $t^{\nu \pm \ld+N-1}$ and
$t^{\nu\pm \ld+N-1}\,\log t$, with $N\geq 0$. Partial integration now
leads to singularities that may have second order.

In this way we arrive at a meromorphic extension of $\ph_{j_r}$ with
singularities along lines $\ld-\nu = n$ and $\ld+\nu=n$, with
$n\in \ZZ_{\leq 0}$. The possibility that $\Mu_{\Nfu;h,p,q}(\xi,\nu)$
has singularities at $\nu\in \ZZ_{\leq -1}$ may cause singularities of
$\ph_{j_r}$. The Fourier coefficient $C_\Nfu(\xi_r,\bar \ld)$ may also
cause singularities.
\end{proof}

The function $\check \Mu_\Nfu(\xi,\nu)$ has a compact support determined
by the cut-off function~$\ph$. It depends meromorphically on $\nu$ with
singularities only in $\ZZ_{\leq -1}$. Hence
\be \check \varphi_{j_r}(\ld,\nu) \= \int_{\Gm\backslash G} \check
\M_\Nfu \, (\xi, \nu)(g)\, \overline{E(\xi_r,\bar\ld)(g)} \, dg \ee
is meromorphic, with singularities along $\nu =n$, $n\in \ZZ_{\leq -1}$
and along $\ld=\overline{\ld_0}$ for points $\ld_0$ at which the
Eisenstein series has a singularity.

The functions $\tilde \M_\Nfu \, (\xi, \nu)(g)$ and
$\check \M_\Nfu \, (\xi, \nu)(g)$ and their derivatives by
$\XX \in \mathfrak g$ are bounded. Differentiation under the integral
sign gives
\begin{align}\nonumber
\check\varphi_{j_r}(\ld,\nu) & \=\Bigl( \tilde \M_\Nfu \, (\xi, \nu),
\bigl( C - \overline{\ld_2(\xi,\nu)}\bigr) E(\xi_r,\bar\ld) \Bigr)
\displaybreak[0]\\
\nonumber
&\= \bigl(\ld_2(\xi_r,\ld)- \ld_2(\xi,\nu) \bigr) \, \bigl( \tilde
\M_\Nfu(\xi,\nu), E(\xi_r,\bar\ld) \bigr)
\displaybreak[0]\\
\label{chphd}
&\= \Bigl( \ld^2-\nu^2+ \frac13\bigl( (j_r)^2-j^2\bigr) \Bigr) \,
\ph_{j_r}(\ld,\nu)\,,
\end{align}
valid for $\ld\in \Om$ and $\re \nu >2$, and extended as a relation
between meromorphic functions on $\CC^2$, with use of
Lemma~\ref{lem-mercph}. This relation describes the function
$\ph_{r_j}(\ld,\nu)$ in terms of the better behaved function
$\check \ph_{r_j}(\ld,\nu)$. The singularities of $\ph_{j_r}$ come
either from singularities of $\check \ph_{j_r}$, or from the polynomial
factor in~\eqref{chphd}. By repeating the differentiation in
\eqref{chphd} we obtain quick decay of $\check \ph_r(\nu, \ld)$ for
$\nu$ on vertical strips and for $\re\ld$ in bounded sets. This implies
the same result for $\ph_{j_r}(\ld,\nu)$ outside neighborhoods of its
singularities.

We use the meromorphy of $\check \ph_{j_r}$ and $\ph_{j_r}$ to get the
meromorphic continuation of the projections of
$\check \M_\Nfu(\xi,\nu)$ and $\M_\Nfu(\xi,\nu)$ to
$L^{2,\cont}(\Gm\backslash G)$.
\begin{align*}
\check \M_\Nfu^{\cont} \, (\xi, \nu)(g) &\= \sum_{r}
\frac{m_{h,p,q}}{2\pi i} \int_{\re \ld=0} \check
\varphi_{j_r}(\ld,\nu)\, E(\xi_r,\ld)_{h,p,q} (g)\, d\ld\,,
\displaybreak[0]\\
\tilde \M_\Nfu^{\cont} \, (\xi, \nu)(g) &\= \sum_{r}
\frac{m_{h,p,q}}{2\pi i} \int_{\re \ld=0} \varphi_{j_r}(\ld,\nu)\,
E(\xi_r,\ld)_{h,p,q} (g)\, d\ld
\displaybreak[0]\\
&\= \sum_{r} \frac{m_{h,p,q}}{2\pi i} \int_{\re \ld=0} \frac{\check
\ph_{j_r}(\ld,\nu)}{(-\ld)^2 -\nu^2+\frac{(j_r)^2-j^2}3}\,
E(\xi_r,\ld)_{h,p,q} (g)\, d\ld\,.
\end{align*}
Initially, these relations hold for $\re\nu>2$. The aim is to give a
meromorphic extension. In the proof of Theorem~2.5 in~\cite{MW89} an
analogous situation arose; see p~428--429.

In \cite{MW89}, the complex plane is divided into connected regions on
which the integrals above describe a holomorphic function of~$\nu$.
Deformation of the path of integration causes changes of these regions.
The resulting functions of $\nu$ are related by Cauchy's residue
theorem. Suitable choices of a complex square root of
$\nu^2 + \frac 13\bigl( j^2-{j_r}^2 \bigr)$ lead to the desired
meromorphic continuation. Below, we work out the proof of the
meromorphic continuation of $\tilde \M_\Nfu^\cont(\xi,\nu)$  in an alternative way, in the
present context and notation.

\rmrk{The general integral} We consider functions $I^+_r(\Ld,\nu)$ on
the region $\bigl\{ (\Ld,\nu)\in \CC^2\;:\; \re\Ld >0 \bigr\}$, and
$I^-_r(\Ld,\nu)$ on
$\bigl\{ (\Ld,\nu) \in \CC^2\;:\; \re \Ld <0\bigr\}$, both given by the
integral
\be\label{i0} \frac 1{2 \pi i} \int_{\re \ld = 0} \frac {\check
\varphi_{j_r}(\ld,\nu)}{\ld^2 -\Ld^2} E (j_r,\ld)_{h,p,q}(g)\,
d\ld\,.\ee
If $\Ld^2=\nu^2+\frac13(j^2-j_r^2)$, then $I^+_r(\Ld,\nu)$ corresponds
to the integral in the term of order $r$ in the expression for
$\tilde\M_\Nfu^\cont(\xi,\nu)(g)$. We first consider $\Ld$ and $\nu$ as
two independent variables. We drop $g$ from the notation in the further
computations.

The functions $I^\pm_r$ are meromorphic, with singularities at most
along (complex) lines of the form
$\bigl\{ (\Ld,n)\;:\ \pm \re\Ld>0,\; n\in \ZZ_{\leq -1}\bigr\}$. We
have the relation $I^-_r(\Ld,\nu) = I^+_r(-\Ld,\nu)$. The singularities
of the integrand in \eqref{i0} as a function of $\ld$ do not occur on
the path of integration.

We take $a>0$, and consider two deformations of the path of integration,
as depicted in Figure~\ref{fig-Ce}.
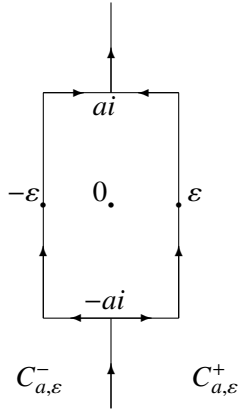
\begin{figure}[ht]
\[ \setlength\unitlength{.6cm} % this scale can be adapted
\begin{picture}(6,9)(-3,-4.5)
\put(0,-4.5){\line(0,1){2}}
\put(0,-4.5){\vector(0,1){1}}
\put(0,-2.5){\line(1,0){1.5}}
\put(0,-2.5){\vector(1,0){.9}}
\put(0,-2.5){\line(-1,0){1.5}}
\put(0,-2.5){\vector(-1,0){.9}}
\put(1.5,-2.5){\line(0,1){5}}
\put(1.5,-2.5){\vector(0,1){1.6}}
\put(-1.5,-2.5){\line(0,1){5}}
\put(-1.5,-2.5){\vector(0,1){1.6}}
\put(1.5,2.5){\line(-1,0){1.5}}
\put(1.5,2.5){\vector(-1,0){.9}}
\put(-1.5,2.5){\line(1,0){1.5}}
\put(-1.5,2.5){\vector(1,0){.9}}
\put(0,2.5){\line(0,1){2}}
\put(0,2.5){\vector(0,1){1}}
\put(0,0){\circle*{.1}}
\put(1.5,0){\circle*{.1}}
\put(-1.5,0){\circle*{.1}}
\put(-.4,.1){$0$}
\put(-.4,2.0){$ai$}
\put(-.6,-2.3){$-ai$}
\put(1.7,.1){$\e$}
\put(-2.3,.1){$-\e$}
\put(1.8,-4){$C^+_{a,\e}$}
\put(-2.1,-4){$C^-_{a,\e}$}
\end{picture}
\]
\caption{\it Deformed contours in the $\ld$-plane, inside the
neighborhood $\Om$ of the imaginary axis.} \label{fig-Ce}
\end{figure}
We take $\e>0$ sufficiently small such that both deformed paths are
contained in~$\Om$. Denote by $R_{a,\e}$ the open rectangular region
bounded by the joined paths $C^\pm_{a,\e}$.

Deforming the path $\re\ld=0$ to $C_{a,\e}^-$ we get an extension of
$I^+_r(\Ld,\nu)$ with $\Ld$ in the region on the right of $C^-_{a,\e}$.
Similarly, we have an extension of $I^-_r(\Ld,\nu)$ with $\Ld$ in the
region on the left of $C^+_{a,\e}$.

For $\Ld\neq 0 \in R_{a,\e}$, the difference of these two extensions is
given by
\begin{align} \nonumber
I^-_r(\Ld,\nu) &- I^+_r(\Ld,\nu) \= \frac 1{2\pi i} \biggl(
\int_{C_{a,\e}^+ }
-\int_{C_{a,\e}^-}\biggr)
\frac{\check \ph_{j_r}(\ld,\nu)}{\ld^2-\Ld^2} E(j_r,\ld) \, d\ld
\displaybreak[0]\\
\nonumber
&\= \frac1{2\pi i} \int_{\partial R_{a,\e}} \frac{\check
\ph_{j_r}(\ld,\nu)}{\ld^2-\Ld^2} E(j_r,\ld) \, d\ld
\displaybreak[0]\\
&\= \sum_\pm \mathop{\mathrm{Res}}_{\ld=\pm \Ld} \frac{\check
\ph_{j_r}(\ld,\nu)}{\ld^2-\Ld^2} E(j_r,\ld) = \sum_\pm \frac{\pm
1}{2\Ld} \, \check \ph_{j_r}(\pm \Ld,\nu)\, E(j_r,\pm \Ld).
\label{p+p-}
\end{align}
The right hand side is meromorphic on $\CC^2$. The singularities are
along lines $\CC \times \{n\}$ with $n\in \ZZ_{\leq -1}$, and along
lines $\{\ld_0\}\times \CC$ with $\pm \ld_0$ or $\pm \overline{\ld_0}$
a singularity of the Eisenstein series. Such points $\ld_0$ do not
occur on the paths $C^\pm_{a,\e}$ or in the rectangle $R_{a,\e}$. In
the sum $\sum_\pm$ the singularities at $\Ld=0$ cancel each other. The
formula
\be I^+_r(\Ld,\nu) \= I^-_r(\Ld,\nu) + \sum_\pm \frac{\mp 1}{2\Ld} \,
\check \ph_{j_r}(\pm \Ld,\nu)\, E(j_r,\pm \Ld)\,,\ee
valid at first for $\Ld\in R_{a,\e}$, gives the meromorphic extension
for $(\Ld, \nu)$ on $\CC^2$ minus the subsets $i[a,\infty) \times \CC$
and $-i[a,\infty)\times \CC$. Since $a>0$ can be chosen arbitrarily,
the function $(\Ld,\nu)\mapsto I^+_r(\Ld,\nu)$ is meromorphic on
$\CC^2$. On the region $\re\Ld\geq 0$ it can have singularities only
along lines $\nu=n$, $n\in \ZZ_{\leq -1}$. For $\re\Ld<0$ there may be
singularities along lines $\{\ld_0\}\times \CC $ as well.

\rmrk{Meromorphic continuation for $|j_r|= |j|$}To obtain the
meromorphic continuation as a function of~$\nu$ of the contributions to
$\tilde\M_\Nfu^\cont(\xi,\nu)$, we restrict the functions
$I_r^+( \Lambda, \nu)$ to the curve of equation $\Ld^2 = \nu^2+\dt$,
with $\dt=\frac13\left( j^2-j_r^2\right)$. In the case $|j_r|=|j|$ we
can do that by taking $\Ld=\nu$. The singular set of $I^+_r(\Ld,\nu)$
is contained in a union of lines $\CC\times\{n\}$,
$n\in \ZZ_{\leq -1}$, and lines $\{\ld_0\}\times \CC$ with $\ld_0$
determined by the singularities of the Eisenstein series. This singular
set intersects the line $\Ld=\nu$ discretely. Hence the restriction
$\nu \mapsto I^+_r(\nu,\nu)$ gives a meromorphic function of $\nu$ with
singularities at values related to singularities of Eisenstein series
and at negative integers. The extension satisfies the following
identity of meromorphic functions on~$\CC$.
\be\label{fe0}
I_r^+(\nu,\nu) \= I^+_r(-\nu,\nu)+ \sum_\pm \frac{\mp 1}{2\nu} \, \check
\ph_{j_r}(\pm \nu,\nu) \, E(j_r,\pm \nu)\,. \ee We note that the
singularities of $\check \ph_{j_r}(\Ld,\nu)$ occur along lines of the
form $\CC\times\{\ld_0\}$ and $\{n\}\times \CC$, although
$\ph_{j_r}(\Ld,\nu) $ can have singularities along the line
$\pm \Ld=\nu$. See Lemma~\ref{lem-mercph}.

\rmrk{Meromorphic continuation for $|j_r|\neq |j|$} Now we are concerned
with the restriction of $I^+_r(\Ld,\nu)$ to a quadratic curve $C_r$
with equation $\Ld^2 = \nu^2+\dt$ with $\dt\neq0$.

The meromorphic function $(\Ld,\nu)\mapsto 
\frac12 I^+(\Ld,\nu) + \frac12 I^-(\Ld,\nu)$ on $\CC^2$ is even in
$\Ld$. Hence it is of the form
$(\Ld,\nu) \mapsto \tilde I_r(\Ld^2,\nu)$ for some meromorphic function
$\tilde I_r$ on~$\CC^2$. This function inherits from $I^+_r$ and
$I^-_r$ the property that its singularities occur only along lines
$\CC\times \{n\}$ and $\{\ld_0\}\times \CC$. Hence $(\Ld,\nu) 
\mapsto \tilde I_r(\Ld^2,\nu)$ has a meromorphic restriction to $C_r$.
Relation \eqref{p+p-} implies the following
\[ \tilde I_r(\Ld^2,\nu) \= I^+_r(\Ld,\nu) + \sum_\pm \frac{\pm 1}{2\nu}
\, \check \ph_{j_r}(\pm \Ld,\nu) \, E(j_r,\pm \Ld)\,.\]

The set of singularities of $\ph_{j_r}$ is a union of lines described in
Lemma~\ref{lem-mercph}. This set intersects $C_r$ in isolated points.
So the restriction of $\ph_{j_r}$ to $C_r$ is a meromorphic function,
and by \eqref{chphd}, the restriction of
\[ \check \ph_{j_r}(\Ld,\nu) \= ( \Ld^2-\nu^2-\dt)\,
\ph_{j_r}(\Ld,\nu)\]
to $C_r$ is zero. Hence we get the equality
$\tilde I_r(\Ld^2,\nu) = I^+_r(\Ld,\nu)$ on~$C_r$. The meromorphic
function $\nu \mapsto \tilde I_r\bigl( \nu^2+\dt,\nu\bigr)$ gives the
desired meromorphic continuation.

\rmrk{Regularity} In each of the three cases, the meromorphically
continued quantity is given by integrals and terms with a factor that
is quickly decreasing for $\ld$ on vertical strips. So all
contributions to the meromorphic continuation of
$\tilde\Mu^\cont_\Nfu(\xi,\nu)(g)$ are locally $C^\infty$ in $(\nu,g)$
and holomorphic in $\nu$ at all points $(\nu,g)$ with $\re \nu \ge 0$.

\rmrk {Remark} We note that identity \eqref{p+p-} shows that in the cases
$|j_r|=|j|$ there is no pole of the continuation of $I^0_+ (\nu)$
at $\nu =0$, and hence $\M_\Nfu^\cont(\xi,\nu)$ is holomorphic at
$\nu =0$. Actually, in the same way one can see  that
$\M_\Nfu^\cont(\xi,\nu)$ is holomorphic at $\nu =0$ for the abelian
Poincar\'e series studied in \cite{MW89} in the case of arbitrary
semisimple Lie groups of real rank one.

\subsection {Conclusion} \label{sect-conclusion}
The considerations in the previous subsections show that the meromorphic
continuation can be given in the form
\be \M_\Nfu(\xi,\nu) \= \tilde\M_\Nfu^\discr(\xi,\nu)+
\tilde\M_\Nfu^\cont(\xi,\nu)
+ \poin (1-\ph) \Mu_\Nfu(\xi,\nu)\,.\ee
For open sets $\Om \subset \CC$ in which no singularities occur,
$\M_\Nfu(\xi,\nu) \in C^\infty( \Om \times G)$, and it depends
holomorphically on $\nu \in \Om$. At a singularity at $\nu=\nu_0$ with
$\re\nu_0\geq 0$ the function $\M_\Nfu(\xi,\nu)$ is the sum of an
explicit finite sum, describing the singularity, plus an element of
$C^\infty(\Om\times G)$ that is holomorphic in $\nu$ on a neighborhood
$\Om$ of $\nu_0$. The order of singularities in the closed right
half-plane is $1$, or possibly $2$ at $\nu=0$. The Laurent series at a
singularity at $\nu=\nu_0$ with $\re\nu_0\geq 0$, starts with a term
given by a square integrable automorphic form for which the eigenvalue
of the Casimir operator is equal to~$\ld_2(\xi,\nu_0)$.

It might happen that the sum
$\sum_{\gm\in \Gm_{\!N} \backslash \Gm_P} \Mu_\Nfu(\xi,\nu;\gm g)$ is
identically zero. In this case, $\poin \Mu_\Nfu(\xi,\nu)$ is zero, and
Theorem~\ref{thm-mer-cont} is trivially true for $\Poin_\Nfu(\xi,\nu)$.

\subsection {Main part and order} \label{sect-mainpart}
To be able to complete the proof of Theorem~\ref{thm-mer-cont}, we now
introduce a useful concept.
\begin{defn}\label{mainpart} Assume the family
$\nu\mapsto\poin \Mu_\Nfu(\xi,\nu)$ is not identically zero. We let the
{\it order} at the point $\nu_0$ of the Poin\-ca\-r\'e family, denoted
$\po_\Nfu(\nu_0)=\po\bigl( \M_\Nfu;\nu_0)$, be the integer such that
the limit
\be \label{def:mainpart} P_{\nu_0}\bigl(\M_\Nfu\bigr)=
\lim_{\nu\rightarrow\nu_0} (\nu-\nu_0)^{\po_\Nfu(\nu_0)} \,
\M_\Nfu(\xi,\nu)\ee
exists and is non-zero. We call the limit function, 
$P_\Nfu(\nu_0)$, the \emph{main part} of $\M_\Nfu(\xi,\cdot)$ at
$\nu_0$. If the Poincar\'e family is identically zero we put
$\po_\Nfu(\nu_0)=-\infty$ and $P_\Nfu(\nu_0)=0$ for all $\nu_0\in \CC$.
\end{defn}

Thus, if the order is a positive integer, then $\poin \Mu_\Nfu(\xi,\nu)$
has a singularity at $\nu=\nu_0$ and if it equals $1$, the main part is
the residue at $\nu_0$. If the order is non-positive, the main part at
$\nu_0$ is the value $\M_\Nfu(\xi,\nu_0)$ (for $m=0$) or the value of a
derivative at $\nu_0$ (for $m\geq 1$).

\begin{prop}\label{prop-mpPf}
Let $\re\nu\geq 0$.
\begin{enumerate}
\item[i)] The main part $P_{\nu_0}\bigl(\M_\Nfu\bigr)$ is an automorphic
form with moderate exponential growth lying in
$\Au!\bigl(\ps[\xi,\nu_0]\bigr)$.
\item[ii)] If $\M_\Nfu(\xi,\nu)$ has a singularity at $\nu=\nu_0$, then
the main part is a square integrable automorphic form in
$\Au{(2)}\bigl( \ps[\xi,\nu_0]\bigr)$.
\end{enumerate}
\end{prop}
\begin{proof}On the region $\re\nu>0$ we have
$u \M_\Nfu(\xi,\nu) = \ps[\xi,\nu](u) \, \M_\Nfu(\xi,\nu)$ for all
$u \in ZU(\glie)$. The pointwise holomorphy implies that this relation
persists in the complement of all singularities. The same reasoning
applies to the function
$\nu \mapsto (\nu-\nu_0)^{\po_\Nfu(\nu_0)} \, \M_\Nfu(\xi,\nu)$ on a
neighborhood of $\nu_0$, and extends to the main part by holomorphy.

The Fourier terms $\Four_{\Nfu'}\M_\Nfu(\xi,\nu)$ depend holomorphically
on~$\nu$, since they are defined by integration over compact sets. For
almost all Fourier term orders $\Nfu'$ we have for $\re\nu>2$
\[ \Four_{\Nfu'}\M_\Nfu(\xi,\nu) \in \Wfu_{\Nfu';h,p,q}^{\xi,\nu}\,.\]
Since $\nu \mapsto \Wfu_{\Nfu';h,p,q}^{\xi,\nu}$ is a holomorphic family
of one-dimensional subspaces of
$C^\infty(\Gm_{\!N}\backslash G)_{h,p,q}$ this relation is preserved by
holomorphy. Thus, we obtain part~i).

If $\M_\Nfu(\xi,\nu)$ has a singularity at~$\nu=\nu_0$, then
by~\eqref{tMdext}, its main part % of $\M_\Nfu$ 
at $\nu_0$ is a linear combination of basis functions
$f_l \in L^{2,\discr}(\Gm\backslash G)_{h,p,q}$ which lie in
$\Au{(2)}\bigl( \ps_l \bigr)$ for a character of $ZU(\glie)$ that
satisfies $\ps_l(C) = \ld_2(\xi,\nu_0)$ and
$\ps_l(\Dt_3)= \ld_3(\xi,\nu_0)$. By i), the only $f_l$ that can occur
satisfy $\ps_l = \ps[\xi,\nu_0]$. This gives ~ii).
\end{proof}

Proposition~\ref{prop-mpPf} implies the assertion in
Theorem~\ref{thm-mer-cont} that $\M_\Nfu(\xi,\nu)$ is a family of
automorphic forms with moderate exponential growth. It also proves
part~(i) of the theorem, and gives the additional information that the
main part is in $\Au{(2)}\bigl( \ps[\xi,\nu_0]\bigr)$.

Having shown that the family is locally in $C^\infty(\Om\times G)$, it
follows that the intertwining property in part~ii) of the theorem stays
valid under meromorphic continuation. This completes the proof of
Theorem~\ref{thm-mer-cont}.

\section{Statements of completeness results}\label{sect-stcr}
For cuspidal Maass wave forms on $\Gm\backslash\SL_2(\RR)$ all square
integrable automorphic forms arise from Poincar\'e series. The spaces
of holomorphic cusp forms are spanned by absolutely convergent
Poincar\'e series. Spaces of real analytic Maass forms are obtained by
  resolving meromorphically continued Poincar\'e series. Theorem~3.2
in~\cite{MW89} gives a generalization of the latter results to Lie
groups of real rank one under restrictions on the isomorphism class of
the $(\glie,K)$-module generated by the square integrable automorphic
form.

For $\SU(2,1)$ not all square integrable automorphic forms are values or
residues of Poincar\'e series. However, we will see that in all cases
when this does not hold, a non-zero square integrable automorphic form
can be detected by using the family of scalar products
$\nu \mapsto \bigl( f, \tilde\M_\Nfu(\xi,\nu) \bigr)_{\Gm\backslash G}$
with a suitable family of truncated Poincar\'e series. As a preparation
for a more precise formulation, we  start with a discussion  on $(\glie,K)$-modules
generated by square integrable automorphic forms.

\subsection{Preparation}\label{sect-mp}
As in Definition~\ref{mainpart}, given a \emph{non-zero} meromorphic
family $\nu \mapsto F(\nu)$ of functions in $C^\infty(G)$ and
$\nu_0\in \CC$, the \emph{main part} of $F$ at $\nu_0$ is defined as
the non-zero function
\be \lim_{\nu\rightarrow\nu_0} (\nu-\nu_0)^m \, F(\nu) \in
C^\infty(G)\,,\ee
where $m$ is the \emph{order} of $F$ at $\nu_0$, i.e. the unique integer
such that the limit exists and is non-zero. For the family that is
identically zero we define the order to be $-\infty$ and the main part
to be zero.

We will work with Poincar\'e families for Fourier term orders
$\Nfu \neq \Nfu_0$. We allow $\nu\mapsto\poin^\infty \Mu_\Nfu(\xi,\nu)$
(and hence $\nu\mapsto\poin \Mu_\Nfu(\xi,\nu)$)
to be identically zero.

As before, for $\nu_0$ with $\re\nu_0\geq 0$, we denote the main part at
$\nu_0$ of Poincar\'e families $\M_\Nfu(\xi,\nu)$ by
$P_\Nfu(\xi,\nu_0)$. For truncated Poincar\'e families
$\tilde\M_\Nfu(\xi,\nu)$, it will be denoted by
$\tilde P_\Nfu(\xi,\nu_0)$. We denote by $\po_\Nfu(\xi,\nu_0)$
the order of $\nu \mapsto \M_\Nfu(\xi,\nu)$ at $\nu_0$. If
$\po_\Nfu(\nu_0)\leq 0$, the order of the truncated family may depend
on the choice of the cut-off function.

We note that for singularities at $\nu=\nu_0$ we have
$P_\Nfu(\xi,\nu_0) = \tilde P_\Nfu(\xi,\nu_0)$, independently of the
cut-off function used for truncation. This follows from the fact that
$\M_\Nfu(\xi,\nu)-\tilde \M_\Nfu(\xi,\nu)$ is holomorphic at
$\nu \in \CC\setminus Z_{\leq -1}$.

\subsubsection{Modules of square integrable automorphic forms}Let $\ps$
be a character of $ZU(\glie)$. There are only finitely many isomorphism
classes of $(\glie,K)$-modules in which $ZU(\glie)$ acts according to a
given character. The $K$-types in these different isomorphism
classes are disjoint (see for instance \cite[Proposition 12.2]{BMSU21}).
Hence $\Au{(2)}(\ps)$ is the direct sum of finitely many irreducible
$(\glie,K)$-modules, and for a given $K$-type $\tau^h_p$ the space
$\Au{(2)}(\ps)_{h,p}$ generates a finite sum of irreducible
$(\glie,K)$-modules with isomorphism class determined by $\ps$
and~$\tau^h_p$.

We leave the one-dimensional $(\glie,K)$-module out of our
consideration. It consists %(for suitable $\ps$) 
of constant functions, which occur as residues of Eisenstein series. We
proceed with an infinite-dimensional irreducible $(\glie,K)$-module $V$
  contained in $\Au{(2)}(\ps)$. Since it does not consist of constant
functions, some of the Fourier term operators
\be \Four_\Nfu : V \rightarrow \Wfu_\Nfu (\ps)\ee
with $\Nfu\neq \Nfu_0$ must be non-zero.

We will need a number of quantities, determined by the isomorphism class
of~$V$:
\begin{itemize}
\item A specific choice of spectral parameters for the character $\ps$
of $ZU(\glie)$: a character $\xi$ of $M$ and a complex number $\nu_0$
in the closed right half plane.

\item The minimal $K$-type $\tau^{h_0}_{p_0}$ of~$V$.
\end{itemize}
Table~\ref{tab-Vlist} gives a list of relevant isomorphism classes and
the corresponding values of the $4$-tuple $(\xi,\nu_0,h_0,p_0)$.
\begin{table}[htp]
{\small
\[
\begin{array}{|ccc|}\hline
&\text{type}&\text{ parameters}\\ \hline\hline
\multicolumn{3}{|c|}{\text{irreducible principal series}}
\\ \hline
\text{unitary principal series}& \II(j,\nu_0) & \nu_0 \in i[0,\infty)
\\
&&\text{if }\nu_0=0,\text{ then }j\in \{0\}\cup 1+2\ZZ
\\
&& h_0= 2j\,,\; p_0=0\\\hline
\text{complementary \ series}&\II(j,\nu_0) & 0<\nu_0<2,\; j=0\\
&&\text{ or }0<\nu_0<1,\; j\in 1+2\ZZ
\\
&& h_0= 2j\,,\; p_0=0\\ \hline \hline
\multicolumn{3}{|c|}{\text{discrete series type}} \\ \hline
\text{large d.s.t.}&\II_+(j,\nu_0)& \nu_0\equiv j\bmod 2,\; \nu_0 \geq
|j|, \; \nu_0\geq 1 \\
&& h_0= -j\,,\; p_0=\nu_0\\\hline
\text{holomorphic d.s.t.}& \IF(j,\nu_0)& \nu_0\equiv j \bmod 2,\; 0 \leq
\nu_0 \leq j-2\\
&& h_0=2j\,,\; p_0=0 \\\hline
\text{antiholomorphic d.s.t.}& \FI(j,\nu_0)& \nu_0 \equiv j\bmod 2,\; 0
\leq \nu_0 \leq -j-2\\
&& h_0=2j\,,\; p_0=0 \\\hline \hline
\multicolumn{3}{|c|}{\text{thin representations (non-tempered)}} \\
\hline
T^+_{-1} & \IF(1,-1)& j=1\,,\; \nu_0=1\quad h_0=2\,,\;p_0=0\\\hline
T^+_k,\, k\in \ZZ_{\geq 0} & \IF_+(j,-1) & \nu_0=1,\; j=3+2k\\
&& h_0=k+3\,,\; p_0=k+1 \\ \hline
T^-_{-1} & \FI(-1,-1) & j=1,\; \nu_0=1\quad h_0=-2\,,\; p_0=0 \\
\hline
T^-_k, \, k \in \ZZ_{\geq 0}
&\FI_+(j,-1)& \nu_0=1,\; j = -3-2k\\
&& h_0=-k-3\,,\; p_0=k+1 \\
\hline
\end{array}
\] }
\caption[]{List of isomorphism classes of infinite-dimensional
irreducible $(\glie,K)$-modules with a unitary structure, from
\cite[Theorem 16.1]{BMSU21}. We use the notation for the isomorphism
types given in \cite[\S12.2]{BMSU21}.
\\
We give spectral parameters $(j,\nu_0)=(j_\xi,\nu_0)$. The Weyl group
orbit of $(j,\nu_0)$ determines a character $\ps=\ps[j,\nu_0]$ of
$ZU(\glie)$. We give also the $K$-type $\tau^{h_0}_{p_0}$ with minimal
dimension $p_0+1$ occurring in the module (and generating it).}
\label{tab-Vlist}
\end{table}

We fix such a $4$-tuple and consider square integrable automorphic forms
$f\in \Au{(2)}\bigl(\ps[\xi,\nu_0]\bigr)_{h_0,p_0,q_0}$. Each non-zero
$f$ in this space generates a $(\glie,K)$-module
$V_f \subset \Au{(2)}\bigl(\ps[\xi,\nu_0]\bigr)$. These modules $V_f$
are all submodules of $L^{2,\discr}(\Gm\backslash G)$ in the
isomorphism class specified by $(\xi,\nu_0,h_0,p_0)$.

Given the family of basis functions
$\nu \mapsto \Om_\Nfu(\xi,\nu)= \Om_{\Nfu;h_0,p_0,q_0}(\xi,\nu)$ in
Proposition~\ref{prop-MuOm}, we define the Fourier coefficient
$c_\Nfu(f)$ by
\be\label{Cdef} \Four_\Nfu f = c_\Nfu(f)\, \Om_\Nfu(\xi,\nu_0)\,.\ee

\subsection{General completeness result} We can now state the relation
between the non-vanishing of Fourier term operators and Poincar\'e
series, valid in the context indicated above.

\begin{thm}\label{thm-complgen}Let $V$ be an infinite-dimensional
irreducible $(\glie,K)$-module of square integrable automorphic forms,
with isomorphism type determined by the quadruple
$(\xi,\nu_0,h_0,p_0)$, as indicated in Table~\ref{tab-Vlist}. For each
Fourier term order $\Nfu\neq \Nfu_0$ we use the Poincar\'e family
$\M_\Nfu(\xi,\nu) = \poin_\Nfu \Mu_{\Nfu;h_0,p_0,p_0}(\xi,\nu)$. Let
$f \in  V_{h_0,p_0,p_0}$ be nonzero. Then the following statements are
equivalent:
\begin{itemize}
\item[i)] The Fourier
%term operator$\Four_\Nfu: V \rightarrow \Wfu_\Nfu^{\xi,\nu_0}$ 
coefficient $c_\Nfu(f)$ is non-zero.
\item[ii)] The scalar product
$\bigl( f, \tilde P_\Nfu(\xi,\nu_0) \bigr)_{\Gm\backslash G}\neq 0$ for
some cut-off function $\ph$. Here $\tilde P_\Nfu(\xi,\nu_0)$ is the
main part at $\nu_0$ of the truncated Poincar\'e family
$\nu \mapsto \poin_\Nfu\bigl( \ph\Mu_\Nfu(\xi,\nu) \bigr)$.
\end{itemize}
In ii), in many cases depending on $V$, one has that
$P_\Nfu(\xi,\nu_0) = \tilde P_\Nfu(\xi,\nu_0)$, independently of the
cut-off function.

\end{thm}

\rmrk{Remarks}(1)
See Proposition~\ref{prop-MuOm} for the family
$\Mu_\Nfu(\xi,\nu) = \Mu_{\Nfu;h_0,p_0,p_0}(\xi,\nu)$. For the
completeness result we use Poincar\'e families with minimal $K$-type
$\tau^{h_0}_{p_0}$ of $V$, and maximal weight $p_0$ in that $K$-type.
We will often omit the index $(h_0,p_0,p_0)$ from the notation.
\smallskip

\noindent
(2) The spaces $V_{h_0,p_0,p_0}$ and
$\Wfu_{\Nfu;h_0,p_0,p_0}^{\xi,\nu_0}$ with minimal $K$-type
$\tau^{h_0}_{p_0}$ and weight~$p_0$ have dimension one, so the Fourier
term operator $\Four_\Nfu$ is determined by the constant $c_\Nfu(f)$.
\smallskip

\noindent
(3) Since $V$ is non-zero and is not the trivial module, there are
Fourier term orders $\Nfu$ for which statement~i) holds. This implies
that the scalar products in statement~ii) detect all non-zero elements
in $\Au{(2)}(\ps)_{h_0,p_0,p_0}$.\\
\medskip
The proof of Theorem~\ref{thm-complgen} will be completed in
Subsection~\ref{sect-pfcomplgen}.

\subsection{Singularities and zeros of Poincar\'e
families}\label{sect-rPs}
The completeness result Theorem~\ref{thm-complgen} puts the emphasis on
square integrable automorphic forms. Now we put the emphasis on
Poincar\'e families and their singularities and zeros at a point
$\nu_0$.

We handle not all isomorphism classes in Table~\ref{tab-Vlist} together,
but consider the various cases separately. Given $(\xi,\nu_0,h_0,p_0)$,
we consider the behavior of the associated meromorphic family
$\nu \mapsto \poin_\Nfu \Mu_{\Nfu; h_0,p_0,p_0}(\xi,\nu) =  \M_\Nfu(\xi,\nu)$
at $\nu=\nu_0$.

The proofs of the following propositions are
in~\S\ref{sect-pfs}.

\begin{prop}\label{prop-upcstr} Let $(\xi,\nu_0,h_0,p_0)$ be the
parameters of a representation in the {\bf irreducible unitary
principal series}, the {\bf complementary series}, or the {\bf thin
representations} $T^+_{-1}$ or $T^-_{-1}$. In particular $h_0=2j_\xi$
and $p_0=0$. Let $\Nfu$ be a generic abelian or non-abelian Fourier
term order.
%\begin{enumerate}
%\item[i)]
If $\nu_0\neq 0$ (resp. $\nu_0 = 0$), then either the Poincar\'e family
$\nu \mapsto \M_\Nfu(\xi,\nu)$ is holomorphic at~$\nu=\nu_0$ or it has
a singularity of order one (resp. 2).
\begin{enumerate}
\item[a)]The family $\nu \mapsto \M_\Nfu(\xi,\nu)$ has a singularity
%of order one (resp. 2) 
at $\nu=\nu_0$ if and only if there are %are elements
$f\in \Au{(2)}(\ps)_{2j_\xi,0,0}$ with $c_\Nfu(f) \neq 0$.
\item[b)] If the family is holomorphic at $\nu=\nu_0$ and is not
identically zero, then $\M_\Nfu(\xi,\nu_0)$ in
$ \Au!(\ps)_{2j_\xi,0,0} $ is not square integrable.
\end{enumerate}
\end{prop}

Proposition~\ref{prop-upcstr} a) couples the presence of a singularity
to the presence of square integrable automorphic forms for which the
specific Fourier coefficient $c_\Nfu$ does not vanish. If there is no
singularity at $\nu_0$, there may exist non-zero square integrable
automorphic forms in $\Au{(2)}(\ps)_{2j_\xi,0,0}$, but $c_\Nfu$
vanishes for all of them.

\begin{prop}\label{prop-ldstr}Let $(\xi,\nu_0,h_0,p_0)$ be parameters of
a representation of {\bf large discrete series type} or of a {\bf thin
representation} $T^\pm_k$ with $k\in \ZZ_{\geq 0}$. In particular
$h_0=-j_\xi$ and $p_0 = \nu_0\in \ZZ_{\geq 1}$ for the large discrete
series and $h_0=\pm(k+3)$, $p_0=k+1$ for $T^\pm_k$.

If the Poincar\'e family $\nu \mapsto \M_\Nfu(\xi,\nu)$ is not
identically zero, then it is holomorphic at $\nu=\nu_0$, with a value
at $\nu_0$ that is not square integrable.
% If i) holds, 
Furthermore, in this case there are cut-off functions $\ph$ such that
for any $f\in \Au{(2)}(\ps)_{h_0,p_0,p_0}$
\[ c_\Nfu( f) \neq 0 \;\Longleftrightarrow \Bigl( \poin_\Nfu\bigl( \ph\,
\Mu_\Nfu(\xi,\nu_0)\bigr),f \Bigr) \neq 0\,. \]
\end{prop}
\begin{prop}
\label{prop-hads}Let $(\xi,\nu_0,h_0,p_0)$ be parameters of a
representation of {\bf holomorphic or antiholomorphic discrete series
type}.
%or of  {\bf antiholomorphic discrete series type}. 
In particular, $h_0=2j_\xi$, $p_0=0$. Let $\Nfu=\Nfu_\n$ be a
non-abelian Fourier term order.
\begin{enumerate}
\item[i)] If $\nu_0\geq 1$,
%\begin{enumerate} \item[a)] 
the Poincar\'e family $\nu \mapsto \M_\n(\xi,\nu)$ is holomorphic at
$\nu_0$, and its value $\M_\n(\xi,\nu_0)$ is in
$ \Au{(2)}(\ps)_{2j_\xi,0,0}$.
\begin{enumerate}
\item[a)] The value $\M_\n(\xi,\nu_0)$ is non-zero if and only if there
are square integrable automorphic forms
$f\in \Au{(2)}(\ps)_{2j_\xi,0,0}$ for which $c_\Nfu(f) \neq 0$.
%$\Four_\n f \neq 0$.
\item[b)] If $\M_\n(\xi,\nu_0)=0$ and the family is not identically
zero, then the family has a first order zero at $\nu=\nu_0$ and
$\partial_\nu \M_\n(\xi,\nu)\bigr|_{\nu=\nu_0}$ is not square
integrable.
\end{enumerate}
\item[ii)] If $\nu_0=0$, then $\M_\n(\xi,\nu)$ can have at most a first
order singularity at $\nu_0$. Furthermore:
\begin{enumerate}
\item[a)] The family $\nu\mapsto \M_\n(\xi,\nu)$ has a first order pole
at $\nu=0$ if and only if there are $f\in \Au{(2)}(\ps)_{2j_\xi,0,0}$
with $c_\Nfu(f) \neq 0$.
% $\Four_\n f \neq 0$.
\item[b)] If the family is holomorphic at $\nu=0$ and not identically
zero, then $\M_\n(\xi,0)=0$, and
$\partial_\nu \M_\n(\xi,\nu) \bigr|_{\nu=0} \in
\Au!(\ps)_{2j_\xi,0,0} \setminus \Au{(2)}(\ps)_{2j_\xi,0,0}$.
\end{enumerate}\end{enumerate}
\end{prop}
\rmrk{Remark} Part i), with $\nu_0 \in \ZZ_{\geq 1}$, refers to discrete
series representations. In part ii) we deal with limits of discrete
series, occurring as direct summands of reducible principal series
representations. The proofs of the propositions will be given in
Section~\ref{sect-pfc} (see Subsection~\ref{sect-pfs}).

\subsection{Comparison}\label{comparison}
In \cite{MW89} Poincar\'e series attached to a character $\chi_\bt$ of
$\Ld_\sigma\backslash N$ were studied for $G$ semisimple Lie group of
real rank one. It was proved that the meromorphically continued family
$\M_\bt(\xi,\nu)$ has at most simple poles located at spectral points
$\nu_0$, except for $\nu_0 =0$ where it may have a double pole. In
particular, if $\nu_0\neq 0$,\emph { under suitable conditions on the
parameters}, for any square integrable automorphic form $f$ one has the
relation
\be \Bigl(f, \textrm{Res}_{\nu=\nu_0} \M_\bt(\xi,\nu) \Bigr) \ddis
c_\bt(f), \label{innerprfla}\ee
where $\ddis$ indicates equality up to a nonzero constant. That is, the
residues of the  Poincar\'e  families $\M_\bt(\xi,\cdot)$ 
detect all cusp forms having a
non-zero abelian Fourier coefficient.
There is a similar result for
$\nu_0=0$. However, these results do not give information on
non-abelian Fourier terms, and in particular, on non-generic
automorphic forms, i.e. those having all abelian coefficients
$c_\bt(f)$ equal to zero. Thus, for the group $\SU(2,1)$, they do not
include cusp forms generating a $(\glie,K)$-module lying in the
holomorphic (or antiholomorphic)
discrete series or in the thin representations $T^\pm_k$, (with
$\k\ge -1$), which do not admit abelian Whittaker models.

In Theorem~\ref{thm-complgen} we prove a version of \eqref{innerprfla}
that involves both, abelian and non-abelian Fourier terms, showing that
the main parts of the corresponding Poincar\'e series (in some cases
their truncated main parts) do detect the whole space of cusp
forms.\medskip

If we compare with the classical results for $\SL(2,\RR)$ (\cite{Pe32},
\cite{Neu73},\cite{Nie73},\cite{Ra77}), we see that there are
similarities and some differences.

In the case of the unitary principal series and the complementary series
the results in Proposition~\ref{prop-upcstr} are similar to the results
for Maass forms of weight zero.

In the case of holomorphic and antiholomorphic series, for
$\nu_0 \in \ZZ_{\geq 3}$ the absolutely convergent Poincar\'e series
span the corresponding spaces of square integrable automorphic forms.
This extends to $\nu_0=2$ and $\nu_0=1$ by working with the values of
the analytically continued Poincar\'e series. In the case when
$\nu_0=0$, we deal with limits of holomorphic or antiholomorphic
discrete series, and, for the detection, we need the residues at the
first order singularities of the non-abelian Poincar\'e series. All
this is similar to what happens with holomorphic automorphic forms on
the upper half-plane. For weights $k\in \ZZ_{\geq 2}$ the values of
Poincar\'e series span the spaces of integrable automorphic forms, with
absolute convergence for $k\geq 3$ and conditional convergence for
$k=2$. Weight $1$ corresponds to $\nu_0=0$. Then we need residues of
Poincar\'e families to span the spaces of square integrable automorphic
forms. In Appendix~\ref{PSSl2R} we give a more detailed discussion
following \cite{B94}.

There are also differences between $\SL_2(\RR)$ and $\SU(2,1)$. Maximal
unipotent subgroups of $\SL_2(\RR)$ are abelian, and the Fourier
expansions are simpler than for $\SU(2,1)$. Since $\SL_2(\RR)$ has a
non-trivial center, we need to work with automorphic forms with a
character of $\Gm$ when we consider odd weights.

The cases of the large discrete series and of the representations
$T_k^{\pm}$ do not have analogues for $\SL(2,\RR)$. For $\SL_2(\RR)$
the square integrable automorphic forms are always obtained from
values or by resolving singularities of Poincar\'e series. We do not
need statements like in Proposition~\ref{prop-ldstr}, where we had to use
scalar products of automorphic forms with truncated Poincar\'e
families.

\section{Proofs of completeness results}\label{sect-pfc}

\subsection{Scalar product formulas and Wronskian
order}\label{sect-spwo}
The results to prove depend on spectral parameters $(\xi,\nu_0)$ and on
a $K$-type $\tau^{h_0}_{p_0}$ %and weight~$q$ in 
that together determine an irreducible class of irreducible
$(\glie,K)$-modules in Table~\ref{tab-Vlist}. We first state and prove
some facts that are valid more generally. Proposition~\ref{prop-MSW} is
valid for an arbitrary $k$-type $\tau^{h}_{p}$ and a weight $q$ in that
$K$-type. In Subsection~\ref{scalarprods} we assume that the $K$-type
$\tau^{h}_{p}$ occurs in $L^{2,\discr}(\Gm\backslash G)$. For a square
integrable automorphic form $f$ and a Fourier term order
$\Nfu \neq \Nfu_0$, the Fourier coefficient $c_\Nfu(f)$ is defined as
in \eqref{Cdef} by the equation
\be \Four_\Nfu f = c_\Nfu(f)\, \Om_\Nfu(\xi,\nu_0)\,.\ee

The first main point in this section is to obtain the
relation~\eqref{fr}, which expresses the scalar product
$\bigl( \tilde\M_\Nfu(j,\nu) , f \bigr)_{\Gm\backslash G}$ as the
product of $c_\Nfu(f)$ and a quantity that is a meromorphic function
of~$\nu$, not depending on~$f$. After that, we specialize to
$(h,p,q) = (h_0,p_0,p_0)$ and proceed to prove the results in
Section~\ref{sect-stcr}. We shall make use of Theorem~\ref{thm:orders},
concerning Wronskian orders.
%In Theorem~\ref{thm:orders} we  give  the Wronskian order for each type of irreducible $(\glie ,K)$-modules. 
%at~$\nu=\nu_0$, a main point in the section. 
%In it %that theorem 
%We take $(h,p,q) = (h_0,p_0,p_0)$, like in Theorem~\ref{thm-complgen} and in the other results in Section~\ref{sect-stcr}. 
Comparison of the orders with the Wronskian orders will lead to a proof
of Theorem~\ref{thm-complgen}.

\subsection{Sesquilinear forms} The relations we will obtain in
Proposition~\ref{prop-MSW} below are valid in more generality than we
need it. In this subsection the $K$-type $\tau^h_p$ is general.

Given $F_1 , F_2\in L^2( \Ld_\s\backslash G)_K$ we have
\bad \bigl( F_1,F_2\bigr)_{\Gm\backslash G}
&\= \int_{n\in \Ld_\s\backslash N} \int_{t=0}^\infty \int_{k\in K}
F_1\bigl( n\am(t)k \bigr)\, \overline{ F_2\bigl( n \am(t) k \bigr)}\,
dn \, \frac{dt}{t^5}\, dk
\\
&\= \int_{t=0}^\infty \bigl\{ F_1,F_2\bigr\}(t)\,\frac{dt}{t^5}\,,
\ead
where we use the sesquilinear form
\be (F_1,F_2)\mapsto \bigl\{ F_1,F_2\bigr\}(t)\= \int_{\Ld_\s\backslash
N} \int_K F_1\bigl( n \am(t) k \bigr)\,\overline{F_2\bigl(n\am(t) k
\bigr)}\, dn\, dk\ee
on $C^\infty(\Ld_\s\backslash G)_K$ with values in $C^\infty(0,\infty)$.
The integration is over compact sets, so the sesquilinear map is
well-defined.

We shall also use the Maass-Selberg form
\be \MS(F_1,F_2) (t) \= \bigl\{ F_1,CF_2\bigr\}(t) - \bigl\{ C
F_1,F_2\bigr\}(t).\ee

Functions in $C^\infty(\Ld_\s\backslash G)_{h,p,q}$ of weight $q$ in the
$K$-type $\tau^h_p$ that transform on the left according to the same
representation of $N$ have an expansion %of the following type
\be \label{Fjexp}
f\bigl( n \am(t) k \bigr) \= \sum_r w_r(n)\, f_r(t)\, \Kph h p r
q(k)\,,\ee
where $r$ runs over $|r|\leq p$, $r\equiv p\bmod 2$. If $\Nfu=\Nfu_\bt$,
then all $w_r$ are equal to the character $\ch_\bt$ of~$N$ in
\eqref{chbt}. If $\Nfu=\Nfu_{\ell,c,d}$ then the $w_r$ are
theta-functions $\Th_{\ell,c}(h_{\ell,m})$ as in \S\ref{sect_thfonb}.
They are mutually orthogonal in $L^2(\Ld_\s\backslash N)$ and have norm
$\sqrt {\frac 2\s}$. The $\Kph hprq$ are elements of an orthogonal
system in $L^2(K)$, with norms satisfying relation~\eqref{Kph-norm}.

We compute
\begin{align*}
\bigl\{ F_1,F_2\bigr\}(t) &:\= \int_{n\in \Ld_\s\backslash N} \int_{k\in
K} \sum_r w_r(n)\, f_{1,r}(t)\, \Kph h p r q(k)\\
&\qquad\hbox{} \cdot
\sum_{r'} \overline{ w_{r'}(n)} \, \bar f_{2,r}(t)\, \overline{ \Kph
hp{r'}q(k)}\, dn\, dk
\displaybreak[0]
\end{align*}
Using the properties of the functions $w_r$ and $\Kph hprq$, we arrive
at
\be \label{bcp} \bigl\{ F_1,F_2\bigr\}(t) \=\frac 2\s\,\bigl\|\Kph
h{p}{p}{q}\bigr\|_K^2\, \sum_r \binom p{\frac{p+r}2} f_{1,r}(t)\, \bar
f_{2,r}(t)\,. \ee

We apply this formula to get an expression for the Maass-Selberg form.
\begin{prop} %\txt Maass-Selberg form and Wronskian
\label{prop-MSW} Let $F_1,F_2\in C^\infty(\Ld_\s\backslash G)_{h,p,q}$
have expansions as in~\eqref{Fjexp}. Consider the generalized Wronskian
for $h_1,h_2\in C^\infty(0,\infty)$ and $n\in \ZZ_{\geq 1}$, given by
\be \Wr_n(h_1,h_2)(t) \= h_1(t)\, h_2^{(n)}(t) - h_1^{(n)}(t)
\,h_2(t)\,.\ee
Then
\badl{MSW2} \MS(F_1,F_2)(t)&\= \frac 2\s\,\, \bigl\|\Kph
h{p}{p}{q}\bigr\|_K^2\sum_{r\equiv p(2),\, |r|\leq p} \binom
p{\frac{p+r}2}\\
&\qquad\hbox{} \Bigl( -3t \, \Wr_1 \bigl(f_{1,r},\bar f_{2,r} \bigr)(t)
+ t^2\,\Wr_2\bigl( f_{1,r},\bar f_{2,r}\bigr )(t) \Bigr)\,,\\
&\=\frac 2\s\,\, \bigl\|\Kph h{p}{p}{q}\bigr\|_K^2 \, t^5\, \frac{d}{dt}
\bigl( t^{-3}\, \W(F_1,F_2)(t)\bigr),
\eadl
with
\be \W(F_1,F_2)(t) \= \sum_{r\equiv p(2),\, |r|\leq p} \binom
p{\frac{p+r}2}\, \Wr_1(f_{1,r}, \bar f_{2,r})(t)\,. \ee

The form $\Wr=\Wr_1$ is the usual Wronskian.
\end{prop}
\begin{proof}The action of the Casimir element $C$ on $F_1$ and $F_2$
can be made explicit in the components $f_{j,r}$ with the Mathematica
routines discussed in \cite[\S9.2]{BMSU21}. This gives the first
relation in~\eqref{MSW2}. The second equality in \eqref{MSW2} becomes
clear in a direct check.
\end{proof}

\subsection{Scalar product formulas} \label{scalarprods} Our goal in
this subsection is to prove two useful relations. We will work with the
basis families
$\nu \mapsto \Mu_\Nfu(\xi,\nu) =\Mu_{\Nfu;h,p,q}(\xi,\nu)$, and the
Poincar\'e family
$\nu \mapsto \poin(\xi,\nu) \Mu_{\Nfu;h,p,q}(\xi,\nu) = \M_\Nfu(\xi,\nu)$.
Given a square integrable automorphic form $f$, the Fourier
coefficients $c_\Nfu(f)$ are defined by the identity
\be\label{Fourcoeff}
\Four_\Nfu f \= c_\Nfu(f)\, \Om_\Nfu(\xi,\nu_0)\qquad (\Nfu\neq
\Nfu_0)\,.\ee
See Proposition~\ref{prop-MuOm} for the basis function
$\Om_\Nfu$.

\begin{lem}\label{lem-fr0}For $\re\nu$ sufficiently large
\be \label{fr0old}
\bigl( \tilde \Poin_\Nfu(\xi,\nu) , f \bigr)_{\Gm\backslash G} \=
\overline{c_\Nfu(f)}\, \int_{t=0}^\infty
\ph(t)\,\bigl\{\Mu_\Nfu(\xi,\nu), \Om_\Nfu(\xi,\nu_0) \bigr\} (t)\,
\frac{dt}{t^5}\,.\ee
\end{lem}

\begin{proof}
The automorphic form $f$ has a decomposition into Iwasawa coordinates
\be f\bigl( n ak \bigr) \= \sum_r f_r(na ) \,\Kph hpr{ q }(k)\,,\ee
with components $f_r \in C^\infty(\Ld_\s\backslash NA)$. The sum is over
$r\equiv p \bmod 2$, $|r|\leq p$.

The function $\Om_\Nfu(\xi,\nu)$ has also a description in components:
\be \Om_\Nfu(\xi,\nu)\bigl( n \am(t)k \bigr) \= \sum_r w_r(n) \,
h_r^\om(\nu,t)\, \Kph hprq(k)\,,\ee
where the functions $h_r^\om(\nu,\cdot)$ are linear combinations of
$K$-Bessel functions or $W$-Whit\-taker functions with exponential
decay as $t\uparrow \infty$, that satisfy an estimate
$\oh\bigl( t^{2-\al_r}\bigr)$ as $t\downarrow 0$ for some $\al_r > 0$.
See \eqref{OmMu-est}.

In the abelian case the Fourier term in \cite[(5.5)]{BMSU21} is
\begin{align*}
\Four_\bt f \bigl(n\am(t)k\bigr)
&\= \frac 2\s\, \int_{n'\in\Ld_\s\backslash N}\overline{\ch_\bt(n')}\,
\sum_r f_r\bigl(n'n\am(t)\bigr) \, \Kph h p r q(k) \,
dn'\displaybreak[0]\\
&\=\frac 2\s\, \ch_\bt(n) \int_{n'\in \Ld_\s\backslash N}
\overline{\ch_\bt(n')} f_r\bigr(n'\am(t) \bigr) \, dn'\; \Kph
hprq(k)\,.
\end{align*}
%(We use that $\Gm_{\!N} = \Ld_\s$.)
Using \eqref{Fourcoeff} and the notation in \eqref{OmMuh}, we
see that the  
left-hand side is equal to
\begin{align*} c_\bt(f) &\, \Om_\bt(\xi,\nu_0)\bigl( n\am(t) k \bigr)
\= c_\bt(f) \, \sum_r \ch_\bt(n)\, h_r^\om(\nu_0,t) \, \Kph hprp(k)\,.
\end{align*}
This gives
\be\label{cof} c_\bt(f)\, h_r^\om(\nu_0,t)
\= \frac 2\s\, \int_{\Ld_\s\backslash N} \overline{\ch_\bt(n)}\,
f_r\bigl( n \am(t)\bigr)\, dn\,.\ee

In the non-abelian case we use \cite[(8.21)]{BMSU21}, with
$\n=(\ell,c,d)$, $w_r = \Th_{\ell,c}(h_{\ell,m(h,r)})$, and $m(h,r)$ as
in \eqref{parms-nab}.
\begin{align*}
\Four_\n f \bigl( n \am(t)k \bigr)&\= \sum_r w_r(n) \, \frac{\Kph
hprq(k)}{\bigl\|\Kph h p r q \bigr\|^2_K}\, 2 \int_{n'\in
\Ld_\s\backslash N} \int_{k'\in K} \overline{w_r(n')}\\
&\qquad\hbox{} \cdot
\sum_{r'} f_{r'}\bigl( n' \am(t)\bigr)\, \Kph hp{r'}q(k')\,
\overline{\Kph hprp(k')}\, dk'\, dn'
\displaybreak[0]\\
&\= \frac 2\s\,\sum_r w_r(n) \, \Kph hprq(k)\, \int_{n'\in
\Ld_\s\backslash N} \overline{w_r(n')} \, f_r\bigl( n' \am(t)\bigr) \,
dn'\,.
\end{align*}
We compare this with
\[ \Four_\n f \= c_\n(f) \,\Om_\n(\xi,\nu_0)\,,\]
to obtain
\be \label{cof1} c_\n(f) \, h_r^\om(\nu_0,t) \= \frac 2\s\,\int_{n'\in
\Ld_\s\backslash N} \overline{w_r(n)} \, f_r\bigl( n \am(t)
\bigr)\, dn\,,\ee
which is similar to the expression in the abelian case.

For the function $\Mu_\Nfu(\xi,\nu)$ as well, we use the
notation in~\eqref{OmMuh}:
\be \Mu_\Nfu(\xi,\nu)\bigl( n\am(t) k \bigr) \= \sum_r w_r(n)\,
h^\mu_r(\nu,t) \, \Kph hprq(k)\,,\ee
with $h^\mu_r$ expressed in $I$-Bessel functions or $M$-Whittaker
functions, and with $w_r$ as above. We have
\be \label{bound-mur} h^\mu_r(\nu,t) \= \oh(t^{2+\nu })\,.\ee

Now we take $\re\nu>2$ to have convergence of the Poincar\'e series, and
$\re \nu > \al_r$ in the exponent of the estimate for
$h^\om_r(\nu_0,t)$. We take apart the absolutely convergent Poincar\'e
series and with the use of \eqref{bcp} we arrive at the scalar product
formula in the lemma. Indeed
\begin{align*}
\bigl( \tilde\Poin_\Nfu&(\xi,\nu) , f\bigr)_{\Gm\backslash G} \=
\int_{\Gm \backslash G}\sum_{\gm\in \Gm_{\!N}\backslash G} \ph(\gm g)
\Mu_\Nfu(\xi,\nu)(\gm g)
\,\overline{f(g)}\, dg\\
&\= \int_{\Gm_{\!N} \backslash G } \ph(g)\,\Mu_\Nfu(\xi,\nu)(g)\,
\overline{f(g)}\, dg\displaybreak[0]\\
&\= \int_{n\in \Gm_{\!N}\backslash N} \int_{t=0}^\infty \int_{k\in K}
\ph(t)
\, \sum_r w_r(n)\, h^\mu_r(\nu;t) \, \Kph hprq(k) \\
&\qquad\hbox{} \cdot
\sum_{r'} \overline{f_{r'}\bigl( n \am(t) \bigr)} \,\overline{\Kph
hp{r'}q(k)}\, dn\, \frac{dt}{t^5}\, dk
\displaybreak[0]\\
&\= \overline{c_\Nfu(f)}\,\bigl\| \Kph hprq\bigr\|_K^2 \, \frac 2\s
\,\,\sum_r \int_{t=0}^\infty \ph(t) \, h^\mu_r(\nu,t) \,
\overline{h_r^\om(\nu_0,t)}\, \frac{dt}{t^5} \,.\end{align*}
The convergence as $t\uparrow\infty$ is no problem by the presence of
the cut-off function $\ph$. The assumption on $\re\nu$ makes
$h^\mu_s(\nu,t)$ counteract the growth of the functions
$h^\om_r(\nu_0,t)$ as $t\downarrow0$.
\end{proof}

\begin{prop}\label{prop-scpf} Let $\Nfu\neq \Nfu_0$, and let
$f\in \Au{(2)}\bigl(\ps[\xi,\nu_0]\bigr)_{h,p,q}$. Then we have
\badl{fr} \bigl( \tilde \Poin_\Nfu(\xi,\nu)& , f \bigr)_{\Gm\backslash
G}\= \overline{c_\Nfu(f)}\, \bigl\|\Kph hppq\bigr\|^2_K \frac 2\s\\
&\qquad\hbox{} \cdot
\frac1{(\nu^2-\nu_0^2)}\, \int_{t=0}^\infty \ph'(t)\, \W\bigl(
\Mu_\Nfu(\xi,\nu),\Om_{\Nfu}(\xi,\nu_0)
\bigr)(t) \, \frac{dt}{t^3}\,,
\eadl
as an identity of meromorphic functions in $\nu\in \CC$. The
singularities of the function
$\nu \mapsto \W\bigl( \Mu_\Nfu(\nu),\Om_{\Nfu}(\nu_0)
\bigr)$ are contained in the set~$\ZZ_{\leq -1}$.
\end{prop}

\begin{proof}
Since %the elements of the family 
$\Mu_\Nfu(\xi,\nu)$ and %the function
$\Om_\Nfu(\xi,\nu_0)$ are eigenfunctions of the Casimir element~$C$, we
have
\be\label{Crel}\MS\bigl( \Mu_\Nfu(\xi,\nu),\Om_\Nfu(\xi,\nu_0) \bigr)(t)
\=
(\nu_0^2-\nu^2)\, \bigl\{ \Mu_\Nfu(\xi,\nu),\Om_\Nfu(\xi,\nu_0)
\bigr\}(t)\,. \ee

Applying Lemma~\ref{lem-fr0} and \eqref{MSW2} we get for $\re\nu$
sufficiently large
\begin{align*}
\bigl( \tilde\M_\Nfu(\xi,\nu),& f\bigr)_{\Gm\backslash G} \=
\overline{c_\Nfu(f)} \frac1{\nu_0^2-\nu^2} \, \bigl\|\Kph h p p q
\bigr\|_K^2 \frac 2\s\\
&\qquad\hbox{} \cdot
\int_{t=0}^\infty \ph(t)\, \frac{d}{dt} \Bigl( t^{-3}\, \W\bigl(
\Mu_\Nfu(\xi,\nu), \Om_\Nfu(\xi,\nu_0)(t) \bigr) \Bigr) \, dt.
\end{align*}
To apply partial integration, the Wronskian function has to have
sufficient decay at zero. This is based on the components
$h^\mu_r(\nu,t)$ and $h^\om_r(\nu_0,t)$ and their derivatives, with
expressions in modified Bessel functions or Whittaker functions for
which we have \eqref{IKest} and~\eqref{MWest}. Using the contiguous
relations in Appendix \ref{sect-rpfWr} we can handle the derivatives.
Enlarging $\re\nu$ sufficiently, we arrive at \eqref{fr} in the
proposition.

The left-hand side depends meromorphically on $\nu$. The integral in the
right-hand side is over a compact set in $(0,\infty)$ and we thus get
its meromorphic continuation with singularities only in
$\ZZ_{\leq -1}$, due to the Wronskian functions.
\end{proof}
\subsection{Proof of Theorem~\ref{thm-complgen}} \label{sect-pfcomplgen}
The results up till here hold for any square integrable automorphic form
$f\in \Au{(2)}\bigl( [\xi,\nu_0] \bigr)_{h,p,q}$. The isomorphism class
of the irreducible module $V_f = U(\glie)\, f$ is determined by the
spectral parameters $(\xi,\nu_0)$ and the $K$-type. To use
Proposition~\ref{prop-scpf} we have to study the Wronskian function
$ \W\bigl( \Mu_\Nfu(\xi,\nu), \Om_\Nfu(\xi,\nu_0)\bigr)$. We now
proceed with the minimal $K$-type $\tau^{h_0}_{p_0}$ in the irreducible
module under consideration in order to use the explicit formulas
available for that $K$-type.

Let $V$ be an irreducible submodule of
$\Au{(2)}\bigl(\ps[\xi,\nu_0]\bigr)$, with isomorphism class
characterized by $(\xi,\nu_0,h_0,p_0)$ as indicated in
Table~\ref{tab-Vlist}. This submodule is of the form $V_f$ for any
non-zero $f\in \Au{(2)}\bigl( \ps[\xi,\nu_0]\bigr)_{h_0,p_0,p_0}$.

For $\Nfu\neq\Nfu_0$ we put
\be\label{difWrint} I_\Nfu(\ph;\nu,\nu_0) \= \int_{t=0}^\infty \ph'(t)\,
\W\bigl( \Mu_\Nfu(\xi,\nu),\Om_{\Nfu}(\xi,\nu_0)
\bigr)
(t) \, \frac{dt}{t^3}\,.\ee
The parameters $\xi$, $h_0$ and $p_0$ are not visible in the notation.

At points $\nu=\nu_0$ with $\re\nu_0\geq 0$ the function
$\nu\mapsto I_\Nfu(\ph;\nu,\nu_0)$ is holomorphic. So the order at
$\nu=\nu_0$ of the function
\be\label{W-f} \nu \mapsto \frac1{\nu^2-\nu_0^2} \,
I_\Nfu(\ph;\nu,\nu_0)
\ee
is at most $1$ if $\nu_0\neq0$, and at most $2$ if $\nu_0=0$. The
cut-off function $\ph$ may have an influence on this order.

\begin{defn}\label{def-wo} The \emph{Wronskian order}, denoted by
$\wo_\Nfu(\nu_0)$, is the %we denote the
maximum of the orders of the functions \eqref{W-f} as $\ph$ varies over
all cut-off functions. A cut-off function $\ph$ is said to be
\emph{good} if the order of the function in \eqref{W-f} coincides with
the Wronskian order.
\end{defn}

\begin{lem}\label{lem-gco}For each $T>0$ there is a good cut-off
function $\ph_T$ for $\wo_\Nfu(\nu_0)$ such that $\ph_T=1$ on $[0,T]$,
$\ph_T(t)\geq 0$ for $T\leq t\leq T=1$ and $\ph_T=0$ on $[T+1,\infty)$.
\end{lem}
\begin{proof}The function
$(\nu,t) \mapsto \W\bigl( \Mu_\Nfu(\xi,\nu),\Om_{\Nfu}(\xi,\nu_0)
\bigr)(t)$ is of the form $(\nu,t) \mapsto (\nu-\nu_0)^u\, A(\nu,t)$,
with $u\in \ZZ$ and $A\in C^\om\left( U \times(0,\infty) \right)$ a
real-analytic function which is holomorphic in $\nu$ on a neighborhood
$U$ of $\nu_0$ in~$\CC$, and such that $t\mapsto A(\nu_0,t)$ is not
identically zero. Then the order of the function in \eqref{W-f} is at
most $1-u$ if $\nu_0\neq 0$ and $2-u$ if $\nu_0=0$.

For each $T>0$ there exists $\al\in C_c^\infty(0,\infty)$ with
$\mathrm{Supp} \,\al \subset[T,T+1]$, $\al\geq 0$, such that
$\int_{t=0}^\infty \al(t)\, A(\nu_0,t)\, \frac{dt}{t^3}\neq 0$ and
$\int_{t=T}^{T+1}\al(t) \, dt=1$. Then
$\ph_T(t) = \int_{x=t}^\infty \al(x)\, dx$ is a cut-off function for
which the order of the function in~\eqref{W-f} is equal to $1-u$ if
$\nu_0\neq 0$ and $2-u$ if $\nu_0=0$. So $\ph_T$ is a good cut-off
function.
\end{proof}

Each non-zero square integrable automorphic form
$ f\in \Au{(2)}(\ps)_{h_0,p_0,p_0}$ generates a $(\glie,K)$-module
$V_f$ in $\Au{(2)}$ in the isomorphism class specified by
$(\xi,\nu_0,h_0,p_0)$. We now consider Poincar\'e families
$\M_\Nfu(\xi,\nu)_{h_0,p_0,p_0}$ for a Fourier term order
$\Nfu\neq \Nfu_0$.

\begin{lem}\label{lem-complgen}Let $\ph$ be a good cut-off function. Let
$\ps=\ps[\xi,\nu_0]$ and let $f$ be a non-zero square integrable
automorphic form in $\Au{(2)}(\ps)_{h_0,p_0,p_0}$.

Then the Fourier coefficient $c_\Nfu(f)$ % in \eqref{Cdef} 
is non-zero if and only if the scalar product
$\bigl( \tilde P_\Nfu(\nu_0), f \bigr)_{\Gm\backslash G}$ is non-zero.
  Furthermore, if for some $f \in  \Au{(2)}(\ps)_{h_0,p_0,p_0}$ these
  equivalent conditions hold, then %we have that
  $\po_\Nfu(\nu_0) =\wo_\Nfu(\nu_0)$.

In particular, if $\M_\Nfu(\xi,\nu)$ is not identically zero and the
main part $P_\Nfu(\nu_0)$ is square integrable then
$c_\Nfu (P_\Nfu(\nu_0))\neq 0$. If $\nu_0=0$ and
$\po_\Nfu(0) = 2 \;(\textrm {resp.} 1)$ then
$\wo_\Nfu(0) =2 \;(\textrm {resp.} 1)$.
\end{lem}

\begin{proof}
Since
$\tilde \M(\xi,\nu) = \poin_\Nfu \bigl( \ph \Mu_\Nfu(\xi,\nu) \bigr)$
is a meromorphic family, the scalar product in the lemma is also
meromorphic in~$\nu$. We use Proposition~\ref{prop-scpf} to see that
\be\label{spf} \bigl( \tilde\M_\Nfu(\xi,\nu) , f \bigr)_{\Gm\backslash
G} \dis \overline{c_\Nfu(f)}\, \frac{ I_\Nfu(\ph;
\nu,\nu_0)}{\nu^2-\nu_0^2}\,.\ee
If $c_\Nfu(f)=0$, then the scalar product vanishes identically in~$\nu$,
and hence
\[\bigl( \tilde P_\Nfu(\nu_0) ,f\bigr)_{\Gm\backslash G}\=0\,.\]
If $c_\Nfu(f) \neq 0$, then
\begin{align}\notag &\bigl( \tilde \M_\Nfu(\xi,\nu),
f\bigr)_{\Gm\backslash G} \;\stackrel.\sim\;
(\nu-\nu_0)^{-\wo_\Nfu(\nu_0)} \quad\textrm{ and hence } \\\label{ri}
&\bigl( \tilde P_\Nfu(\nu_0) , f\bigr)_{\Gm\backslash G}
\;\stackrel.\sim\;
(\nu-\nu_0)^{\po_\Nfu(\nu_0) -\wo_\Nfu(\nu_0)}
\end{align}
as $\nu \rightarrow\nu_0$.

We see from \eqref{ri} that $\po_\Nfu(\nu_0)  <\wo_\Nfu(\nu_0)$ is not
possible, and if $\po_\Nfu(\nu_0)  > \wo_\Nfu(\nu_0)$ then
$\po_\Nfu(\nu_0)>0$ and $\tilde P_\Nfu(\nu_0) =  P_\Nfu(\nu_0)$ is
square integrable and nonzero. Now by taking $f=  P_\Nfu(\nu_0)$ we
have that
\be \bigl(P_\Nfu(\nu_0) , P_\Nfu(\nu_0)\bigr)_{\Gm\backslash G}
\;\stackrel.\sim\; \overline{c_\Nfu(P_\Nfu(\nu_0))}
(\nu-\nu_0)^{\po_\Nfu(\nu_0) -\wo_\Nfu(\nu_0)}, \ee
a contradiction, since the r.h.s tends to $0$ as $\nu\rightarrow \nu_0$.

Therefore we must have that $\po_\Nfu(\nu_0) = \wo_\Nfu(\nu_0)$, hence
$\bigl( \tilde P_\Nfu(\nu_0) , f\bigr)_{\Gm\backslash G}\stackrel.\sim 1$,
and $\bigl( \tilde P_\Nfu(\nu_0) , f\bigr)_{\Gm\backslash G}$ is
non-zero, as asserted.

\medskip

Suppose that $P_\Nfu(\nu_0)$ is square integrable. By the first
assertion, applied with $f=P(\nu_0)$, the coefficient
$c_\Nfu\bigl( P_\Nfu(\nu_0) \bigr)\neq 0$ if and only if
$\bigl( \tilde P(\nu_0),P(\nu_0) \bigr)\neq0$ for all good cut-off
functions.

If $\po_\Nfu(\nu_0)>0$, then $\tilde P_\Nfu(\nu_0) \= P_\Nfu(\nu_0)$,
and the statement follows.

If $\po_\Nfu(\nu_0)=0$, then
$\bigl( \tilde P_\Nfu(\nu_0), P_\Nfu(\nu_0) \bigr) =
\bigl( \tilde M_\Nfu(\xi,\nu_0) ,\M_\Nfu(\xi,\nu_0) \bigr)$
approximates $\bigl\|\M_\Nfu(\xi,\nu_0) \bigr\|^2 \neq 0$, if we take
good cut-off functions $\ph_T$ as in Lemma~\ref{lem-gco} and let
$T\uparrow\infty$.

The case $\po_\Nfu(\nu_0)<0$ is impossible, by \eqref{ri} and the
relation $\wo_\Nfu(\nu_0)\geq 0$ (Theorem~\ref{thm:orders}).

The last assertions follow by similar arguments, by taking
$f=P_\Nfu(\nu_0)$ and again using \eqref{ri}.
\end{proof}
\begin{proof}[Proof of Theorem~\ref{thm-complgen}] Let $V$ be an
irreducible $(\glie,K)$-module generated by a square integrable
automorphic form~$f$. Then
$\Four_\Nfu f = c_\Nfu(f)\,\Om_\Nfu(\xi,\nu_0)$, and the vanishing of
the Fourier term operator is determined by the vanishing of
$c_\Nfu (f)$. Thus the theorem follows from Lemma~\ref{lem-complgen}.
\end{proof}

\subsection{Singularities and zeros of Poincar\'e
families}\label{sect-szPf}

The general completeness result in Theorem~\ref{thm-complgen} tells us
about the relation between vanishing of Fourier terms of square
integrable automorphic forms and its detection by truncated Poincar\'e
series. The results in \S\ref{sect-rPs} involve the Poincar\'e families
themselves, their singularities and zeros. We point out that the proof
does not use any of the results in the previous section, since it only
considers Wronskians and not Poincar\'e series.

\begin{prop}\label{prop-po}Let $\Nfu$, $j=j_\xi$ and $\nu_0$ be as
above. Let $\ps= \ps[j,\nu_0]$.

Then, the order $\po_\Nfu(\nu_0)$ at $\nu_0$ of the Poincar\'e family
$\nu \mapsto \M_\Nfu(\xi,\nu)$ can only have values in
$\{2,1,0,-1,-\infty\}$ and the main part $P_\Nfu(\nu_0)$ at $\nu_0$ of
the family is an automorphic form %with moderate exponential growth 
in $\Au!(\ps)_{h_0,p_0,p_0}$.
\begin{itemize}
\item[i)] If $\po_\Nfu(\nu_0) >0$, then %the main part 
$P_\Nfu(\nu_0)= \tilde P_\Nfu(\nu_0)$
%is equal to the main part at $\nu_0$ of the truncated Poincar\'e family
for each choice of the cut-off function $\ph$, and $P_\Nfu(\nu_0)$ is a
square integrable automorphic form in the space
$\Au{(2)}(\ps)_{h_0,p_0,p_0}$. 
\item[ii)] If $\po_\Nfu(\nu_0)=0$, then $P_\Nfu(\nu_0)\in \Au{(2)}(\ps)$
if and only if the $V_f$ are of holomorphic or antiholomorphic discrete
series type.
\item[iii)] If $\po_\Nfu(\nu_0)=-1$, then the $V_f$ are of holomorphic
or antiholomorphic discrete series type, and $P_\Nfu(\nu_0)$ is not a
square integrable automorphic form.
\item[iv)] If $\po_\Nfu(\nu_0)=-\infty$, then $P_\Nfu(\nu_0)=0$.
\end{itemize}
\end{prop}

\begin{proof} Most of these statements have been established in earlier
sections. Theorem~\ref{thm-mer-cont} (meromorphic continuation)
and Proposition~\ref{prop-mpPf} imply that $P_\Nfu(\nu_0)\in \Au!(\ps)$,
and $P_\Nfu(\nu_0) \in \Au{(2)}(\ps)$ if $\po_\Nfu(\nu_0)>0$. Also, we
have that $\po_\Nfu(\nu_0) \leq 1$ if $\nu_0\neq 0$, and
$\po_\Nfu(0) \leq 2$. The fact that
$P_\Nfu(\nu_0) = \tilde P_\Nfu(\nu_0)$ if $\po_\Nfu(\nu_0)>0$ was
discussed in~\S\ref{sect-mp}. Part iv) of the proposition is a direct
consequence of the definitions. This leaves us to prove ii)
and~iii) and the fact that Poincar\'e orders in $\ZZ_{\leq -2}$ do not
occur.

We suppose that % $\po_\Nfu(\nu_0)=0$, i.e.
$\M_\Nfu(\xi,\nu)$ is holomorphic at $\nu=\nu_0$. 
We use the
decomposition~\eqref{poin-split}. In the region of absolute
convergence, the term $\poin' \Mu_\Nfu(\xi,\nu)$ has polynomial decay,
so its Fourier term of order $\Nfu$ is a multiple of
$\Om_\Nfu(\xi,\nu)$. Hence the 
Fourier term of order $\Nfu$
has the 
following form, first in the region of absolute
convergence, and then extended by meromophic continuation:
\bad \Four_\Nfu\M_\Nfu(\xi,\nu) &\= \Four_\Nfu\M^\infty_\Nfu(\xi,\nu) +
C_{\Nfu,\Nfu}(\xi,\nu) \Om_\Nfu(\xi,\nu)\,,\\
\M^\infty_\Nfu(\xi,\nu)(g) &\= \sum_{\gm\in \Gm_{\!N}\backslash \Gm_P}
\Mu_\Nfu(\xi,\nu)(\gm g)\,.
\ead The function $\nu\mapsto C_{\Nfu,\Nfu}(\xi,\nu)$ is holomorphic on
a neighborhood of~$\nu_0$.

If $\M_\Nfu^\infty(\xi,\nu)$ is identically zero, then we are in the
case $\po_\Nfu(\nu_0)=-\infty$. Otherwise, the term $\M_\Nfu^\infty$
is a linear combination of non-zero elements in a finite number of
one-dimensional spaces $\Mfu_{\Nfu';h_0,p_0,p_0}^{\xi,\nu}$, among
which $\Nfu'=\Nfu$ occurs with a non-zero factor. 
 With
Proposition~\ref{prop-Ps} this  
leads to
\be\label{FoNN} \Four_\Nfu\M_\Nfu(\xi,\nu) \= c\, \Mu_\Nfu(\xi,\nu) +
C_{\Nfu,\Nfu}(\xi,\nu) \Om_\Nfu(\xi,\nu)\,,\ee
with a non-zero factor~$c$. If $\Mu_\Nfu(\xi,\nu_0)$ and
$\Om_\Nfu(\xi,\nu_0)$ are linearly independent, then the total Fourier
term is non-zero, and $\po_\Nfu(\nu_0) = 0$. Moreover, the function
$t\mapsto \Mu_\Nfu(\xi,\nu_0)\bigl(  \am(t) \bigr)$ has exponential
growth, and $P_\Nfu(\nu_0) = \M_\Nfu(\xi,\nu_0)$ is not square
integrable. Linear dependence of the two functions occurs only in the
non-abelian case, when $V$ is of holomorphic or antiholomorphic
discrete series type. See \S\ref{sect-afU}.

Thus, to complete the proof we have to take a closer look at the case
when $V$ is of holomorphic or of antiholomorphic discrete series type.
Then, $\Mu_\n(\xi,\nu_0) $ and $\Om_\n(\xi,\nu_0)$ are proportional,
and we better use the basis family $\Ups_\n(j_\xi,\nu)$ in
\S\ref{sect-afU}. The Fourier term takes the form
\be \Four_\n \M_\n(\xi,\nu) \= a(\nu) \, \Ups_\n(j_\xi,\nu) + b(\nu)\,
\Om_\n(j_\xi,\nu)\,,\ee
with holomorphic functions $a(\nu)$ and $b(\nu)$ in a neighborhood
of~$\nu_0$, with $a(\nu_0)=0$. The function $a(\nu)$ is essentially the
factor in front of $V_{\k,s}$ in \eqref{MinWV}. This implies that the
zero of $a(\nu)$ at $\nu_0$ has order~$1$.

Clearly, $\po_\n(\nu_0)=0$ implies that $b(\nu_0) \neq 0$. If
$\M_\n(\xi,\nu_0)=0$, then $b(\nu_0)=0$ as well. The derivative
$ \partial_\nu \M_\n(\xi,\nu) \bigm|_{\nu=\nu_0}$
is an element of $\Au!(\ps)_{h_0,p_0,p_0}$ (possibly zero). It has
Fourier term of order~$\n$ of the form
\[ a'(\nu_0) \ups_\n^{0,0}(j_\xi,\nu_0) + b'(\nu_0)
\om^{0,0}_\n(j_\xi,\nu_0)\,.\]
Since $a'(\nu_0) \neq 0$, the limit is non-zero, and $\po_\n(\nu_0)=-1$,
the main part $P_\n(\xi,\nu_0) $ is the derivative
$ \partial_\nu \M_\n(\xi,\nu) \bigm|_{\nu=\nu_0}$,
and $\Four_\n P_\n(\xi,\nu_0)$ has an exponentially increasing Fourier
term; see~\eqref{MWest}.
\end{proof}

\subsection{Proofs of the results in~\S\ref{sect-rPs}}\label{sect-pfs}

For this task we will make use of the values of the Wronskian order for
each type of irreducible $(\glie, K)$-module, which are given in
Theorem~\ref{thm:orders}. The proof of the theorem is very technical
and we postpone it to Section~\ref{sect-Wr}.

\begin{thm}\label{thm:orders}
Table~\ref{tab-ord} specifies the values of $\ord_\Nfu(\nu_0)$ for each
type of irreducible $(\glie, K)$-module. When two values are given, the
value between brackets applies to $\nu_0=0$. The empty spaces do not
occur in the Fourier expansion.
\end{thm}
\begin{table}[ht]{\small
\renewcommand\arraystretch{}
\begin{tabular}{|c|c|c|c|c|c|c|}
\hline
$\Nfu$
&u.principal s.&{compl.s.}&{large d.s.}&{holo-aholo d.s.}&$T^\pm_{-1}$
& $T^\pm_k,\, k\ge 0$\\ \hline
{gen.abelian}&1\;(2)&1&0&\;&&
\\
{non-abelian}&1\;(2)&1&0&0\;(1)&1&0\\ \hline
\end{tabular}}
\smallskip
\caption[]{Values of the Wronskian order $\ord_\Nfu(\nu_0)$}
\label{tab-ord}
\end{table}
We now turn to the proofs of the results in ~\S\ref{sect-rPs}.

\begin{proof}[Proof of Proposition~\ref{prop-upcstr}]We have
$h_0=2j_\xi$ and $p_0=0$ in the cases of irreducible unitary principal
series, complementary series and thin representations $T^\pm_{-1}$. The
Wronskian order is positive. By Table~\ref{tab-ord},
$\wo_\Nfu(\nu_0)=1$ if $\nu_0\neq 0$ and $\wo_\Nfu(0)=2$. We consider
Fourier term orders $\Nfu\neq \Nfu_0$.

Proposition~\ref{prop-po} i) implies that if
$\nu\mapsto \M_\Nfu(\xi,\nu)$ has a singularity at $\nu_0$, then
$P_\Nfu(\nu_0)\in \Au{(2)}(\ps)_{2j_\xi,0,0}$. If $\nu_0\neq 0$, the
second statement in Lemma~\ref{lem-complgen} implies that
$c_\Nfu (P_\Nfu(\nu_0)) \neq 0$ and
$\po_\Nfu(\nu_0) = \wo_\Nfu(\nu_0)$.

If there is a singularity at $\nu_0=0$ we have to rule out the
possibility that $\po_\Nfu(0)=1$. %In that case we get 
From \eqref{spf} with $f=P_\Nfu(0)$ we get the relation
\[ \bigl( \tilde\M_\Nfu(\xi,\nu) , P_\Nfu(0) \bigr)_{\Gm\backslash G}
\dis \overline{ c_\Nfu\bigl( P_\Nfu(0)
\bigr)}\,\frac{I_\Nfu(\ph;\nu,\nu_0) } {\nu^2}\,.\]
If $c_\Nfu\bigl( P_\Nfu(0)\bigr)\neq 0$, the right-hand side has order
two by Lemma~\ref{lem-complgen} and Theorem~\ref{thm:orders}. So, the
assumption that $\po_\Nfu(0)=1$ implies that
$c_\Nfu\bigl( P_\Nfu(0) \bigr)$ vanishes, and that
\[0 = \lim_{\nu \rightarrow 0} \bigl( \nu \tilde\M_\Nfu(\xi,\nu) ,
P_\Nfu(0) \bigr)_{\Gm\backslash G} = \bigl\| P_\Nfu(0) \bigr\|^2\,.\]
This implies that $P_\Nfu(0)=0$, a contradiction. Thus $\po_\Nfu(0)=2$.

Conversely, suppose that there exists a square integrable automorphic
form $f\in \Au{(2)}(\ps)_{2j_\xi,0,0}$ satisfying $c_\Nfu(f) \neq 0$.
For such an $f$ we have by Lemma~\ref{lem-complgen}
\[ \bigl( \tilde \M_\Nfu(\xi,\nu), f \bigr) _{\Gm\backslash G}
\;\stackrel.\sim \;
(\nu-\nu_0)^{-\wo(\Nfu(\nu_0)}\qquad (\nu \rightarrow \nu_0) \,,\]
and, since $\po_\Nfu(\nu_0)=\wo(\Nfu(\nu_0)$,
\[ \bigl( (\nu-\nu_0)^{\po_\Nfu(\nu_0)}\tilde \M_\Nfu(\xi,\nu), f \bigr)
_{\Gm\backslash G} \;\stackrel.\sim \;
%(\nu-\nu_0)^{\po_\Nfu(\nu_0)-\wo(\Nfu(\nu_0)} 
1 \qquad
(\nu \rightarrow \nu_0) \,.\]

Since $\wo_\Nfu(\nu_0)>0$, this implies that $ \M_\Nfu(\xi,\nu)$ has a
singularity at $\nu=\nu_0$.

In case $\M_\Nfu(\xi,\nu)$ is holomorphic at $\nu_0$, we use
Proposition~\ref{prop-Minf}. We have the Fourier term
\be \Four_\Nfu \M_\Nfu(\xi,\nu_0) \= a \, \Mu_\Nfu(\xi,\nu) +
C_{\Nfu,\Nfu}(\xi,\nu_0) \, \Om_\Nfu(\xi,\nu_0)\,,\ee
with $a \neq 0$.

The second term on the right is square integrable and the first term has
exponential growth for the isomorphism classes under consideration.
This proves assertion b) in the proposition and completes the proof.
\end{proof}

\begin{proof}[Proof of Proposition \ref{prop-ldstr}]For the large
discrete series type and the thin representations $T^\pm_k$ with
$k\geq 0$, the Wronskian order is $\wo_\Nfu(\nu_0)=0$, according to
Table~\ref{tab-ord}. A singularity of the Poincar\'e family at
$\nu=\nu_0$ would contradict the fact that
$\po_\Nfu(\nu_0) \leq \wo_\Nfu(\nu_0)$, as shown in
Lemma~\ref{lem-complgen}. Further, one shows that $\M_\Nfu(\xi,\nu_0)$
is not square integrable by arguing as in
Proposition~\ref{prop-upcstr}.

The last assertion in the proposition is a direct consequence of
Lemma~\ref{lem-complgen}.
\end{proof}

\begin{proof}[Proof of Proposition \ref{prop-hads}]For the holomorphic
and antiholomorphic discrete series types we need to consider only
non-abelian Fourier term orders $\Nfu=\Nfu_\n$ with $\n=(\ell,c,d)$.
The Wronskian order is $\wo_\n(\nu_0)=0$ for $\nu_0\neq 0$, and
$\wo_\n(0)=1$.

\rmrk{Case $\nu_0\neq 0$}Parts ii)--iv) of Proposition~\ref{prop-po}
show that the Poincar\'e family is holomorphic at $\nu=\nu_0$. If it
has a non-zero value it is square integrable, if it has a zero of order
$1$ then the derivative at $\nu=\nu_0$ is not square integrable. The
last possibility is that the Poincar\'e family is identically zero.

If there are $f\in \Au{(2)}(\ps)_{2j_\xi,0,0}$ with $\Four_\n f\neq 0$,
then we have, by Proposition~\ref{prop-scpf} and \eqref{Wr0}, for each
cut-off function~$\ph$:
\be \Bigl( \tilde \Poin_\n(\xi,\nu) , f\Bigr)_{\Gm\backslash G} \;\sim\;
\overline{c_\n(f)} \frac{-4\pi |\ell|}{\nu^2-\nu_0^2}
\frac{\Gf(1+\nu)}{\Gf\bigl( \frac{1+\nu}2-\k_0\bigr)} \cdot (-1)\,, \ee
as $\nu \rightarrow \nu_0$, using that
$\int_{t=0}^\infty \ph'(t) \, dt=-1$.

Since $\frac{1+\nu_0}2- \k_0 \in \ZZ_{\leq 0}$ the gamma factor in the
denominator has a pole at $\nu_0$, canceling the zero in
$\nu^2-\nu_0^2$. This implies that the limit as $\nu \rightarrow\nu_0$
of the left-hand side $\bigl( \tilde \M_\n(\xi,\nu_0),f \bigr)$ does
not depend on~$\ph$. Letting $\ph=1$ on a growing interval we have that
$\tilde \M_\n(\xi,\nu_0)\rightarrow \M_\n(\xi,\nu_0)$ in the
$L^2$-sense, since $\Mu_\n(\xi,\nu_0) \dis \Om_\n(\xi,\nu_0)$ has quick
decay. Thus we get $\bigl( \M_\n(\xi,\nu_0),f \bigr)_{\Gm\backslash G}
\dis \overline{c_\n(f)}\neq 0$, and hence $\po_\n(\nu_0)=0$.

Conversely, if $\po_\n(\nu_0)=0$, by taking $f=\M_\n(\xi,\nu_0)$ and
arguing as before we see that ${c_\n(f)}\neq 0$. In this way we have
obtained the assertions in part~i) of the proposition.

\rmrk{Case $\nu_0=0$}Now $\wo_\n(0)=1$, hence the singularity of
$\M_\n(\xi,\nu)$ at $\nu=0$ can be of order at most $1$, by
Lemma~\ref{lem-complgen}. At a singularity the main part~$P_\n(0)$ is
square integrable, and $\tilde P_\n(0)=P_\n(0)$ for each cut-off
function~$\ph$. So if $c_\n(f)\neq 0$ for some
$f\in \Au{(2)}(\ps)_{2j_\xi,0,0}$, then $\po_\n(0)$ must be equal to
one.

Assume now $\M_\n(\xi,\nu)$ is holomorphic at $\nu=0$, and not
identically zero. Then $\M_\n(\xi,0) \in \Au{(2)}(\ps)_{2j_\xi,0,0}$,
since $\Mu_\n(\xi,0) \dis \Om_\n(\xi,0)$. Proposition~\ref{prop-scpf},
together with the fact that $\wo_\n(0)=1$, implies that
\[ \bigl( \tilde \M_\n(\xi,\nu), \M_\n(\xi,0) \bigr)_{\Gm\backslash G}
\= \overline{ c_\n \bigl( \M_\n(\xi,0)\bigr)} \frac1\nu \bigl( c_0 +
\oh(\nu) \bigr)\]
as $\nu\rightarrow 0$, with $c_0\neq 0$. The left-hand side is
holomorphic in a neighborhood of $\nu=0$. Hence $c_\n
\bigl( \M_\n(\xi,0)\bigr)=0$. This gives
\[ \bigl( \tilde \M_\n(\xi,0) , \M_\n(\xi,0) \bigr)_{\Gm\backslash
G)}\=0\,.\]
On the other hand, $\tilde \M_\n(\xi,0) \rightarrow \M_\n(\xi,0)$ in
$L^2(\Gm\backslash G)$ as $\ph\uparrow 1$. Hence we have
$\M_\n(\xi,0) = 0$. Then Proposition~\ref{prop-po} leaves only the
possibility that $\po_\n(0)=-1$, and $P_\n(0) $ not square integrable.
This completes the proof.
\end{proof}

\section{Wronskian orders for irreducible
\texorpdfstring{$(\glie,K)$}{(g,K)}-modules.} \label{sect-Wr}
Our last and final task in the paper will be to prove
Theorem~\ref{thm:orders}. We first formulate and prove a general result
on the Wronskian order and next we carry out a separate scrutiny for
each type of irreducible $(\glie,K)$-module in Table~\ref{tab-Vlist}.

\begin{prop}\label{prop-ord}Let $\nu\mapsto \Mu_\Nfu(\xi,\nu)$ and
$\Om_\Nfu(\xi,\nu_0)$ be as above, with
$\nu_0 \in i \RR \cup(0,\infty)$. Then
\begin{enumerate}
\item[i)] If
$\W\bigl(\Mu_\Nfu(\xi,\nu_0), \Om_\Nfu(\xi,\nu_0) \bigr)(t)\neq 0$ for
some $t>0$, then $\ord_\Nfu(\nu_0)=2$ if $\nu_0=0$ and
$\ord_\Nfu(\nu_0)=1$ otherwise.
\item[ii)] Let $\nu_0\neq0$. If
$\,\W\bigl(\Mu_\Nfu(\xi,\nu_0),\Om_\Nfu(\xi,\nu_0)\bigr)(t)$ is
identically zero and, for some $t>0$,
$\bigl\{ \Mu_\Nfu(\xi,\nu_0),\Om_\Nfu(\xi,\nu_0) \bigr\}(t)\neq 0$,
then $\ord_\Nfu(\nu_0)=0$.
\item[iii)]Let $\nu_0=0$, and suppose that
$\W\bigl(\Mu_\Nfu(\xi,0),\Om_\Nfu(\xi, 0 )\bigr)$ is identically zero.
Then the derivative $\partial_\nu \W\bigl(\Mu_\Nfu(\xi,\nu),
  \Om_\Nfu(\xi,0)\bigr)(t)\bigr|_{\nu=0}$ has the form $ a t^3$ for some
$a\in \CC$.
\begin{enumerate}
\item[a)] If $a\neq 0$, then $\ord_\Nfu(0)=1$.
\item[b)] If $a=0$, and
$\bigl\{ \Mu_\Nfu(\xi,0),\Om_\Nfu(\xi,0) \bigr\}(t)\neq 0$ for some
$t>0$, then $\ord_\Nfu(0)=0$.
\end{enumerate}
\end{enumerate}
\end{prop}

\begin{proof}
We first recall some notation. If two quantities $a$ and $b$ differ by a
non-zero factor, we write $a\dis b$. Furthermore, by
$a \stackrel.\sim b$ we indicate that the ratio between $a$ and $b$
tends to a non-zero value.

We get into the proof of the lemma. The integral $I_\Nfu(\ph;\nu,\nu_0)$
in~\eqref{difWrint} is a holomorphic function of $\nu$ in a
neighborhood of~$\nu_0$. If
$\W\bigl( \Mu_\Nfu(\xi,\nu_0),\Om_\Nfu(\xi,\nu_0) \bigr)\not \equiv 0$,
we can find $\al\in C_c^\infty[0,\infty)$ such that
$\nu \mapsto \int_0^\infty \al(t)\, \W\bigl( \Mu_\Nfu(\xi,\nu), \Om_\Nfu(\xi,\nu_0)\bigr)\, t^{-3}\, dt  \neq 0$.
This implies that there is a smooth function $\ph$, such that $\ph=1$
on $[0,T]$ and $\ph=0$ on $[T_1,\infty)$, with $0<T<T_1$, and such that
$I_\Nfu(\ph;\nu,\nu_0)$ is non-zero. This gives part~i).

We now consider ii). Since
$\W\bigl( \Mu_\Nfu(\xi,\nu_0),\Om_\Nfu(\xi,\nu_0) \bigr)$ is
identically zero, then $I_\Nfu(\ph;\nu,\nu_0)$ has a zero at
$\nu=\nu_0$. To establish the order of this zero we denote by
$\dot\Mu_\Nfu$ the derivative
$\frac{d\Mu_\Nfu(\xi,\nu)}{d\nu}\bigr|_{\nu = \nu_0}$. Differentiation
with respect to $\nu$ commutes with the differentiations involved in
the definition of the Maass-Selberg form in~\eqref{MSW2}, hence we get
from \eqref{Crel}
\begin{align*} \MS\bigl( \dot \Mu_\Nfu,\Om_\Nfu(\xi,\nu_0) \bigr) &\=
\tfrac{d}{d\nu}{\bigr|_{\nu = \nu_0}}
\MS\bigl(\Mu_\Nfu(\xi,\nu),\Om_\Nfu(\xi,\nu_0) \bigr)\\
&\= -2\nu_0\, \bigl \{
\Mu_\Nfu(\xi,\nu_0),\Om_\Nfu(\xi,\nu_0)\bigr\}\,.\end{align*}
Now, by Proposition~\ref{prop-MSW} we have
\be \label{dtWr}\partial_ t \bigl( t^{-3}\, \W\bigl( \Mu_\Nfu(\xi,\nu),
\Om_\Nfu(\xi,\nu_0)
\bigr)(t) \bigr)
\= c_1 \,t^{-5}\, \MS(\Mu_\Nfu(\xi,\nu),\Om_\Nfu(\xi,\nu_0)
\bigr)(t)\ee
for some $c_1\neq 0$. Differentiations with respect to $t$ and $\nu$
commute. Thus, by \eqref{dtWr}
\begin{align}
\nonumber
\partial_t \bigl( t^{-3} \,\W \bigl( \dot \Mu_\Nfu,\Om_\Nfu(\xi,\nu_0)
\bigr)(t)
\bigr) \= &
\partial_\nu \partial_t \bigl( t^{-3} \, \W\bigl(
\Mu_\Nfu(\xi,\nu),\Om_\Nfu(\xi,\nu_0)
\bigr)(t)\bigr)
\bigr|_{\nu=\nu_0}\\ \= &
-2\nu_0 \, c_1\, t^{-5}\,\bigl\{ \Mu_\Nfu(\xi,\nu_0),\Om_\Nfu(\xi,\nu_0)
\bigr\}(t)\,. \label{diftdifnu}
\end{align}

If $\nu_0\neq 0$, then the assumption that
$\bigl\{ \Mu_\Nfu(\xi,\nu_0),\Om_\Nfu(\xi,\nu_0)\bigr\}$ is non-zero
implies that
$t\mapsto t^{-3}\, \W\bigl( \dot \Mu_\Nfu,\Om_\Nfu(\xi,\nu_0) \bigr)(t)$
is a real-analytic non-zero function. Hence, by the assumption,
%\red for any \endred cut-off function $\ph$, 
as a function of $\nu$ %$\nu\rightarrow\nu_0$
\begin{align*} \int_{t=0}^\infty \ph'(t)&\,
\W\bigl(\Mu_\Nfu(\xi,\nu),\Om_\Nfu(\xi,\nu_0)\bigr)(t)\,
\frac{dt}{t^3}\\
&\;\sim\; (\nu-\nu_0) \, \int_{t=0}^\infty \ph'(t)\,
\W\bigl(\dot\Mu_\Nfu,\Om_\Nfu(\xi,\nu_0)\bigr)(t)\, \frac{dt}{t^3}
\ddis (\nu-\nu_0) \,, \end{align*}
by \eqref{diftdifnu}, for some choice of cut-off function $\ph$. This
gives part~ii).
\smallskip

Now let $\nu_0=0$. Put
$\ddot\Mu_\Nfu= \partial_\nu^2\Mu_\Nfu(\xi,\nu) \bigr|_{\nu=0}$.
We start with the holomorphic family
$\nu \mapsto \W\bigl( \Mu_\Nfu(\xi,\nu),\Om(\xi,0)\bigr)(t)$. On a
neighborhood of $\nu=0$ it has an expansion
\badl{e0} %\label{expansion}
\W\bigl( \Mu_\Nfu(\xi,&\nu),\Om_\Nfu(\xi,0)\bigr)(t) \= \W\bigl(
\Mu_\Nfu(\xi,0),\Om_\Nfu(\xi,0) \bigr)(t)\\
&\quad\hbox{} + \nu \, \W\bigl( \dot \Mu_\Nfu,\Om_\Nfu(\xi,0)\bigr)(t)
+ \nu^2 \, \W\bigl( \ddot \Mu_\Nfu,\Om_\Nfu(\xi,0)\bigr)(t) +
\oh(\nu^3)\,.
\eadl
Since we are in part iii) of the lemma, this expansion starts with a
zero term. This brings us to
\badl{e1} t^5 \partial_t t^{-3} &\W\bigl(
\Mu_\Nfu(\xi,\nu),\Om_\Nfu(\xi,0)\bigr)(t) \= \nu \, t^5 \partial_t
t^{-3} \W\bigl( \dot \Mu_\Nfu,\Om_\Nfu(\xi,0)\bigr)(t)\\
&\qquad\hbox{}
+ \nu^2 \, t^5 \partial_t t^{-3} \W\bigl( \ddot
\Mu_\Nfu,\Om_\Nfu(\xi,0)\bigr)(t) + \oh(\nu^3)\,.
\eadl
On the other hand, by \eqref{dtWr}
\badl{e2} t^5 \partial_t t^{-3} \W\bigl(
\Mu_\Nfu(\xi,\nu),\Om_\Nfu(\xi,0)\bigr)(t) &\ddis \MS\bigl(
\Mu_\Nfu(\xi,\nu), \Om_\Nfu(\xi,0) \bigr)(t)\\
& \= - \nu^2 \,\bigl\{ \Mu_\Nfu(\xi,\nu), \Om_\Nfu(\xi,0) \bigr\}(t)\\
& \= -\nu^2 \, \bigl\{ \Mu_\Nfu(\xi,0), \Om_\Nfu(\xi,0) \bigr\}(t) +
\oh(\nu^3)\,.
\eadl
Comparison of the expansions in \eqref{e1} and~\eqref{e2} gives
\badl{c1}
\partial_t t^{-3} \W \bigl( \dot \Mu_\Nfu,\Om_\Nfu(\xi,0) \bigr)(t) &\=
0\,,\\
\partial_t t^{-3} \W \bigl( \ddot \Mu_\Nfu,\Om_\Nfu(\xi,0)
\bigr)(t)&\ddis t^{-5}\bigl\{ \Mu_\Nfu(\xi,0),\Om_\Nfu(\xi,0)
\bigr\}(t)\,.
\eadl
The first relation gives
$\W\bigl( \dot\Mu_\Nfu(\xi,0),\Om_\Nfu(\xi,0) \bigr)(t) = at^3$ for
some $a\in \CC$.

If $a\neq 0$, we get \red from \eqref{e0} \endred
\be \int_{t=0}^\infty \ph'(t)\, \W\bigl(\dot\Mu_\Nfu(\xi,\nu)
,\Om_\Nfu(\xi,0)
\bigr) \,\frac{dt}{t^3} \ddis a\nu+ \oh(\xi,\nu^2)\,. \ee
This implies that $\wo_\Nfu(\xi,0)=1$, and gives part iii)~a).

Suppose now that $a = 0$. From the second line in \eqref{c1} we see that
\be \W\bigl( \ddot \Mu_\Nfu, \Om_\Nfu(\xi,0)\bigr) (t) \= t^3 F_\Nfu
(t),\ee
where $F_\Nfu$ is a primitive of
$t \rightarrow t^{-5}\bigl\{ \Mu_\Nfu(\xi,0),\Om_\Nfu(\xi,0) \bigr\}(t).$
The assumptions in iii)~b) imply that $F_\Nfu$ is a non-zero real
analytic function on $(0, \infty)$.

Then, by \eqref{e0} we get
\bad \int_{t=0}^\infty \ph'(t)&\, \W\bigl(\Mu_\Nfu(\xi,\nu)
,\Om_\Nfu(\xi,0)
\bigr) (t) \,\frac{dt}{t^3} \\
&\ddis \nu^2 \int_{t=0}^\infty \ph'(t) \, t^3 F_\Nfu(t)\, {dt} +
O(\nu^3)\,.
\ead
The assumptions in iii)~b) imply that we can choose the cut-off function
in such a way that the integral on the right is non-zero. This gives
part iii)~b).
\end{proof}

\subsection{Proof of Theorem~\ref{thm:orders}} Our next and final task
will be to check the values of $\ord_\Nfu(\nu_0)$ in
Table~\ref{tab-ord} %Theorem~\ref{thm:orders}, 
by applying Proposition~\ref{prop-ord}. We have to handle the various
isomorphism classes in Table~\ref{tab-Vlist} for the generic abelian
cases $\Nfu=\Nfu_\bt$, $\bt\neq0$, and for the non-abelian cases
$\Nfu=\Nfu_\n$. Some cases will require quite long computations and
references to~\cite{BMSU21}. Namely, we will consider:
\begin{itemize}
\item \emph{Unitary principal series and complementary series } in
\S\ref{d1a} and \S\ref{d1na}
\item \emph{Holomorphic and antiholomorphic discrete series types}
mostly in \S\ref{d1na}. Some computations for the case $\nu_0=0$
(limits of discrete series) in \S\ref{d1na-app}.
\item \emph{Large discrete series type} in \S\ref{hdna-app}.
\item \emph{Thin representations: } $T^\pm_{-1}$ in \S\ref{d1na},
$T^\pm_k$ with $k\geq 0$ in \S\ref{hdna-app}.
\end{itemize}

\subsubsection{One-dimensional minimal $K$-types, generic abelian
cases}\label{d1a}
Let $p_0=0$. Then $h_0=2j_\xi$, and
\begin{align*} \Mu_{\bt;2j_\xi,0,0}(\xi,\nu) \bigl( n \am(t)k \bigr)
&\= \ch_\bt(n) \, t^2 I_\nu(2\pi |\bt|t) \, \Kph{2j_\xi}000(k)\,,\\
\Om_{\bt;2j_\xi,0,0} (\xi,\nu_0) \bigl( n \am(t)k \bigr)
&\= \ch_\bt(n) \, t^2 K_{\nu_0}(2\pi |\bt|t) \, \Kph{2j_\xi}000(k)\,,
\end{align*}
which leads with \eqref{Wr0} to
\be \W\bigl( \Mu_\bt(\xi,\nu_0) , \Om_\bt(\xi,\nu_0) \bigr)(t)
\= - t^3\,. \ee
We use that $\nu_0\in (0,\infty) \cup i \RR$, and that $K_\nu$ is even
in $\nu$, to get rid of the complex conjugation.

Proposition~\ref{prop-ord} i) gives $\wo_\bt(\nu_0) = 1$ if
$\nu_0\neq 0$ and $\wo_\bt(0)=2$. This concerns the unitary principal
series and the complementary series.

\subsubsection{One-dimensional minimal \texorpdfstring{$K$}{K}-types,
non-abelian cases}\label{d1na}
Again $p_0=0$, $h_0=2j_\xi$. $\Om_\n(\xi,\nu_0)=\om_\n^{0,0}(j_\xi,\nu)$
and $\Mu_\n(\xi,\nu) = \mu_\n^{0,0}(\xi,\nu)$ in \eqref{om-mu-nab}.
From~\eqref{Wr0} we obtain
\be \W\bigl( \Mu_\bt(\xi,\nu_0) , \Om_\bt(\xi,\nu_0) \bigr)(t)
\= - 4\pi |\ell| \,\Gf(\nu_0+1) \, t^3 \, \bigm/ \Gf\bigl(
\tfrac{\nu_0+1}2-\k\bigr)\,.\ee
Since $\re\nu_0\geq 0$, zeros may be due only to the gamma factor in the
denominator. 
In \eqref{m0kap} we see that
$\k = - m_0 -\frac12 - \frac12 j_\xi\,\sign(\ell)$, with a non-negative
integer $m_0$. In the cases under consideration we use 
\be \frac{\nu_0+1}2-\k \= m_0 +1 + \frac12 \bigl( \nu_0 +
j_\xi\sign(\ell) \bigr)\,.\ee

The present assumption concerns the following isomorphism classes:
\begin{itemize}
\item Isomorphism class $\II(j,\nu_0)$, \emph{unitary principal series
and complementary series,} with $\nu_0 \not \equiv j_\xi \bmod 2$ or
$(j_\xi,\nu_0)=(0,0)$. The quantity $W(\nu_0)$ is non-zero, and we have
$\wo_\n(\nu_0)=1$ if $\nu_0 \neq 0$, and $\wo_\n(0)=\red 2 \endred$.

\item Isomorphism class $\IF(j_\xi,\nu_0)$, \emph{holomorphic discrete
series type, } and isomorphism class $\FI(j_\xi,\nu_0)$,
\emph{antiholomorphic discrete series type,} with
$\nu_0 \equiv j_\xi\bmod 2$, $0 \leq \nu_0 \leq |j_\xi|-2$.

Table~23 in \cite{BMSU21} shows that it occurs in
$\Wfu_\n^{\ps[j_\xi,\nu_0]}$ under the conditions
\badl{IFFIcond} \text{ for }\IF(j_\xi,\nu_0) : \quad& \ell<0,\; j_\xi >
0,\;0 \leq m_0 <\frac12(j_\xi-\nu_0)\,,\\
\text{ for }\FI(j_\xi,\nu_0):\quad&\ell>0,\; j_\xi<0,\; 0 \leq m_0 <
-\frac12(j_\xi+\nu_0)\,.
\eadl
This gives in both cases $\frac{\nu_0+1}2-\k \in \ZZ_{\leq 0}$, and
vanishing of the Wronskian function at $\nu=\nu_0$. We cannot apply
Proposition~\ref{prop-ord}, i).

\item \emph{Thin representations } $T^+_{-1}$ (type $\IF(1,-1)$) and
$T^-_{-1}$ (type $(\IF(-1,-1)$), with $\nu_0=1$, $j_\xi=\pm 1$. It
occurs in $\Wfu_\n^{\ps[\pm 1,1]}$ under the condition $m_0 =0$,
$\mp\ell>0$. We obtain a non-zero Wronskian function, and can apply
part i) in Proposition~\ref{prop-ord} to conclude that $\wo_\n(1)=1$.
\end{itemize}

We proceed with the holomorphic and antiholomorphic series types. Since
$\W\bigl( \Mu_\bt(\xi,\nu_0) , \Om_\bt(\xi,\nu_0) \bigr)=0$, the
components $ t\,M_{\k,\nu_0/2}(2\pi|\ell|t^2)$ and
$\, W_{\k,\nu_0/2}(2\pi |\ell|t^2)$ are proportional,
\be \bigl\{\Mu_\Nfu(\nu_0),\Om_\Nfu(\nu_0) \bigr\}(t) \ddis t^2\,
W_{\k_0,\nu_0/2}(2\pi|\ell|t^2)^2\ee
is non-zero. This gives $\wo_\n(\nu_0)=0$ if $\nu_0\neq 0$, by
Proposition~\ref{prop-ord}, ii). See \S\ref{d1na-app} for the remaining
case $\nu_0=0$.

\subsubsection{Higher-dimensional minimal K-types, generic abelian
case}\label{hda}Here only the large discrete series type
$\II_+(j_\xi,\nu_0)$ has to be considered, with the basis functions as
in~\eqref{hdWa} and~\eqref{hdMa}.

The Wronskian function has the following value at $\nu=\nu_0$.
\begin{align*}
\W\bigl( \Mu_\bt&(\xi,\nu_0), \Om_\bt(\xi,\nu_0) \bigr)
\ddis \sum_r \binom {p_0}{\frac{p_0+r} 2} \ \Wr\biggl( \Bigl(
\frac{-i\bt}{|\bt|} \Bigr)^{(r+p_0)/2}\, t^{2+p_0}\,
I_{|h_0-r|/2}(2\pi|\bt|t),\\
&\qquad\hbox{} \overline{ \Bigl( \frac{i \bt}{|\bt|}\Bigr)^{(r+p_0)/2}\,
t^{2+p_0} \, K_{|h_0-r|/2}(2\pi|\bt|t) } \biggr)\= t^{2p_0}\,
\sum_{n=0}^{p_0} \binom{p_0}n \, (-1)^n (-t^3) \= 0\,.
\end{align*}
Here we have used~\eqref{Wr0}.

We have to turn to Proposition~\ref{prop-ord} ii).
\bad \bigl\{ \Mu_\bt(\xi,&\nu_0) \, \Om_\bt(\xi,\nu_0) \bigr\} \ddis
\sum_{n=0}^{p_0} \binom{p_0}n\, (-1)^n\, t^{4+p_0}\\
&\qquad\hbox{} \cdot
I_{| (h_0+p_0)/2-n|}(2\pi|\bt |t) \, K_{|(h_0+p_0)/2-n|}(2\pi|\bt |t)
\,.
\ead
In a higher-dimensional minimal $K$-type $\tau^{h_0}_{p_0}$ we have
$h_0=-j_\xi$ and $p_0=\nu_0$. So $m:=\frac{h_0+p_0}2-n$ runs through
the range of integers from $-\frac{j_\xi+\nu_0}2 $ to
$\frac{\nu_0-j_\xi}2$. Since $\nu_0\geq j_\xi$, the value $m=0$ occurs
in this range. With~\eqref{KI-expint} we have the following asymptotic
behavior as $t\downarrow0$
\be I_{|m|}(2\pi|\bt| t)\, K_{|m|}(2\pi|\bt| t) \;\sim\;\begin{cases}
- \log\pi |\bt|t & \text{ if }m=0\,,\\
\frac1{2|m|}& \text{ otherwise}\,.
\end{cases}\ee
For $t$ near zero, the (non-zero) term with $m=0$ is larger than the
other terms. Hence we have that $\wo_\bt(\nu_0)=0$ by ii) in
Proposition~\ref{prop-ord}.

\subsubsection{Limits of holomorphic and antiholomorphic discrete
series}\label{d1na-app}For the isomorphism classes $\IF(j_\xi,0)$ and
$\FI(j_\xi,0)$ the Wronskian function is zero if $\nu$ takes the value
$\nu_0=0$. To apply Proposition~\ref{prop-ord} iii) we need to compute
the derivative
$\partial_\nu \W\bigl( \Mu_\n(\xi,\nu),\Om_\n(\xi,0) \bigr)(t)\bigr|_{\nu=0}$.

In this situation we have
$\frac{1+\nu}2 - \k = m_0 + \frac{\nu + \e j_\xi}2$ with
$\e=\sign(\ell)$, $-\e j_\xi\in 2\ZZ_{\geq 1}$, and
$0\leq m_0 < -\frac \e 2 j_\xi$. See \cite[Table 23]{BMSU21}. This
implies that $\nu \mapsto \frac{1+\nu}2-\k $ has values in
$\ZZ_{\leq -1}$ as $\nu=\nu_0=0$. The zero of the Wronskian in (C.11)
at $\nu=0$ implies that $\mu_\n(\xi,0)$ and $\Om_\n(\xi,0)$ are
proportional. Hence, the holomorphic function
$\nu \mapsto \Gf\bigl(\frac{1+\nu}2-\k)\, \W\bigl( \Mu_\n (\xi,\nu),\Om_\n(\xi,0) \bigr)(t)$
has a zero at $\nu=0$, and we apply Proposition~\ref{prop-ord} iii). We
use the following proportionality
\be \partial \W\bigl( \Mu_\n (\xi,\nu),\Om_\n(\xi,0)
\bigr)|_{\nu=0} \simeq \lim_{\nu\rightarrow 0}\Gf\bigl(\tfrac{1+\nu}2
-\k\bigr) \, \Wr\bigl(\Mu_\n (\xi,\nu),\Om_\n(\xi,0)\bigr)\,.\ee
For $\nu $ near $0$, $\nu \neq 0$, and the abbreviation
$2\pi|\ell| = \al$, we have
\begin{align*}
\Gf\bigl(\tfrac{1+\nu}2&-\k) \, \W\bigl( \Mu_\n (\xi,\nu),\Om_\n(\xi,0)
\bigr)(t)
\= \Gf\bigl(\tfrac{1+\nu}2 -\k) \, \Wr\bigl( t\, M_{\k,\nu/2}(\al t^2),
t\, W_{\k,0}(\al t^2) \Bigr)_t
\displaybreak[0]\\
&\= \Gf\bigl(\tfrac{1+\nu}2 -\k) \, 2 \al t^3 \, \Bigl( M_{\k,\nu/2}(\al
t^2) \, W_{\k,0}'(\al t^2) - 
 M_{\k,\nu/2}'
(\al t^2)\,W_{\k,0}(\al t^2)
\displaybreak[0]\\
& \= \Gf\bigl(\tfrac{1+\nu}2 -\k) \, 2\al t^3 \Bigl( - \frac1{\al t^2}
M_{\k,\nu/2}(\al t^2)\, W_{\k+1,0}(\al t^2) \\
&\qquad\hbox{} - \frac{ 1+2\k+\nu}{2\al t^2} M_{\k+1,\nu/2}(\al t^2)
W_{\k,0}(\al t^2)
\Bigr)\,,
\end{align*}
where we have used the contiguous relations in \cite[(A.18)]{BMSU21}.

Next we determine the asymptotic behavior as $t\rightarrow \infty$,
using the formulas in~\eqref{MWest}. The main term in the expansion of
the first summand is
\[ -2 t \Gf(1+\nu) \, (\al t^2)^{-\k+\k+1} \= -4\pi|\ell| \Gf(1+\nu) \,
t^3\,. \]
The second  term is of size $\oh(t^{-1})$. 
Hence the
value of $a$ in Proposition \ref{prop-ord} iii) equals $a=-4\pi|\ell|$,
thus $\wo_\n(0)=1$.

\subsubsection{Higher-dimensional minimal K-types, non-abelian
cases}\label{hdna-app}
We still have to consider the large discrete series type and the thin
representations $T^\pm_k$ with $k\in \ZZ_{\geq 0}$ in the non-abelian
cases. The minimal $K$-type $\tau^{h_0}_{p_0}$ has dimension $p_0+1>1$,
for which we have the explicit descriptions \eqref{hdWna} and
\eqref{hdMna} of the basis functions $\Om_\n$ and $\Mu_\n$. These
expressions depend on $\k$ and $s$ in~\eqref{parms-nab}. We collect in
Table~\ref{tab-i} information from Tables 15, 23, and equation (16.6)
in~\cite{BMSU21}, and from Table~\ref{tab-Vlist} above.
\begin{table}[ht]
\[
\begin{array}{|c|c|cc|}\hline
&\text{large discr. s.t.}&T^+_k & T^-_k \\\hline
\text{type} & II_+(j_\xi,\nu_0) & \IF_+(j_\xi,-1) & \FI_+(j_\xi,-1) \\
\nu_0 &\nu_0\geq 1,\; \nu_0\geq |j_\xi|
& 1 & 1
\\
j_\xi & j_\xi \equiv \nu_0 \bmod 2 & 2k+3
& -(2k+3)
\\
h_0 & -j_\xi & k+3 & -(k+3)
\\
p_0 & \nu_0 & k+1 & k+1 \\
\e=\sign(\ell) & \e\in \{1,-1\} & -1 & 1
\\
m_0 & \hbox{} \geq \frac{\nu_0 - \e j_\xi} 2 & k+1 & k+1
\\
m(h_0,r) & m_0 + \frac\e 2(r+j_\xi) & \frac{k+1-r}2 & \frac{k+1+r}2
\\
\k(r) &- m_0 - \frac \e 4(r+j_\xi) - \frac12&
- \frac{k+1-r} 4 & - \frac{k+1+r}4
\\
s(r)= s(h_0,r) &-\frac14(j_\xi+r) &\frac{k+3-r} 4&
- \frac{k+3+r}4
\\ \hline
\end{array}
\]
\caption{Parameters for the cases of large discrete series type and thin
representations.}\label{tab-i}
\end{table}
In \eqref{hdWna} and \eqref{hdMna} the summation variable runs over
$r\equiv p_0\bmod 2$, $|r|\leq p_0$, and $m(h_0,r) \geq 0$. For the
large discrete series type this gives no restriction, since
$m(h_0,r) \geq \frac{\nu_0+ \e r} 2 = \frac{p_0+\e r}2\geq 0$. For the
thin representations the same holds:
$m(h,r) = \frac{k+1\mp r}2 = \frac{p_0\mp r}2\geq 0$.

First we consider the Wronskian function at $\nu = \nu_0$.
\begin{align*}
\W\bigl(\Mu_\n(\xi,&\nu_0), \Om_\n(\xi,\nu_0) \bigr)
\= \sum_r \binom{p_0}{\frac{p_0+r}2}\, t^{2p_0} \,(-e^{\pi
i(m(h_0,r)-\k(r)}) \, \frac{\Gf\bigl( \frac12+|s(r)| - \k(r) \bigr)}
{\sqrt{m(h_0,r)!}\,(2|s(r)|)!} \\
&\qquad\hbox{} \cdot
(-i)^{m(h_0,r)} \sqrt{m(h_0,r)}\, \Wr\bigl(t\,
M_{\k(r),|s(r)|}(2\pi|\ell|t^2) ,\, t\, W_{\k(r),|s(r)}(2\pi|\ell|t^2)
\bigr)\,.
\end{align*}
For all terms in the sum we have
$\frac12+|s(r)|-\k(r) \geq 1+m_0\geq 1$.

The factor $-e^{\pi i(m(h_0,r)-\k(r))}$ is invariant under
$r\mapsto r+2$, and can be taken outside the sum as a non-zero
constant. The factor $(-1)^{m(h_0,r)}$ can be replaced by
$(-1)^{(p_0+r)/2}$ up to a non-zero constant. With \eqref{Wr0} we
arrive at
\begin{align*}
\W\bigl(\Mu_\n(\xi,&\nu_0), \Om_\n(\xi,\nu_0) \bigr)
\ddis t^{2p_0}\, \sum_r \binom{p_0}{\frac{p_0+r}2}\, (-1)^{(p+r_0)/2}\\
&\qquad\hbox{} \cdot \frac{\Gf\bigl( \frac12+|s(r)| - \k(r) \bigr)}
{(2|s(r)|)!} \frac{\Gf\bigl( 2|s(r)|+1)\,t^3}{\Gf\bigl(
\frac12+|s(r)|-\k(r)\bigr)}
\displaybreak[0]\\
&\ddis t^{2p_0} \sum_{n=0}^{p_0} \binom {p_0}n \, (-1)^n \, t^3 \= 0\,.
\end{align*}
This brings us to part ii) of Proposition~\ref{prop-ord}.

With similar simplifications we obtain
\bad \bigl\{ \bigl(\Mu_\n(\xi,&\nu_0), \Om_\n(\xi,\nu_0) \bigr\}(t)
\ddis \sum_r \binom{p_0}{\frac{p_0+r}2}\, t^{2p_0+2}
\,(-1)^{(p_0+r)/2}\\
&\qquad\hbox{} \cdot
\frac{\Gf\bigl( \frac12+|s(r)| - \k(r) \bigr)} {(2|s(r)|)!}\,
M_{\k(r),|s(r)|}(2\pi|\ell|t^2) \, W_{\k(r),|s(r)|}(2\pi |\ell|t^2)\,.
\ead
For the behavior as $t\downarrow 0$ of the separate terms we use
\eqref{MWest} and \eqref{West0}. If $s(r)\neq 0$ we find
\bad&\hbox{} \sim \; t^{2p_0+2}\, \binom{p_0}{\frac{p_0+r}2}
\,(-1)^{(p_0+r)/2}\, \frac{\Gf\bigl( \frac12+|s(r)| - \k(r)
\bigr)}{(2|s(r)|)!}\\
&\qquad\hbox{} \cdot
\, \frac{\Gf\bigl(2|s(r)|\bigr)}{\Gf\bigl( \frac12+|s(r)|-\k(r)\bigr)}
2\pi|\ell|t^2\\
&\= t^{2p_0+2}\,
\binom{p_0}{\frac{p_0+r}2}\frac{(-1)^{(p_0+r)/2}}{2|s(r)|}\,
2\pi|\ell|t^2\,.
\ead
For $s(r)=0$ we have
\be t^{2p_0+2} \binom{p_0}{\frac{p_0+r}2} \, \,(-1)^{(p_0+r)/2}\,
\frac{-1}{\Gf\bigl( \frac12-\k(r)\bigr)} 2\pi|\ell|t^2 \log2\pi|\ell|
t^2\,.\ee
If $s(r)=0$ occurs in the sum, then the corresponding term has the
largest size near $t=0$, and by ii) we thus get $\wo_\n(\nu_0)=0$. We
next verify that this is the case.

We have $-p_0\leq r \leq p_0$. For the large discrete series type
\[ s(-p_0) \= \frac14(- j_\xi+ p_0 ) \= \frac14(-j_\xi+\nu_0)\geq 0,
\qquad s(p_0) \= \frac14(-j_\xi-\nu_0)\leq 0\,.\]
Since $h_0\equiv p_0 \equiv r\bmod 2$, the values of $2s(r)$ are
integral. Thus the value $s(r)=0$ occurs, and we get $\wo_\n(\nu_0)=0$
in this case.

In the case of $T^\pm_k$ we have
$$ \pm s(r) \= \frac{k+3\mp r}4 \geq \frac{k+3-p_0}4 \= \frac12\,.$$
So $s(r)$ does not  have the value $0$ in the sum over~$r$, and all 
terms in the sum  have an asymptotic expansion with a main term 
that is a non-zero constant times $t^{2p_0+4}$.
We have to 
determine whether the following quantity is zero or non-zero.
\[ 2\pi|\ell| t^{2p_0+4 } \sum_r \binom{p_0}{\frac{p_0+r}2}\;
(-1)^{(p_0+r)/2}\,\frac2{k+3\mp r}\,.\]
With the substitution
$r=\mp 2n\pm p_0$ we get, up to a non-zero factor,
\begin{align*} \sum_{n=0}^{p_0} &\binom{p_0} n \;(-1) ^n\, \frac1{n+1} 
\= \sum_{n=0}^{p_0}  
\frac{ p_0! \; (-1)^n}{n!\; (p_0-n)! \;(n+1)}
\\
&
\= \sum_{n=0}^{p_0+1} \binom{p_0+1}{n+1}\, \frac{(-1)^n}{p_0+1} - 1\,
\frac{(-1)^{p_0+1}}{p_0+1} 
\= \frac{(1-1)^{p_0+1}}{p_0+1} +
 \frac{(-1)^{p_0}}{p_0+1}  \neq0\,.
\end{align*}
 So the value of
$\bigl\{ \bigl(\Mu_\n(\xi,\nu_0), \Om_\n(\xi,\nu_0) \bigr\}(t)$ is
non-zero for $t$ sufficiently close to~$0$. Therefore, this implies
that $\wo(0)=0$ for the thin representations $T^\pm_k$, with $k\geq 0$,
as well.

In this way we have taken care of the last case in the list of
isomorphism classes, thus the proof is now complete.

\clearpage 
% \footnote{\purple We start the appendix on a new page. I
% wonder whether that is necessary. \endpurple I guess the journal will decide on this.}

\appendix
\begin{center}\large\bf Appendix\end{center}
In this appendix we give further explanations and computations, mostly
based on~\cite{BMSU21}. In this way we avoid in the main text frequent
references to that paper.

\section{ Lie algebra and subgroups of
\texorpdfstring{$\SU(2,1)$}{SU(2,1)}}\label{app-group}

\subsection {Subalgebras and Casimir element.}\label{app-Lie}
We let $\glie$ be the real Lie algebra of $G=\SU(2,1)$, and let
$\glie_c$ be the complex Lie algebra $\CC\otimes_\RR \glie$. For Lie
algebras of subgroups we use a similar notation.

We fix a basis of $\glie$ by indicating bases for the summands in
$\glie=\nlie\oplus\alie\oplus \klie$:
\begin{align}\label{nlie}\nlie&\=\RR\,
\XX_0\oplus\RR\,\XX_1\oplus\RR\,\XX_2\,,\\
\nonumber
e^{t\XX_0}&\= \nm(0,t/2)\,,\quad e^{t\XX_1} \= \nm(t,0)\,,\quad
e^{t\XX_2}\= \nm(it,0)\,; \displaybreak[0]\\
\label{alie}\alie&\= \RR\, \HH_r\,,\quad e^{t\HH_r}\= \am(e^t)\,;
\displaybreak[0]\\
\label{klie}\klie&\= \RR\,\HH_i\oplus\RR\,\WW_0\oplus\RR\,\WW_1
\oplus\RR\,\WW_2\,,\\
\nonumber
e^{t\HH_i}&\= \mm(e^{it})
\,, \quad e^{t\WW_0} \= \mathrm{diag}\bigl\{ e^{it}, e^{-it},1\bigr\}
\,,
\\
\nonumber
e^{t\WW_1}&\= \begin{pmatrix}\cos t&\sin t& 0\\
- \sin t& \cos t & 0\\
0&0&1\end{pmatrix}
\,,\quad e^{t\WW_2}\= \begin{pmatrix}\cos t&i\sin t& 0\\
i\sin t& \cos t & 0\\
0&0&1\end{pmatrix}
\,. \end{align}

The Lie algebra of $M$ is $\mlie=\RR\, \HH_i$. The Lie algebra $\glie$
is a real form of $\mathfrak{sl}_3$, of type~$\mathrm{A}_2$. A Cartan
algebra contained in $\klie_c$ is $\CC\, \CK_i \oplus \CC\,\WW_0$,
where $\CK_i = 3\WW_0- 2\HH_i$. It has root spaces spanned by
\badl{Zij}
&\text{in $\klie_c$:}& \quad \Z_{12}&\= \WW_1-i \WW_2\,,\qquad\quad
\Z_{21}\=\WW_1+i\WW_2\,;\\
&\text{not in $\klie_c$:}& \Z_{13}&\= \frac12 \HH_r + i \,\Bigl(
\XX_0-\frac14\WW_0-\frac14 \CK_i\Bigr)\,,\\
& &\Z_{31}&\= \frac12 \HH_r -i \,\Bigl( \XX_0-\frac14\WW_0-\frac14
\CK_i\Bigr)\,,
\\
&&\Z_{23}&\= \frac12\Bigl( \XX_1-\WW_1+i \XX_2-i \WW_2\Bigr)\,,\\
&& \Z_{32}&\= \frac12\Bigl( \XX_1-\WW_1-i \XX_2+i \WW_2\Bigr)\\
\eadl
with corresponding roots
\bad \al_{12}(\CK_i) &\= 0&\quad\al_{13}(\CK_i)&\= 3i&
\quad \al_{23}(\CK_i)&\=3i
\\
\al_{12}(\WW_0) &\=2i&\al_{13}(\WW_0)&\=i& \al_{23}(\WW_0) &\= -i\\
\al_{ji}&\=-\al_{ij} \ead

The center of the enveloping algebra can be generated by two independent
elements. One of them is the
%\il{Co}{Casimir operator}
Casimir operator given by $\sum_j X_j X_j^\ast$, where $\{X_j\}$ is a
basis of $\glie_c$ and $\{X_j^\ast\}$ the dual basis 
 with
respect to a suitable multiple of the Killing form.

For the two bases discussed above this leads to %\ir{Cas}{C}
\begin{align}\label{Cas}
C &\= \HH_r^2-4\HH_4-\frac13\HH_i^2+4\XX_0\HH_i-8\XX_0\WW_0+4\XX_0^2
-2\XX_1\WW_1\\
\nonumber
&\qquad\hbox{}
+\XX_1^2-2\XX_2\WW_2+\XX_2^2
\displaybreak[0]\\
\label{CasZ}
&\= \frac13\CK_i^2 +2i\CK_i-\WW_0^2+2i \WW_0 -\Z_{12}\Z_{21}+4
\Z_{13}\Z_{31}+4\Z_{23}\Z_{32}\,.
\end{align}

We take as second generator the element $\Dt_3$ given in
\cite[Proposition 3.1]{GPT02}.

\subsection{Realization of irreducible representations
of~\texorpdfstring{$K$}{K}} \label{app-K}
We define polynomial functions $\Kph hprq$ on $K$ by the identity
\bad \sum_{r\equiv p(2), |r|\leq p} \Kph hprq(k) \, x^{(p-r)/2}
&\=\dt^{(h+p)/2}\, ( ax+c)^{(p-q)/2}\, (bx+d)^{(p+q)/2}\,,\\
\text{ with\, } k &\= \begin{pmatrix} a&b&0\\ c&d&0\\ 0&
0&\dt\end{pmatrix}\,.
\ead

These functions form an orthogonal system in $L^2(K)$. It is not
orthonormal.

Translation by the central elements $\exp(i t \CK_i)$ multiplies
$\Kph hprq$ by $e^{-i h t}$. Right (resp. left) translation by
$e^{t\WW_0}$ multiplies $\Kph h p r q$ by $e^{-iq t}$ (resp. by
$e^{- i r t}$). We call $q$ the weight of $\Kph h p r q$. Letting $r$
run over $-p, -p+2,\cdots, p-2,p$ we get orthogonal realizations of the
irreducible representations $\tau^h_p$ in $L^2(K)$ under the action of
right translation, on the spaces
\be \bigoplus _{q\equiv p \bmod 2,\; |q|\leq p} \CC\,\Kph h p r q \,.\ee
The norms of $\Kph h p r q$ satisfy %these functions satisfy
\be\label{Kph-norm}
\|\Kph hp r q \|^2_K = \frac{p! } {\bigl( \frac{p+r}2\bigr)!\; \bigl(
\frac{p-r}2\bigr)!} \|\Kph h{p}{p}{q}\|^2_K\,. \ee

The Lie algebra $\klie$ acts by right differentiation:
\badl{Rk} \CK_i\Kph h{p}{r}{q}&\=- ih \Kph h{p}{r}{q}\,,\\
\WW_0 \Kph h{p}{r}{q}&\= -iq \Kph h{p}{r}{q}\,,\\
(\WW_1\pm i \WW_2)\Kph h{p}{r}{q} &\= (q\mp p)\Kph h{p}{r}{q\pm2} \,.
\eadl
We refer to \cite[\S3]{BMSU21} for more details.

\subsection{Theta functions}\label{app-Theta}
We indicate how theta functions in~\eqref{Thla} arise from the
Schr\"odinger representation.

Under the bilinear form
$\;\;[\ph_1,\ph_2]_\RR \= \int_\RR \ph_1(\xi)\, \ph_2(\xi)\, d\xi, $
the Schr\"odinger representations $\pi_{2\pi\ell}$ and
$\pi_{-2\pi\ell}$ in \eqref{pild} are dual to each other. One can
extend this relation to a relation between Schwartz functions and
tempered distributions. Let the delta-distribution $\dt_c$ be defined
by $\delta_c(\ph)=\ph(c)$ for $\ph \in \Schw(\RR)$. Let
$\s \in\ZZ_{\ge 1}$. For the standard generators of $\Ld_\s$ we have:
\bad \pi_{-2\pi\ell}\bigl( \nm(0,2/\s) \bigr)\, \dt_c \, =
&\,\,e^{-4\pi i\ell/\s}\, \dt_c\,,\\
\pi_{-2\pi\ell}\bigl( \nm(1,0)
\bigr) \, \dt_c =\, e^{4\pi i \ell c}\, \dt_c,\,\,&\,\, \pi_{-2\pi\ell }
\bigl( \nm(i,0)\bigr) \, \dt_c \= \dt_{c-1}\,. \ead
This implies that if we take $\ell \in \frac \s2\ZZ$,
$c\in \ZZ\bmod 2|\ell|$, we have invariant distributions
\be\label{mula} \mu_{\ell,c} \= \sum_{k\in \ZZ} \dt_{k+c/2\ell}\,.\ee

Each such $\pi_{-2\pi\ell}(\Ld_\s)$-invariant tempered distribution
$\mu = \mu_{\ell,c}$ on $\RR$ gives rise to a theta-function by the
formula %\il{Thf}{theta-function}
\be\label{Thmu} \Th_\mu(\ph)(n)\= [\pi_{2\pi\ell}(n)\ph ,\mu]_\RR\qquad
(\ph \in \Schw(\RR))\,, \ee
yielding left-$\Ld_\s$-invariant functions on $N$. In this way we have
the operator $\ph \mapsto \Th_{c,\ell}(\ph)$ given by
\be \Th_{\ell,c}(\ph)\bigl( \nm(x+iy,r)\bigr) \= \sum_{k \in \ZZ} \bigl[
\ph, \pi_{-2\pi\ell}(\nm(-x-iy,-r))\, \dt_{c/2\ell+k}\bigr]_\RR\,,\ee
which expands to~\eqref{Thla}. It is an intertwining operator:
\bad \Th_{\ell,c}(\ph)(n n_1) &\= \Th_{\ell,c}\bigl( \pi_{2\pi
\ell}(n_1) \ph\bigr)(n)&&(n.n_1\in N)\,,\\
X\, \Th_{\ell,c}(\ph) &\=
\Th_{\ell,c}\bigl(d\pi_{2\pi\ell}(X)\ph\bigr)&\qquad&
(X\in \nlie)\,.
\ead

Furthermore, for $\ell,\ell'\in \frac \s2\ZZ_{\neq 0}$, $c,c'\in \ZZ$,
$\ph,\ph'\in \Schw(\RR)$, one has the orthogonality relations

\be\label{Th-un}\bigl(
\Th_{\ell,c}(\ph),\Th_{\ell',c'}(\ps)\bigr)_{\Ld_\s\backslash N} \=
\begin{cases}
\frac 2\s ( \ph,\ps)_\RR & \text{ if }\ell'=\ell,\; c'\equiv c\bmod
2|\ell|\,,\\
0&\text{ otherwise}\,.
\end{cases}
\ee

So for each couple $(\ell,c)$ of integers with
$\ell \in \frac \s2\ZZ_{\neq 0}$, $0\leq c \leq 2|\ell|-1$, the linear
map $\ph \mapsto \sqrt {\s/2} \,\Th_{\ell,c}(\ph)$ from $\Schw(\RR)$ to
% \rightarrow 
$C^\infty_\ell(\Ld_\s\backslash N)$, extends to a unitary map
$L^2(\RR) \rightarrow L^2(\Ld_\s\backslash N)$. The images for
different choices of $(\ell,c)$ are orthogonal to each other.

\rmrk{Further transformation properties} The distribution $\mu_{\ell,c}$
in~\eqref{mula} satisfies
\badl{dtth} \pi_{-2\pi\ell}\bigl( \nm(1/2\ell,0)\bigr)\, \mu_{\ell,c}&\=
e^{\pi i c/\ell}\,\mu_{\ell,c},\\
\pi_{-2\pi\ell}\bigl( \nm(i/2\ell,0) \bigr)\, \mu_{\ell,c} &\=
\mu_{\ell,c-1}\,. \eadl
This implies
\bad \Th_{\ell,c}( \ph) (\nm(1/2\ell,0)n)
&\= e^{\pi i c /\ell}\, \Th_{\ell,c}(\ph)(n)\,,\\
\Th_{\ell,c}( \ph)(\nm(i/2\ell,0)n)& \= \Th_{\ell,c+1}(\ph)(n)\,,
\ead
which go beyond the left invariance under $\Ld_\s$.
\medskip

The group $\Gm$ that we have fixed in~\S\ref{discretesubgroups} contains
the element $\mm(i)$, which normalizes the unipotent group $N$, and
also the lattice $ \Gm\cap N$. Indeed, we have that
$\mm(-i) \nm(b,r) \mm(i) = \nm(ib ,r)$ and furthermore (see
\cite[Proposition 4.2]{BMSU21}):
\begin{prop}\label{prop-actmi} Let $m\in \ZZ_{\geq 0}$,
$\s\in \ZZ_{\geq 1}$, and $\ell\in \frac \s2\ZZ_{\neq 0}$. The
automorphism $\nm(b,r) \mapsto \nm(ib,r)$ of $N$ induces 
a linear bijection  
in the $2|\ell|$-dimensional space of theta
functions with basis
$\bigl\{ \Th_{\ell,c}(h_{\ell,m})\;:\; 0\leq c < 2|\ell|\bigr\}$, where
$\Th_{\ell,c}(h_{\ell,m})$ is the linear transformation determined by
\be \Th_{\ell,c}(h_{\ell,m}) \bigl( \nm(ib,r) \bigr) \=
\frac{\bigl(-i\sign(\ell)\bigr)^m}{\sqrt{2|\ell|}}\sum_{c'=0}^{2|\ell|-1}
e^{\pi i c c'/\ell} \Th_{\ell,c'} (h_{\ell,m}\bigl( \nm(b,r)
\bigr)\,.\ee
\end{prop}

\section{Fourier term modules}

\subsection{Fourier term operators}\label{app-Fto}
The fact that for given pairs $(\ell,c)$ the operator
$\ph \mapsto \Th_{\ell,c}(\ph)$ is an intertwining operator makes it
fairly natural to consider the Fourier terms
\be \Four_{\ell,c} f (n a k) \= \frac \s 2 \sum_{m\geq 0} \Th_{\ell,c}(
h_{\ell,m})(n)\, \int_{\Ld_\s\backslash N} \overline
{\Th_{\ell,c}(h_{\ell,m})(n_1)} \, f(n_1 a k) \, dn_1\,,\ee
with $n\in N$, $a\in A$, $k\in K$. Using the orthogonal system
$\bigl\{\Kph hprq\bigr\}$ in $L^2(K)$ we get for $\Four_{\ell,c} f$ an
expansion as in \eqref{Flcd}, without the restriction~\eqref{mtpleq}.
One can work out the action of $U(\glie)$ on the individual terms
explicitly; see \cite[\S7]{BMSU21}. Then it turns out that the action
leaves the quantity in the left hand side of \eqref{mtpleq} constant,
with an odd value. This implies that the operators $\Four_{\ell,c,d}$
are indeed intertwining operators for the action of~$\glie$.

The splitting $\Four_{\ell,c} \= \sum_{d\equiv 1(2)} \Four_{\ell,c,d}$
can be understood from the action of the double cover of~$M$ on the
elements $\Four_{\ell,c}$ in $C^\infty(\Ld_\s\backslash G)_K$. See the
discussion of the metaplectic action in \cite[\S8.2]{BMSU21}.

The convergence of the total Fourier expansion
\[ f \= \sum_\bt \Four_\bt f + \sum_\n \Four_\n f\]
is discussed in \cite[Proposition 5.2]{BMSU21}.

\subsection{\texorpdfstring{$N$}{N}-trivial Fourier term
modules}\label{app-NtFt}
The principal series module $H^{\xi,\nu}_K$ contains the $K$-types
$\tau^h_p$ satisfying $ |h- 2j|\leq p$, with multiplicity one. The
subspace $H^{\xi,\nu}_{K;h,p,q}$ of $K$-type $\tau^h_p$ and weight $q$
is spanned by the function
\be \kph h p r q(\nu) : n\am(t)k \mapsto t^{2+\nu} \,\Kph h p r q(k)\,,
\ee
with $r= \frac{h-2j}3$. Here, by the Iwasawa decomposition, if $g\in G$,
we write $g=n\am(t)k$ with $n\in N$ $\am(t)\in A$ and $k\in K$. See
\cite[\S10.1]{BMSU21}.

If $\nu \not\equiv j_\xi\bmod 2$, then the principal series module
$H^{\xi,\nu}_K$ is an irreducible $(\glie,K)$-mod\-ule. All isomorphism
classes of irreducible $(\glie,K)$-modules are represented by a
submodule in some principal series representation
(see \cite{CM82}).
\medskip

If the Weyl orbit contains elements of the form $(j,0)$, then there are
logarithmic submodules of $\Ffu_0^{\ps[j,\nu]}$, as discussed in
\cite[Propositions 10.4 and 12.3]{BMSU21}.

\subsection{Families of Fourier terms }\label{app-famFt}
The families $\nu \mapsto \om^{0,0}_\Nfu(j_\xi,\nu)$ in \eqref{MWxinubt}
(for $\Nfu=\Nfu_\bt$, $\bt\neq 0$) and \eqref{om-mu-nab} (for
$\Nfu=\Nfu_\n$)
are holomorphic and even in $\CC$. Furthermore, the function
$(\nu,g)\mapsto \om^{0,0}(j,\nu;g)$ is in $C^\infty (\CC\times G)$. The
families $\nu \mapsto \mu^{0,0}_\Nfu(j_\xi,\nu)$ have similar
properties, except that in the non-abelian case $\Nfu=\Nfu_\n$ there
may be first order singularities at points of $\ZZ_{\leq -1}$, as one
can see in the expansion in \cite[(A.9)]{BMSU21}.

The values of these families at points $\nu$ in the closed right
half-plane $\re\nu\geq 0$ are elements of $\Wfu^{\xi,\nu}_\Nfu$ or
$\Mfu^{\xi,\nu}_\Nfu$, respectively. For the values and residues in the
left half-plane we are content to note that they are in
$\Ffu^{\ps[\xi,\nu]}_\Nfu$, where $\ps[\xi,\nu]$ is the character of
$ZU(\glie) $ parametrized by $(j_\xi,\nu)$.

If we multiply these families by a holomorphic function of $\nu$ we get
families with the same properties. The special choice of
$\om^{0,0}_\Nfu$ and $\mu^{0,0}_\Nfu$ is not intrinsic, but governed by
the traditional choice of basis elements of differential equations for
the modified Bessel functions and the Whittaker functions.

Differentiation by an element of $\glie$ preserves these properties. So
if $u\in U(\glie)$ then $u\, \om^{0,0}_\Nfu(j,\nu)$ and
$u\, \mu^{0,0}_\Nfu(j,\nu)$ are again holomorphic (or meromorphic).
There are choices of $u\in U(\glie)$ that send $\Ffu^\ps_{\Nfu;0,0,0}$
to $\Ffu^\ps_{\Nfu;h,p,q}$ for any $K$-type $\tau^h_p$ and weight $q$
in that $K$-type.

In the case $q=p$ this is done in \cite{BMSU21} by repeated application
of the shift operators; see Proposition 6.1, and the application to
Fourier term modules in Tables 9 and~11. Other weights in a given
$K$-type are reached by powers of $\Z_{1,2} \in \klie$. This leads to
the families $\om^{a,b}_\Nfu$ and $\mu^{a,b}_\Nfu$ in
\cite[(10.17)]{BMSU21} with $a,b\in \ZZ_{\geq 0}$. Under generic
parametrization this is used to define $\Wfu^{\xi,\nu}_\Nfu$ and
$\Mfu^{\xi,\nu}_\Nfu$. In the abelian case, this also works under
integral parametrization. See \cite[(10.18), (13.8)]{BMSU21}. In the
non-abelian case the families $\om^{a,b}_\Nfu(\xi,\nu)$ and
$\mu^{a,b}_\Nfu(\xi,\nu)$ may have zeros in the closed right
half-plane, which have to be divided out. See Proposition 14.2 and
Definition 14.8 in~\cite{BMSU21}.

In this way we arrive at the following result:
\begin{prop}\label{prop-MuOm}Let $\xi$ be a character of $M$ and let
$\Nfu\neq \Nfu_0$ a Fourier term order. For each $K$-type $\tau^h_p$
satisfying $|h-2j|\leq 3p$, and each weight $q\equiv p \bmod 2$,
$|q|\leq p$ there are an even holomorphic family
$\Om_{\Nfu;h,p,q}(\xi,\nu)$ in $\CC$ and a meromorphic family
$\Mu_{\Nfu;h,p,q}(\xi,\nu)$ in $\CC$, with at most first order
singularities in points of $\ZZ_{\leq -1}$, with the following
properties:
\begin{enumerate}
\item[a)] The values and residues of these families at $\nu$ are in
$\Ffu^{\ps[j_\xi,\nu]}_\Nfu$.
\item[b)] For $\re\nu \geq 0$ the value $\Om_{\Nfu;h,p,q}(\xi,\nu_0)$
spans the one-dimensional space $\Wfu^{\xi,\nu_0}_{\Nfu;h,p,q}$.
\item[c)] For $\re\nu \geq 0$ the value $\Mu_{\Nfu;h,p,q}(\xi,\nu_0)$
spans the one-dimensional space $\Mfu^{\xi,\nu_0}_{\Nfu;h,p,q}$.
\end{enumerate}
\end{prop}

Any (finite) non-zero linear combination
$\sum_j c_j(\nu) \,\Mu_{\Nfu;h_j,p_j,q_j}(\xi,\nu)$ with holomorphic
functions $c_j$ gives a function as required in
Proposition~\ref{prop-Mudef}.

In all cases, we have a decomposition in Iwasawa coordinates of the form
\badl{OmMuh} \Om_{\Nfu;h,p,q}&(\xi,\nu)\bigl( n \am(t) k \bigr)
\= \sum_r w_r(n)\, h^\om_r(\nu,t)\, \Kph h p r q(k)\,,\\
\Mu_{\Nfu;h,p,q}&(\xi,\nu)\bigl( n \am(t) k \bigr)
\= \sum_r w_r(n)\, h^\mu_r(\nu,t)\, \Kph h p r q(k)\,.
\eadl
The sum is over $r\equiv p \bmod 2$, $|r|\leq p$. In the non-abelian
case there is an additional restriction on $r$; here we just take the
components $h^\om_r(\nu,t)$ and $h^\mu_r(\nu,t)$ to be zero if the
additional condition is not satisfied. By $w_r$ we mean $\ch_\bt$
in~\eqref{chbt} if $\Nfu=\Nfu_\bt$, with $\bt\neq 0$. In the
non-abelian case $\Nfu=\Nfu_{\ell,c,d}$ and we have
$w_r = \Th_{\ell,c}( h_{\ell,m(h,r) })$ with $m(h,r)\in \ZZ_{\geq 0}$
given in~\eqref{parms-nab} below.

The component $h^\om_r(\nu,t)$ is a linear combination of positive
powers of $t$  times factors $K_{\nu+n}(2\pi|\bt|t)$
if $\Nfu=\Nfu_\bt$ or factors $W_{\k+n,\nu/2}(2\pi|\ell|t^2)$ if
$\Nfu=\Nfu_\n$. In both cases $n \in \ZZ_{\geq 0}$. The component
$h^\mu_r(\nu,t)$ has a similar structure, with Bessel functions
$I_{\nu+n}(2\pi|\bt|t)$ and Whittaker functions
$W_{\k+n,\nu/2}(2\pi|\ell|t^2)$. The description of these components is
far from unique. With the contiguous relations mentioned in
\S\ref{sect-rpfWr} we can rewrite the components in many ways. From
\eqref{IKest} and \eqref{MWest}, and from the properties of
$\Mfu^\ps_\Nfu$, we obtain estimates valid at least for
$\re\nu \geq 0$, of the form
\badl{OmMu-est} \Om_{\Nfu;h,p,q}(\xi,\nu) \bigl( n \am(t) k \bigr) &\=
\oh\bigl(e^{-\al t}) &&\text{as }t\uparrow\infty, \text{ for some
}\al>0\,,\\
&\= \oh\bigl( t^{2-\al}\bigr) &&\text{as }t\downarrow 0, \text{ for some
}\al>0\,,\\
\Mu_{\Nfu;h,p,q}(\xi,\nu) \bigl( n \am(t) k \bigr) &\= \oh \bigl(
t^{2+\re\nu}\bigr)&&\text{as }t\downarrow 0\,.
\eadl

\rmrk{Drawbacks} (1) The families $\om^{0,0}_\Nfu(j,\nu)$ and
$\mu^{0,0}_\Nfu(j,\nu)$ generating $\Wfu^{\xi,\nu}_\Nfu$ and
$\Mfu^{\xi,\nu}_\Nfu$ are nicely explicit. The action of the shift
operators can be carried out explicitly, at least with help of a
computer. In the $K$-type $\tau^h_p$ and weight~$q$ the elements have
the form
\begin{align*}
& \sum_{r\equiv p\bmod 2,\;|r|\leq p} \ch_\bt(n) \, F_r(t) \, \Kph h p r
q(k)\,,\\
&\sum_{r\equiv p\bmod 2,\;|r|\leq p} \Th_{\ell,c}(h_{\ell,m(r)})(n)
\, F_r(t) \, \Kph h p r q(k)\,,
\end{align*}
where the $F_r(t)$ are linear combinations of modified Bessel functions
or Whittaker functions multiplied by powers of~$t$. After a few
applications of shift operators these linear combinations become in
general very complicated.

(2) Parts b) and c) of the proposition seem to give a basis of
$\Ffu^{\xi,\nu}_{\Nfu;h,p,q}$. This fails in some cases in the
non-abelian case, for values of $\nu \in \ZZ_{\geq 0}$. (This is due to
the fact that the Whittaker functions $M_{\k,s}$ and $W_{\k,s}$ may be
proportional; see \cite[(A.13)]{BMSU21}.)

\subsubsection{An additional family} \label{sect-afU}
Proposition 14.21, ii) in \cite{BMSU21} tells us that the sole instances
in which $\Mfu_\Nfu^{\xi,\nu}$ and $\Wfu_\Nfu^{\xi,\nu}$ have a
non-zero intersection occur in those non-abelian cases in which
$\Wfu^{\xi,\nu}_\n$ contains a submodule in the holomorphic or in the
antiholomorphic discrete series. Then the intersection
$\Wfu_\n^{\xi,\nu}\cap \Mfu_\n^{\xi,\nu}$ is equal to that submodule.

Proposition 16.6 in \cite{BMSU21} gives a third family
$\nu \mapsto \Ups_{\n;h,p,q}(j,\nu)$ in the non-abelian case. It is
based on the function $V_{\k,s}$ in~\eqref{Vkpsdef}, and is holomorphic
and even in $\CC$ such that for $\re\nu_0\geq 0$
\be \Ffu^\ps_{\n;h,p,q} \= \bigoplus_{(j_\xi,\nu_0) \in \Wo(\ps)_\n^+}
\Bigl( \CC \, \Om_{\n;h,p,q}(j_\xi,\nu_0) \oplus \CC \,
\Ups_{\n;h,p,q}(j_\xi,\nu_0) \Bigr)\,. \ee

\subsubsection{Explicit cases}\label{sect-MuOmexpl}There are only a few
cases where we know explicit and rather simple expressions of the
functions in Proposition~\ref{prop-MuOm}. In the one-di\-men\-sional
$K$-type with $(h,p,q) = (2j,0,0)$ we have the explicit expressions in
\eqref{MWxinubt} and \eqref{om-mu-nab}.

There are some higher dimensional cases as well. We mention the cases
that we will need later on. They are based on Propositions 13.2, 14.4
and~14.15 in \cite{BMSU21}. (Ishikawa gives similar expressions in
\cite[Theorem 5.3.1]{Ish99}.) This concerns $K$-types $\tau^h_p$ that
have minimal dimension among the $K$-types occurring in a submodule of
$\Ffu^\ps_\Nfu$. Then $p\geq 1$, and $\ps=\ps[-h,p]$. We take
$(\xi,\nu)\in \Wo(\ps)^+$ such that $j_\xi = -h$ and $\nu=p$.

If, in the \emph{generic abelian cases}, $\tau^h_p$ is minimal in
$\Wfu_\bt^\ps$, then
\badl{hdWa} \Om_{\bt;h,p,q}&(-h,p)\bigl(n \am(t) k\bigr)\\
&\ddis \sum_{r\equiv p(2),\,|r|\leq p} \ch_\bt(n)\, \Bigl( \frac{i
\bt}{|\bt|}\Bigr)^{(r+p)/2} \, t^{2+p} \, K_{|h-r|/2}(2\pi|\bt|t) \,
\Kph h p r q(k)
\,.\eadl
By $\dis$ we denote equality up to a non-zero factor.

If $\tau^h_p$ is minimal in $\Mfu_\bt^\ps$, $\bt\neq 0$, then
\badl{hdMa} \Mu_{\bt;h,p,q}&(-h,p)\bigl(n \am(t) k\bigr)\\
&\ddis \sum_{r\equiv p(2),\,|r|\leq p} \ch_\bt(n)\, \Bigl( \frac{-i
\bt}{|\bt|}\Bigr)^{(r+p)/2} \, t^{2+p} \, I_{|h-r|/2}(2\pi|\bt|t) \,
\Kph h p r q(k)
\,.\eadl

In the \emph{non-abelian cases} $\Nfu=\Nfu_{\ell,c,d}$ we use the
quantities
\badl{parms-nab} m(h,r) &\=\frac12\sign(\ell)\,\bigl( r-r_0(h)\bigr)\,,&
\hbox{\ } r_0(h)&\= \frac13(h-d) + \sign(\ell)\,,\\
\k(r) &\= - m(h,r) -\sign(\ell)s(h,r)-\frac12& s(r)&\= s(h,r) =
\frac14(h-r)\,.
\eadl
The summation runs over $r\equiv p\bmod 2$ such that $|r|\leq p$ and
$m(h,r)\geq 0$.

If $\tau^h_p$ is minimal in $\Wfu_\n^\ps$, then
\badl{hdWna} \Om_{\n;h,p,q}&(-h,p)\bigl(n \am(t) k\bigr)
\ddis \sum_{r} \Th_{\ell,c}(h_{\ell,m(h,r)})(n) \\
&\hbox{} \cdot i^{m(h,r) } \,\sqrt{ m(h,r)!}\, t^{1+p} \,
W_{\k(r),s(r)}(2\pi|\ell|t^2)
  \, \Kph h p r q(k)
\,.\eadl

If $\tau^h_p$ is minimal in $\Mfu_\n^\ps$, then
\badl{hdMna} \Mu_{\n;h,p,q}&(-h,p)\bigl(n \am(t) k\bigr)\\
& \ddis \sum_{r} \Th_{\ell,c}(h_{\ell,m(h,r)})(n) \, c^M(r)\, t^{1+p} \,
M_{\k(r),|s(r)|}(2\pi|\ell|t^2)
  \, \Kph h p r q(k)
\,,\\
c^M(r) &\= - e^{\pi i(m(h,r)-\k(r))}\,\frac{\Gf \bigl(
\frac12+|s(r)|-\k(r)\bigr)}{\sqrt{m(h,r)!}\, \bigl( 2 |s(r)|
\bigr)!}\,.
\eadl

We stress that these formulas give an explicit expression only for the
value at $\nu=p$ of the holomorphic families
$\nu \mapsto \Om_{\Nfu;h,p,q}(-h,\nu)$ and
$\nu \mapsto \Mu_{\Nfu;h,p,q}(-h,\nu)$.

\section{Special functions and Wronskians}\label{sect-rpfWr}
The $N$-non-trivial Fourier term functions have expressions in
terms of modified Bessel functions and Whittaker functions. 
Appendix~A of~\cite{BMSU21} gives many results and references 
for these special functions. Here we mention facts that we
use here.

\subsection{Modified Bessel functions}The modified Bessel differential
equation
\be \label{Bde}
s^2\, j''(x) + x \, j'(x) - (x^2+\nu^2) j(x)\=0\ee
has the standard solutions $I_\nu$ and $K_\nu$; \cite[(A.2),
(A.4)]{BMSU21}. Their boundary behavior is
\badl{IKest} I_\nu(x) &\;\sim\;
(2\pi x)^{-1/2} \, e^x && \text{as }x\uparrow\infty\,,\\
&\;\sim\; \Gf(\nu+1)^{-1} (x/2)^\nu && \text{as } x\downarrow 0\,;
\\
K_\nu(x) & \;\sim\; (\pi/2x)^{1/2}\, e^{-x}&& \text{as
}x\uparrow\infty\,,\\
&\= \oh\bigl(x^{-|\re\nu|-\e}\bigr)&& \text{as } x\downarrow 0\,.
\eadl
For $n\in \ZZ$ we need the following more precise information, as
$x\downarrow 0$.
\badl{KI-expint} I_n(x) &\= I_{-n}(x) \= (x/2)^{|n|}/|n|!
+\oh(x^{|n|+4})\,,\\
K_n(x) &\= K_{-n}(x)\\
&\= \frac{(|n|-1)!}2 \, (x/2)^{-|n| }
+ \oh (x^{-|n|+2 }\,\log x) &&\text{if }n\neq 0\,,\\
&\= - \log(x/2)
+ \oh\bigl( x^2\log x\bigr)&&\text{if }n=0\,.
\eadl
The contiguous relations in \cite[(A.7)]{BMSU21} allow to express
derivatives of $I_\nu$ and $K_\nu$ as linear combinations of
$I_{\nu\pm 1}$ or of $K_{\nu\pm 1}$, respectively.

\subsection{Whittaker functions} The Whittaker differential equation
\be\label{Wde}
f''(x) \= \left( \frac14- \frac\k x +\frac{s^2-1/4}{x^2}\right)
f(x)\=0\ee
has standard solutions $M_{\k,s}$ and $W_{\k,s}$. Since these two
solutions are proportional for some combinations, we also use an ad hoc
solution $V_{\k,s}$, in \cite[(A.12)]{BMSU21}. It is the holomorphic
continuation in $s$ of
\be\label{Vkpsdef}
V_{\k,s}(x) \= \frac{\pi i}{\sin \pi s} \sum_\pm \frac{\pm e^{\pi \pi i
s}} { \Gf \bigl(\frac12\mp 2+\k\bigr)\, \Gf(1\pm 2s) }\, M_{\k,\pm
s}(x)\,.\ee
This implies the relation
\badl{MinWV} M_{\kappa,s}(\tau) &\= e^{\pi i \kappa}\,\Gamma(1+2s)\,
\biggl( \frac{-i\,e^{-\pi i s}}{\Gamma(1/2+s+\kappa)}
W_{\kappa,s}(\tau)
\\
&\quad\hbox{} - \frac 1{\Gamma(1/2+s-\kappa)} \, V_{\kappa,s}(\tau)
\biggr)\,.\eadl

We note the boundary behavior:
\badl{MWest} W_{\k,s}(x) &\;\sim\; x^\k \, e^{-x/2} &&\text{as }
x\uparrow\infty\,,\\
&\; \= \oh(x^{1/2-s})&&\text{as }x\downarrow 0\text{ and }s\not\in
\ZZ\,;\\
M_{\k,s}(x)&\= \frac{\Gf(1+2s)}{\Gf(\frac12-\k+s)} \, x^{-\kp} \,
e^{x/2} \,\Bigl(1+\oh(1/x)\Bigr) &&\text{as } x\uparrow\infty\text{,
for } \re s\geq 0\,,\\
&\;\sim\; x^{s+1/2}&&\text{ as } x\downarrow 0\text{ for }s\not\in
\ZZ_{\leq -1}\,;\\
V_{\k,s}(x) &\;\sim\; -e^{-\pi i \k}x^{-\k} e^{x/2}
&&\text{as }x\uparrow \infty\,.
\eadl
The functions $W_{\k,s}$ and $V_{\k,s}$ are even in the parameter~$s$.
If $\frac12+s-\k\not\in \ZZ_{\leq 0}$, we have as $x\downarrow 0$
\be \label{West0}
W_{\k,s}(x) \;\sim\; \begin{cases}
\frac{\Gf(2|s|)}{\Gf\bigl( \frac12+|s|-\k\bigr)} x^{1/2-|s|}&\text{ if }
s>0\,,\\
- \frac1{\Gf\bigl(\frac12-\k \bigr)} x^{1/2}\log x&\text{if }s=0\,.
\end{cases}
\ee

There are the specializations
\bad M_{\k,\k-1/2}(x) &\= W_{\k,\pm(\k-1/2)}(x) \= x^\k\, e^{-x/2}\,,\\
M_{\k,-\k-1/2}(x) &\= - e^{\pi i \k} V_{\k,\pm(\k-1/2)}(x) \= \tau^{-\k}
\, e^{\k/2}\,.
\ead

In \cite[(A.18)--(A.21)]{BMSU21} we quote many contiguous relations,
allowing us to express derivatives of Whittaker functions in terms of
Whittaker functions with neighboring values of the parameters.

\subsection{Wronskians}\label{sect-Wro} For two expressions $X$ and $Y$
we denote by $\Wr(X,Y)_t$ the quantity
$X \frac{dY}{dt} - \frac{dX}{dt} Y$. The skew-symmetry implies that
$\Wr(X,X)_t=0$.

For solutions $X$ and $Y$ of the differential equation \eqref{Bde} one
can easily check that $\Wr(X,Y)_x$ is a multiple of $\frac1x$, and for
solutions of \eqref{Wde} the Wronskian is a constant. The asymptotic
behavior near zero gives
\be \Wr\bigl(I_\nu(x), I_{-\nu}(x) \bigr)_x\= - \frac{2\sin \pi \nu}{\pi
x}\,,\quad \Wr\bigl( M_{\k,s}(x), M_{\k,-s}(x) \bigr)_x \= -2 s\,. \ee
Using the expression \cite[(A.4)]{BMSU21} for $K_\nu$ in terms of
$I_\nu$ and $I_{-\nu}$ we get \[ \Wr\bigl( I_\nu(x),K_\nu(x)\bigr)_x \=
-x^{-1}\,,\]
and similarly with \cite[(A.11)]{BMSU21}
\[ \Wr\bigl( M_{\k,s}(x),W_{\k,s}(x)\bigr)_x\=
\frac{\Gf(1+2s)}{\Gf\bigl( \frac12-\k+s\bigr)}\,.\]

This leads to the relations that we use in the proof of
Theorem~\ref{thm:orders}:
\badl{Wr0} \Wr \bigl( t^2 \,I_\nu(2\pi |\bt|t),\, t^2
\,K_\nu(2\pi|\bt|t) \bigr)_t &\= -t^3\,,\\
\Wr\bigl( t\, M_{\k,\nu/2}(2\pi|\ell|t^2),\, t\,
W_{\k,\nu/2}(2\pi|\ell|t^2\bigr)_t
&\= \frac{-4\pi |\ell|\, \Gf(\nu+1)}{\Gf\bigl(
\frac{\nu+1}2-\kp\bigr)}\, t^3\,.
\eadl

\section{Poincar\'e series for
\texorpdfstring{$\SL_2(\RR)$}{SL(2,R)}}\label{PSSl2R}

Here we give some more information related to the discussion in
Subsection~\ref{comparison}. We work with functions on the group
$\SL_2(\RR)$, like we do for $\SU(2,1)$. We use Iwasawa coordinates
$(x,y,\th) \leftrightarrow \nm(x)\am(y)\km(\th)$, with
$\nm(x) = \begin{pmatrix} 1&x\\0&1\end{pmatrix}$, $x\in \RR$;
$\am(y)=\begin{pmatrix}y^{1/2}&0\\0&y^{-1/2}\end{pmatrix}$, $y>0$; and
$\km(\th)= \begin{pmatrix}\cos\th&\sin \th\\-\sin\th&\cos\th\end{pmatrix}$,
$\th\in \RR\bmod 2\pi\ZZ$. A function $g$ on $\SL_2(\RR)$ has weight
$k\in \ZZ$ if $f\bigl(g\km(\th)\bigr)=f(g) e^{ik\th}$, so the weight
coincides with the $K$-type. The Casimir operator $C$ is given by
\be C \= - y^2\partial_x^2 - y^2\partial_y^2+ y \partial_x\partial_
\th\,.\ee

For automorphic forms of odd weight we need a character on the discrete
subgroup of the same parity as the weight.

Automorphic forms on $\Gm\backslash \SL_2(\RR)$ for the character $\ch$
are defined as in Definition~\ref{def-aut}, except that only in a) we
require $f(\gm g) = \ch(\gm) \, f(g)$ for all $\gm\in \Gm$. The ring
$ZU(\mathfrak{sl}_2)$ is the polynomial ring in the Casimir operator.
We write the eigenvalue of the Casimir operator as
$\frac12-\nu^2$.% For

For simplicity we choose the discrete subgroup $\Gm = \SL_2(\ZZ)$ which
has % $12$ 
twelve characters: $\ch _n$, $n\in \frac1{12}\ZZ\bmod \ZZ$, determined
by
%The characters of $\Gm=\SL_2(\ZZ)$ are $\ch_n$,
%$n\in \frac1{12}\ZZ\bmod \ZZ$, determined by
\be \ch_n\bigl(\nm(1)\bigr) \= e^{2\pi i n }\,,\qquad\ch_n\bigl(
\km(-\pi/2)\bigr)\= e^{-6\pi i n }\,.\ee
The character $\ch_n$ is suitable for weight $k$ if
$12n \equiv k \bmod 2$.

\emph{Basis functions} for the spaces of \emph{Fourier terms} of order
$n\neq 0$ are
\bad \Om_k\bigl(n,\nu;\nm(x)\am(y)\km(\th)\bigr) &\= e^{2\pi i n x} \
W_{\e k/2,\nu}(4\pi|n| y) \, e^{ik\th}\,,\\
\Mu_k\bigl(n,\nu;\nm(x)\am(y)\km(\th)\bigr) &\= e^{2\pi i n x} \, M_{\e
k/2,\nu}(4\pi|n|y) \, e^{ik\th}\,,
\ead
with $\e=\sign(n)$. These functions have weight $k$ and eigenvalue
$\frac14-\nu^2$ under the Casimir operator.

The \emph{Poincar\'e series}
\be P_k(n,\nu; g) \= \sum_{\gm\in \Gm_\infty \backslash \Gm}
\ch(\gm)^{-1} \Mu_k(n,\nu;\gm g)\ee
converges absolutely for $\re \nu>\frac12$ and gives a non-trivial
family $\nu\mapsto P_k(n,\nu)$ of real-analytic automorphic forms on
$\Gm\backslash \SL_2(\RR)$ for the character $\ch$, of weight $k$, and
with moderate exponential growth. (Only the Fourier term of order $n$
has exponential growth for general values of $\nu$.)

The family has a meromorphic extension to $\nu\in \CC$ and all its main
parts are automorphic forms with moderate exponential growth. This can
be shown by spectral theory, as in \cite{MW89}.

Several results in Section~\ref{sect-rPs} are similar to results for
$\Gm\backslash \SL_2(\RR)$. We quote some statements from
\cite[Proposition 11.3.9]{B94}.

For even weights $k$, the Poincar\'e series $P_k(n,\nu)$ has a first
order singularity at
$\nu=\nu_0 \in i \RR\setminus \{0\} \cup\bigl(0,\tfrac12\bigr)$ if and
only if there is a square integrable automorphic form $\ph$ of weight
$k$ with character $\ch_n$ and eigenvalue $\frac14-\nu_0^2$ that has a
non-zero Fourier term of order $n$. Such an automorphic form $\ph$
generates a submodule of $L^{2,\discr}_k(\Gm\backslash G, \ch)_K$ in
the unitary principal series or in the complementary series. For
$\nu_0=0$ the same equivalence holds, but now with a second order
singularity.

For odd weights, a similar statement holds, with
$\nu=\nu_0 \in i \RR\setminus \{0\} $ (no complementary series). These
results are similar to those in Proposition~\ref{prop-upcstr} above.
For $\SL_2(\RR)$ there are no irreducible modules analogous to the thin
representations.

The discrete series type representations of $\SL_2(\RR)$ are analogous
to the holomorphic and antiholomorphic discrete series type modules for
$\SU(2,1)$. Part~i) of Proposition~\ref{prop-hads} is analogous to the
fact that spaces of holomorphic cusp forms of weights at least three on
the upper half plane are spanned by values of Poincar\'e series. For
weight $k=2$ one has to use the analytic continuation. The holomorphic
cusp forms of weight $k$ correspond to the lowest weight vectors in
discrete series modules for $\SL_2(\RR)$ with Casimir eigenvalue
$\frac k2(1-\frac k2)$.

The case of weight $1$ corresponds to $\nu_0=0$. For
$\SL_2(\ZZ)\backslash \SL_2(\RR)$ it is known that the space of
holomorphic square integrable automorphic forms has dimension $1$ for
the character $\ch_{1/12}$, and is zero for other characters. For
character $\ch_{1/12}$ the space is spanned by the square of the
Dedekind eta function. The corresponding function on $\SL_2(\RR)$ is
\be \ph_+\bigl( \nm(x)\am(y)\km(\th)\bigr)\= y^{1/2}\, \eta(x+iy)^2 \,
e^{i\th}\,.\ee
We have also
\be \ph_-\bigl( \nm(x)\am(y)\km(\th)\bigr)\= y^{1/2}\,
\overline{\eta(x+iy)^2} \, e^{-i\th}\ee
with character $\ch_{-1/12}$ and weight $-1$. Here, the element $\ph_+$
is a lowest weight vector in a submodule of a principal series
representation, and $\ph_-$ is a highest weight vector in the
complementary module.

From an initial part of the Fourier expansion
\[ \eta(z)^2 \= q^{1/12}\,\Bigl( 1-2q -q^2+2q^3+q^4+2q^5- q^6-2q^8
-2q^9 + \cdots
\Bigr)\,, \quad q=e^{2\pi i z}\,,\]
we can conclude that $P_{\pm 1}(n,\nu)$ has a first order singularity at
$\nu=0$ for $n= \pm\bigl( \frac1{12} +m\bigr)$ with $0\leq m \leq 6$
and $m=8,9$, and that it is holomorphic at $\nu=0$ for $m=7$.

For the other five characters suitable for odd weights, the Poincar\'e
families are holomorphic at $\nu=0$.

\raggedright

\end{document}